\documentclass[a4paper]{amsart}

\usepackage{accents}
\usepackage{amsfonts, amsmath, amssymb, amsthm}
\usepackage{cancel}
\usepackage{color, framed}
\usepackage{graphicx}
\usepackage{hyperref}
\usepackage{tikz} 
\usepackage{verbatim}
\usepackage[all]{xy}

\allowdisplaybreaks 

\definecolor{shadecolor}{gray}{0.8}

\newcommand{\flplus} 
{
\hspace{0.1cm}
\begin{tikzpicture}[baseline=-0.582ex]
    \draw [line width=0.24pt](-0.1129, 0) -- (0.1129, 0) -- (0, 0) -- (0, 0.1129) -- (0, -0.1129) arc (270:90:0.1129) -- (0, 0);
\end{tikzpicture}
\hspace{0.1cm}
}

\newcommand{\frplus} 
{
\hspace{0.1cm}
\begin{tikzpicture}[baseline=-0.582ex]
    \draw [line width=0.24pt](-0.1129, 0) -- (0.1129, 0) -- (0, 0) -- (0, 0.1129) -- (0, -0.1129) arc (-90:90:0.1129) -- (0, 0);
\end{tikzpicture}
\hspace{0.1cm}
}

\usepackage{t1enc}

\newcommand{\btimes}{\boldsymbol{\times}}

\newcommand{\pdot}{{\, \cdot \,}} 

\def\Bbb{\mathbb}
\def\Cal{\mathcal}

\newcommand{\im}{\operatorname{im}}

\DeclareMathOperator{\dom} {dom}
\DeclareMathOperator{\G}   {G}
\DeclareMathOperator{\GL}  {GL}

\DeclareMathOperator{\rank}{rank}
\renewcommand{\Re}{\operatorname{Re}}
\DeclareMathOperator{\Sc}  {Sc}
\DeclareMathOperator{\SL}  {SL}
\DeclareMathOperator{\SO}  {SO}

\DeclareMathOperator{\tr}  {tr}

\newcommand{\ab}[1]{\langle {#1} \rangle}
\newcommand{\wt}[1]{\widetilde{#1}}

\newcommand{\hook}{\, \lrcorner \,}

\newcommand{\bbC}{\mathbb{C}}

\newcommand{\bbI}{\mathbb{I}}
\newcommand{\bbJ}{\mathbb{J}}
\newcommand{\bJ}{\mathbb{J}}

\newcommand{\bbL}{\mathbb{L}}
\newcommand{\bbO}{\mathbb{O}}
\newcommand{\bbR}{\mathbb{R}}

\newcommand{\bbZ}{\mathbb{Z}}

\newcommand{\set}[1]{ \{ {#1} \} }
\newcommand{\ul} [1]{\underline{#1}}
\newcommand{\wh} [1]{\widehat  {#1}}

\newcommand{\mbc}{\mathbf{c}}
\newcommand{\mbg}{\mathbf{g}}
\newcommand{\mbp}{\mathbf{p}}

\newcommand{\mcE}{\mathcal{E}}
\newcommand{\mcG}{\mathcal{G}}
\newcommand{\mcL}{\mathcal{L}}
\newcommand{\mcT}{\mathcal{T}}

\newcommand{\mfg}  {\mathfrak{g}}

\newcommand{\mfinf}{\mathfrak{inf}}
\newcommand{\mfp}  {\mathfrak{p}}
\newcommand{\mfq}  {\mathfrak{q}}
\newcommand{\mfs}  {\mathfrak{s}}
\newcommand{\mfsl} {\mathfrak{sl}}
\newcommand{\mfso} {\mathfrak{so}}

\newcommand{\ulw}{\smash{\underline{w}}}

\newcommand{\bbV}{\mathbb{V}}
\newcommand{\bbW}{\mathbb{W}}

\newcommand{\ol}[1]{\overline{#1}}

\newcommand{\ce}{{\Cal E}}

\newcommand{\cT}{{\mathcal T}}

\newcommand{\rpl}                         
{\mbox{$
\begin{picture}(12.7,8)(-.5,-1)
\put(0,0.2){$+$}
\put(4.2,2.8){\oval(8,8)[r]}
\end{picture}$}}

\def\al{\alpha}

\def\bg{\mbox{\boldmath $g$}}

\newcommand{\om}{\omega}

\newcommand{\si}{\sigma}
\newcommand{\id}{\operatorname{id}}

\renewcommand{\Im}{\operatorname{Im}}

\newcommand{\End}{\operatorname{End}}

\newtheorem{theorem}{Theorem}[section]
\newtheorem{lemma}[theorem]{Lemma}
\newtheorem{proposition}[theorem]{Proposition}
\newtheorem{corollary}[theorem]{Corollary}

\theoremstyle{definition}
\newtheorem{definition}[theorem]{Definition}
\newtheorem{example}[theorem]{Example}
\theoremstyle{remark}
\newtheorem{remark}[theorem]{\rm\bf Remark}
\newtheorem*{definition*}{\rm\bf Definition}

\newcommand{\nn}[1]{(\ref{#1})}

\newcommand{\Ric}{\operatorname{Ric}}
\newcommand{\nd}{\nabla}
\newcommand{\vol}{\scalebox{1.1}{$\boldsymbol{\epsilon}$}}

\def\sideremark#1{\ifvmode\leavevmode\fi\vadjust{\vbox to0pt{\vss
 \hbox to 0pt{\hskip\hsize\hskip1em
 \vbox{\hsize3cm\tiny\raggedright\pretolerance10000
  \noindent #1\hfill}\hss}\vbox to8pt{\vfil}\vss}}}%

\author{A.\ R. Gover, R. Panai, \& T.\ Willse}

\title{Nearly K\"ahler geometry and $(2,3,5)$-distributions\\ via projective holonomy}
\begin{document}

\address{
A.R.G.: Department of Mathematics\\
  The University of Auckland\\
  Private Bag 92019\\
  Auckland 1142\\
  New Zealand;
  Mathematical Sciences Institute\\
  Building 27\\
  Australian National University\\
  ACT 0200\\
  Australia\\
R.P.: Department of Mathematics\\
  The University of Auckland\\
  Private Bag 92019\\
  Auckland 1142\\
  New Zealand;
T.W.\\
  Fakult\"at f\"ur Mathematik\\
  Universit\"at Wien\\
  Oskar-Morgenstern-Platz 1\\
  1090 Wien\\
  Austria
}

\email{r.gover@auckland.ac.nz}
\email{robertopanai@sardus.it}
\email{travis.willse@univie.ac.at}


\begin{abstract}
We show that any dimension $6$ nearly K\"{a}hler (or nearly
para-K\"{a}hler) geometry arises as a projective manifold equipped
with a $\smash{\G_2^{(*)}}$ holonomy reduction.  In the converse direction we
show that if a projective manifold is equipped with a parallel
$7$-dimensional cross product on its standard tractor bundle then the
manifold is: a Riemannian nearly K\"{a}hler manifold, if the cross
product is definite; otherwise, if the cross product has the other
algebraic type, the manifold is in general stratified with
nearly K\"ahler and nearly para-K\"ahler parts separated by a
hypersurface which canonically carries a Cartan
$(2,3,5)$-distribution. This hypersurface is a projective infinity for
the pseudo-Riemannian geometry elsewhere on the manifold, and we
establish how the Cartan distribution can be understood explicitly,
and also in terms of conformal geometry, as a limit of the ambient
nearly (para\nobreakdash-)K\"ahler structures.  Any real-analytic
$(2,3,5)$-distribution is seen to arise as such a limit, because we
can solve the geometric Dirichlet problem of building a collar
structure equipped with the required holonomy-reduced projective
structure.

A model geometry for these structures is provided by the
projectivization of the imaginary (split) octonions. Our approach is
to use Cartan/tractor theory to provide a curved version of this
geometry; this encodes a curved version of the algebra of imaginary
(split) octonions as a flat structure over its projectivization.  The
perspective is used to establish detailed results concerning the
projective compactification of nearly (para\nobreakdash-)K\"ahler
manifolds, including how the almost (para\nobreakdash-)complex
structure and metric smoothly degenerate along the singular
hypersurface to give the distribution there.
\end{abstract}

\thanks{A.R.G. and R.P. gratefully acknowledge support from the Royal
  Society of New Zealand via Marsden Grants 10-UOA-113 and
  13-UOA-018. R.P gratefully acknowledges support from the Regione Sardegna via Grant AF-DR-A2011A-36115. T.W. gratefully acknowledges support from the Australian
  Research Council.}

\maketitle
\pagestyle{myheadings}
\markboth{Gover, Panai, Willse}{Nearly K\"ahler geometry and $(2,3,5)$-distributions via projective holonomy}


\tableofcontents

\section{Introduction}\label{intro}

Nearly K\"{a}hler geometries are one of the most important classes in
the celebrated Gray-Hervella classification of almost Hermitian
geometries \cite{AgFr,GrayH,PA1,S}.  A Cartan $(2,3,5)$-distribution
is the geometry arising from a maximally nondegenerate distribution of
2-planes in the tangent bundle of a 5-manifold \cite{CartanPff}. These
have attracted substantial interest for numerous reasons: they provide
a first case of a geometry of distributions with interesting local
invariants; they arise naturally from a class of second-order ODEs;
they are linked to concrete realizations of the exceptional group
$\G_2$ (and its variants); they have fascinating connections to
rolling ball problems; and they have important links to conformal
geometry
\cite{BorMontgomery,CartanPff,HammerlSagerschnig,NurowskiDifferential}.
Our aim in this work is first to show clearly the link between
projective differential geometry and nearly K\"ahler geometry and then
second to use this to expose and study a beautiful convergence of
nearly K\"{a}hler and $(2,3,5)$ geometry. In particular we show that
any $(2,3,5)$-geometry arises as the induced geometry on the boundary
at infinity of a nearly K\"ahler manifold; this includes a conceptual
and detailed explanation of how the almost complex structure of the
nearly K\"ahler geometry degenerates at the boundary to yield there
the distribution generating the $(2,3,5)$ structure.  This uses the
algebraic structure of the imaginary (split) octonions, and indeed we
use new results and ideas from the general theory of Cartan holonomy
reduction (from \cite{CGHjlms,CGH}) to describe a point-dependent imaginary octonion structure
on projective 6-manifolds. We then exploit tractor calculus to
understand how the differential and algebraic structures interact,
enabling, for example, a holographic program for the
$(2,3,5)$-distribution.

Let $(M^n,J)$ be an almost complex manifold of dimension $n\geq 4$,
and $g$ a (pseudo\nobreakdash-)Riemannian metric on $M$.  The triple $(M, g, J)$ is
said to be {\em almost (pseudo\nobreakdash-)Hermitian}
if $J$ is orthogonal with respect to $g$, that is, if $ g(J U, J V)=g(U, V)
$ for all tangent vector fields $U, V$. When this holds
$\om(\pdot, \pdot):=g(\pdot, J \pdot)$ is a 2-form called the {\em
  K\"{a}hler form}. If the almost complex structure in addition
satisfies
\begin{equation}\label{NKdef}
(\nabla_U J) U = 0,\quad\quad\forall \, U\in \Gamma(TM),
\end{equation}
where $\nabla$ is the Levi-Civita connection, or equivalently that $\nabla \omega$ is totally skew, then the almost
Hermitian manifold is called a {\em nearly K\"ahler} geometry. In
dimension $n=4$ the equation \nn{NKdef} implies the structure is
K\"ahler, but in higher dimensions it is a strictly weaker
condition. Throughout the article we will assume that any nearly
K\"ahler geometry is \textit{strictly} nearly K\"{a}hler, meaning that $\nabla \omega$ vanishes nowhere.  These structures are especially important in
dimension 6 \cite{Gray76,PA1,SSH}, the dimension
which is key in this article.

A projective structure on a manifold $M$ is an equivalence
class ${\bf p}$ of torsion-free affine connections, where two
connections $\nabla$ and $\nabla'$ are said to be equivalent if they
share the same geodesics as unparameterized curves. On a nearly
K\"{a}hler manifold with metric $g$, the Levi-Civita connection
$[\nabla^g]$ determines a projective structure ${\bf p}=[\nabla^g]$;
however this is the trivial aspect of a deeper link. A projective
structure determines, and is equivalent to, a structure called a
projective Cartan connection \cite{CSS,Cartan,KN}; this is very easily
seen using an equivalent associated bundle structure called the
projective tractor connection \cite{BEG,CapGoMac}.  A critical point
is that this higher order structure has a very special symmetry
reduction if a nearly K\"{a}hler geometry underlies the projective
structure as follows. This result is an immediate consequence of Theorem \ref{NPKtoPhi}. 
\begin{theorem}\label{NKtoG2}
A nearly K\"{a}hler 6-manifold $(M, g, J)$ determines a holonomy
reduction of the projective Cartan bundle (of $(M,[\nabla^g])$) to the
holonomy group $\G_2$ if $g$ is Riemannian, or to $\G_2^*$ if $g$ has
signature $(2,4)$.
\end{theorem}
\noindent Here $\G_2$ and $\G_2^*$ denote, respectively, the compact
and noncompact real forms of the exceptional simple complex Lie group
$\smash{\G_2^{\mathbb{C}}}$; we write $\smash{\G_2^{(*)}}$ to indicate
either one of these possibilities.\footnote{There are actually two
  (connected) groups with Lie algebra the split real form $\mfg_2^*$ of the exceptional simple complex Lie algebra $\mfg_2^{\bbC}$: the automorphism
  group of the split octonions (see Subsection
  \ref{subsection:octonions-cross-product-g2}), which has fundamental
  group $\bbZ_2$, and its universal cover \cite{BorMontgomery}. In
  this article, $\G_2^*$ always refers to the former.}  A link between
nearly K\"{a}hler geometry and these exceptional groups has been
previously observed in the literature using pseudo-Riemannian
constructions, namely metric cones
\cite{B,ChSa,CLSSH,Gru,Ka1,Semm}. However these studies do not make
the connection to projective differential geometry. Yet, as we shall
show, understanding the role of projective geometry is crucial for
extracting the full implications of the exceptional group structures.

The nearly K\"{a}hler defining equation \nn{NKdef} determines an
equation on the K\"{a}hler form $\om$ called the Killing-Yano
equation. This is projectively invariant; it is an equation from an
important class of equations known as first BGG equations
(cf.\ \cite{CGHjlms}). We show in Theorem \ref{NPKtoPhi} that on a
$6$-dimensional nearly K\"{a}hler manifold, $\om$ is a {\em normal
  solution} of this equation in the sense of \cite{CGHjlms}. In this
case, this means that (by prolongation) $\om$ determines, and is
equivalent to, a certain tractor 3-form field $\Phi$ that is parallel
for the normal projective tractor connection, and it is this that
gives the holonomy reduction. This perspective provides a natural
geometric framework to extend the structure and connect to other
geometries. For example an important question is how one may
compactify complete nearly K\"{a}hler geometries and, if so, what
geometry is induced on the boundary. A result in this direction is as
follows. Here we use that the definition of a projective manifold
applies to a smooth manifold with boundary. 

\begin{theorem} \label{cify} Let $(\overline{M},{\bf p})$ be a projective 
6-manifold with boundary $\partial M\neq \emptyset $ and interior $M$.
Suppose further that $M$ is equipped with a geodesically complete
nearly K\"{a}hler structure $(g, J)$ such that the projective class
$[\nabla^g]$ of the Levi-Civita connection $\nabla^g$ coincides with
${\bf p}|_M$. Then: $g$ has signature $(2,4)$, the metric $g$ is
projectively compact of order 2, and the boundary has a canonical
conformal structure equivalent to an oriented Cartan
$(2,3,5)$-distribution.
\end{theorem}

 \noindent The notion of {\em projectively compact} used here is a
 projective analogue of conformal compactification, as formulated in
 \cite{CapGo-projC}. The statement concerning signature is correct
 without loss of generality; the signature could of course be $(4,2)$
 instead of $(2,4)$.  This theorem is proved in Section \ref{cpctsec}.

Theorem \ref{NKtoG2} suggests an obvious converse problem.  A Cartan
holonomy reduction determines a canonical stratification of the
underlying manifold into initial submanifolds, with the different
strata (called {\em curved orbits}) equipped with specific geometric
structures determined by the reduction; the general theory is
developed in \cite{CGH} following the treatment of projective geometry
\cite{CGHjlms} and a ``pilot case'' in conformal geometry
\cite{Gal}. Providing the details for this geometric stratification,
specific to our current setting, resolves this converse problem and
more, as in the following Theorem which paraphrases key results from Corollary \ref{firststrat}, Theorem \ref{local}, and Theorem \ref{confRed}. It is this these results that lead to
Theorem \ref{cify}. A parallel tractor 3-form $\Phi$ is said to be
\textit{generic} if it determines, via a certain algebraic
construction (see \nn{HHdef}), a metric $H$ on the tractor
bundle. According to whether $H$ is positive definite or indefinite we
say $\Phi$ is, respectively, {\em definite-generic} or {\em
  split-generic}.

\begin{theorem} \label{orbit}
Suppose that $(M,{\bf p})$ is a 6-dimensional projective manifold
equipped with a parallel generic tractor 3-form
$\Phi$. Then:
\begin{itemize}
\item If $\Phi$ is definite-generic then it
determines a $\G_2$ holonomy reduction of the Cartan bundle, and
$(M,{\bf p},\Phi)$ is equivalent to a signature-$(6,0)$ nearly
K\"{a}hler structure on $M$ that is positive Einstein.

\item If
$\Phi$ is split-generic then it determines a $\G_2^*$ holonomy
reduction of the Cartan bundle and a decomposition $M=M_{+}\cup
M_0 \cup M_-$ of $M$ into a union
of 3 (not necessarily connected) disjoint curved orbits, where $M_{\pm}$ are open and $M_0$ is closed. If $M$ is connected and both
$M_{+}$ and $M_{-}$ are non-empty then $M_0$ is non-empty and is a
smoothly embedded separating hypersurface consisting of boundary
points of both $M_+$ and $M_-$.  From $(M,{\bf p}, \Phi )$ the curved
orbit components inherit canonical geometric structures as follows:
$M_{+}$ has a nearly K\"{a}hler structure of signature $(2,4)$ that is
positive Einstein; $M_{-}$ has a nearly para-K\"{a}hler structure of
signature $(3,3)$ that is negative Einstein; $M_{0}$ has a conformal
structure of signature $(2, 3)$ equipped with a $\G_2^*$ conformal holonomy reduction, and
this means that the conformal structure is equivalent to an oriented Cartan
$(2,3,5)$-distribution.
\end{itemize}
\end{theorem}
\noindent Some remarks are in order: A {\em nearly
  para-K\"{a}hler geometry} is a pseudo-Riemannian manifold satisfying
\nn{NKdef}, but where $J$ is an involution and $g(J U, J V)=-g(U, V)$.
As with our conventions for nearly K\"{a}hler, our default is that
this is strict, and so here nearly para-K\"{a}hler means that $\nabla J$ is nowhere zero, where $\nabla$ is again the Levi-Civita connection of $g$. It is well-known that
$6$-dimensional strictly nearly K\"{a}hler and strictly nearly para-K\"{a}hler structures
are necessarily Einstein \cite{Gray76,IZ,S}.  That a Cartan $(2,3,5)$-distribution
is equivalent to a $\G_2^*$-reduced conformal structure is a result of
Nurowski \cite{NurowskiDifferential}, with further clarification and
characterization given in \cite{HammerlSagerschnig}. These results
play an important role here. Note that Theorem \ref{NKtoG2} combined
with the result here shows that on a 6-manifold a Riemannian nearly
K\"{a}hler structure is simply equivalent to a projective structure with
a parallel definite-generic 3-form tractor.
More generally we see
that nearly K\"{a}hler geometry, its para- variant, and Cartan
$(2,3,5)$-geometry arise in a uniform way from projective geometry.

Given Theorems \ref{cify} and \ref{orbit}, it is natural to ask to
whether all $(2,3,5)$-distributions arise this way. The answer is
positive in the real-analytic setting (and in general formally), as
explained in Section \ref{gDp}. That section uses results from
\cite{FeffermanGraham} and \cite{GrahamWillse} to treat the problem of
taking a distribution as Dirichlet data for the construction of a
projective manifold with a $\G^*_2$ holonomy reduction for which the given
distribution is the induced structure on the projective infinity. See
in particular Theorem \ref{theorem:Dirichlet-problem-G2}, which
interprets in the projective tractor setting Theorem 1.1 from
\cite{GrahamWillse}. The latter theorem itself generalizes to all
(oriented, real-analytic) $(2, 3, 5)$-distributions a result in
\cite[\S4]{NurowskiExplicit} about a particular finite-dimensional
family of such distributions; later Leistner and Nurowski proved that metrics in an explicit subset of that family have holonomy equal to $\G_2^*$ \cite{LeistnerNurowski}. Section \ref{section:examples} gives
solutions to the Dirichlet problems for a special class of $(2, 3,
5)$-distributions studied by Cartan \cite[\S9]{CartanPff} (the solutions themselves are essentially equivalent to a special case of the data given in \cite[\S3]{NurowskiDifferential}), and these yield a 1-parameter family of geometries $(M, \mbp, \Phi)$ whose curved orbits are all homogeneous.

These results establish that we may study $(2,3,5)$-geometry {\em
  holographically}, that is using the associated nearly K\"ahler and
nearly para-K\"ahler geometries of Theorems \ref{cify} (also Theorem
\ref{pcify}) and \ref{orbit}.  This is in the spirit of Fefferman and
Graham's Poincar\'e-Einstein program \cite{FeffermanGraham} and the
usual holographic principle as in e.g.\ \cite{GLW,Juhl,Susskind,GLW},
except that it involves projective compactification and not conformal
compactification and so the asymptotics are rather different; see
Section \ref{pEin}.

A first step
in such a holographic treatment is to understand how (in the notation
of Theorem \ref{orbit}) the distribution on $M_0$ arises as a limit of
the ambient almost (para\nobreakdash-)complex structure on the open
curved orbits $M_{\pm}$. This is treated in detail in Section
\ref{mobileO}. There it is shown that the holonomy reduction
determines a smooth object $\bJ$ (see \nn{Jdef}) that is (essentially)
a field of endomorphisms of the standard tractor bundle $\cT$ (defined in Section \ref{trS}).  This gives the almost
(para-)complex structures on $M_\pm$ while
also determining the distribution on $M_0$ as a quotient of its
kernel, see Theorems \ref{local} and
\ref{theorem:J-characterization-D}.  In Section \ref{inf} it is also
shown how many of the properties of the distribution may be deduced
efficiently via $\bJ$ and the naturally accompanying perspective.

The general theory of curved orbit decompositions from
\cite{CGHjlms,CGH} describes how many features of orbit decompositions
of homogeneous spaces carry over to corresponding holonomy reductions
of Cartan geometries modelled on the given symmetry-reduced
homogeneous space. Thus we should expect to understand the
results in Theorem \ref{orbit} partly as realizing curved generalizations of features
of the model. This is the case, and the model is discussed in
Section \ref{modelS}. As explained there, the models for our structures
are the ray projectivization $\mathbb{P}_+(\bbI )$ of the imaginary
octonions $\bbI$, in the definite signature case, and the ray
projectivization $\mathbb{P}_+(\bbI^{*})$ of the imaginary split
octonions $\bbI^{*}$ in the indefinite case. Both $\bbI$ and $\bbI^*$
are algebraically rich structures: The homogeneous
geometries $\smash{\mathbb{P}_+(\bbI^{(*)})}$ include the models for nearly
K\"ahler geometry (of both possible signatures), nearly
para-K\"ahler geometry, and $(2,3,5)$-geometry, as we explain in
Section \ref{modelS}.  The point of presenting the model at that late
stage is that these features of the model are just specializations of
results that hold in more general settings, and treating the general
cases is no more difficult than treating the model from the perspective developed here.

Recall that a Riemannian manifold carries a point-dependent Euclidean
structure that may be viewed as a holonomy reduction of the structure
given by a manifold equipped with an affine connection. In a similar way the
Cartan and tractor machinery enables the imaginary (split) octonionic
algebraic structure, of either of the spaces $\smash{\bbI^{(*)}}$, to
be carried fiberwise in a point-dependent but parallel manner. However
because, in contrast to affine geometry, projective geometry is a
higher-order structure, this parallel algebra interacts algebraically
not just with the tangent bundle but also part of its associated 1-jet
bundle.

Thus the geometries discussed in Theorem \ref{orbit} above are in a
 precise way {\em curved analogues of} $\smash{\mathbb{P}_+(\bbI^{(*)})}$.
On a projective 6-manifold $(M,{\bf p})$ a tractor 3-form $\Phi$ that
is pointwise generic determines an algebraic binary cross product
$\boldsymbol{\btimes}$ (see Definition \ref{definition:cross-product}) that corresponds fiberwise to the cross product
on the imaginary octonions, cf. \cite{Baez}. This is preserved by the
tractor connection if and only if $\Phi$ is parallel, and hence we
have the following paraphrasing of Proposition \ref{crossprod}:
\begin{proposition}\label{timesth}
Suppose that $(M,{\bf p})$ is a dimension 6 projective manifold. A
generic parallel 3-form tractor $\Phi$  is equivalent to a
 tractor cross product $\boldsymbol{\btimes} : \cT\times
\cT\to \cT $ that is preserved by the tractor connection.
\end{proposition}
\noindent
Thus we may take $(M,{\bf p},\btimes)$ as the fundamental
structure. This has considerable aesthetic appeal, but it is also
practically useful, and after Section \ref{cross-section} much of the
development is based on this point of view. For example $\bbJ$ is
defined via $\btimes$ and then its key properties follow easily from
cross-product identities. These and similar results are developed in
the next section where we introduce the tools that underlie the
algebraic aspects of the article.

Since the work here involves a number of geometric structures,
our aim is to make the treatment as self-contained as possible.
Throughout we shall use either index-free notation or Penrose's abstract
index notation according to convenience. In the latter, a vector $\xi$ is denoted by $\xi^a$, a covector $\eta$ by $\eta_a$, a covariant tensor of rank $r$ by $T_{a_1 \cdots a_r}$, and similarly for mixed and contravariant tensors; we use the same notation for tangent, cotangent, and general tensor fields on a smooth manifold. We denote the natural pairing of a vector $\xi$ and covector $\eta$ by $\eta_a \xi^a$ and general tensor traces analogously. In Section \ref{trS} we extend this notation from the tangent bundle to tractor bundles. We use (concrete) frames in Section \ref{section:examples}.

The authors are grateful to Pawe{\l} Nurowski, who, at an early stage
in this project, pointed out Theorem \ref{NKtoG2} and a proof via
exterior differential systems (in fact this was in an early draft of
\cite{GN}). We are also grateful to Paul-Andi Nagy who assisted
greatly in the proof of Proposition \ref{prop:weyl_id}.  It is also a pleasure to
thank Robert Bryant for several comments, in particular for discussion
connected to the generality of strictly N(P)K structures as in Remark
\ref{remark:generality-of-structures}. We are also thankful to Antonio
Di Scala, who pointed out a gap in the proof of a variant of Theorem
\ref{cify} that appeared in a first version of the article, and to
G.\ Manno for allowing us to use a result from their preprint work
\cite{DiScalaManno}.  Discussions with Michael Eastwood are also much
appreciated.  The explicit data for the family of examples described
in \S\ref{section:examples} was produced in part using the standard
Maple package \verb"DifferentialGeometry".

\section{Algebraic preliminaries} \label{alg}

\subsection{$\varepsilon$-complex structures on vector spaces}\label{subsection:epsilon-complex-structures}

We review some variants of the notion of a complex structure on a vector space. By applying appropriate sign changes, one can define so-called paracomplex analogues of more familiar complex structures. Both here and in Subsections \ref{convs}-\ref{subsection:almost-eps-Hermitian}, where we define related geometric structures on tangent bundles, we define both kinds of structures simultaneously using a parameter $\varepsilon \in \set{\pm 1}$: In the definitions, $\varepsilon = -1$ yields the complex version of a structure and $\varepsilon = +1$ the paracomplex version. One specializes the names of structures to particular values of $\varepsilon$ by simply omitting $-1$- and replacing $+1$- with the prefix \textit{para-}. See \cite{Cruceanu-Fortuny-Gadea} for a survey of paracomplex geometry.

\begin{definition}\label{definition:epsilon-complex-structures}
The \textit{$\varepsilon$-complex numbers} is the ring $\bbC_{\varepsilon}$ generated over $\bbR$ by the single generator $i_{\varepsilon}$, which satisfies precisely the relations generated by $i_{\varepsilon}^2 = \varepsilon$. As an $\bbR$-algebra, the ring $\bbC_{+1}$ of paracomplex numbers is isomorphic to $\bbR \oplus \bbR$.

An \textit{$\varepsilon$-complex structure} on a real vector space $\bbW$ (of necessarily even dimension, say, $2m$) is an endomorphism $J \in \End(\bbW)$ such that
\[
    J^2 = \varepsilon \id_{\bbW} \textrm{;}
\]
if $\varepsilon = +1$, we require furthermore that the $(\pm 1)$-eigenspaces of $J$ both have dimension $m$ (the analogous condition holds automatically for $\varepsilon = -1$). This identifies $\bbW$ with $\bbC_{\varepsilon}^m$ so that the action of $J$ coincides with scalar multiplication by $i_{\varepsilon} \in \bbC_{\varepsilon}$.

An \textit{$\varepsilon$-Hermitian structure} on a real vector space $\bbW$ (again of necessarily even dimension) is a pair $(g, J)$, where
\begin{enumerate}
    \item $g \in S^2 \bbW^*$ is an inner product (in this article, inner products are not necessarily definite unless specified otherwise), and
    \item $J \in \End(\bbW)$ is an $\varepsilon$-complex structure on $\bbW$,
\end{enumerate}
compatible in the sense that
\[
    g(J \pdot , J \pdot) = -\varepsilon g(\pdot , \pdot) \textrm{,}
\]
or, equivalently, that $\omega := g(\pdot , J \pdot)$ is skew-symmetric.
\end{definition}

The compatibility condition imposes restrictions on the signature of an inner product $g$ in a $\varepsilon$-Hermitian structure: For a Hermitian structure $(g, J)$ on a real vector space $\bbW$ of dimension $2 m$, $g$ must have signature $(2 p, 2 q)$ for some nonnegative integers $p, q$, and for a para-Hermitian structure, $g$ must have neutral signature $(m, m)$. (This compatibility condition necessitates the eigenspace condition in the definition of a paracomplex structure.)

\subsection{The octonions, $7$-dimensional cross products, and $\smash{\G_2^{(*)}}$}\label{subsection:octonions-cross-product-g2}

The geometries investigated in this article will be unified by so-called cross products $\btimes: \bbV \times \bbV \to \bbV$ on $7$-dimensional vector spaces.

\begin{definition}\label{definition:cross-product}
On an inner product space $(\bbV, \pdot)$, a \textit{(binary) cross product} is a skew-symmetric bilinear product $\btimes: \bbV \times \bbV \to \bbV$ {\em compatible} with $\cdot$, meaning  that
\begin{enumerate}
    \item \label{item:cross-product-orthogonality} $(x \btimes y) \cdot x = 0$ and
    \item \label{item:cross-product-square}        $(x \btimes y) \cdot (x \btimes y) = (x \cdot x)(y \cdot y) - (x \cdot y)^2$
\end{enumerate}
for all $x, y \in \bbV$.
\end{definition}

In dimension $7$, there are only two such products up to algebra
isomorphism \cite[Theorem 4.1]{BrownGray}. We construct both
simultaneously, one in terms of the octonion algebra $\bbO$, the most
complicated of the four algebras in the celebrated classification of
normed division algebras over $\bbR$, and the other using the split
octonion algebra $\bbO^*$, a close analog of $\bbO$ in which the norm
is replaced by a quadratic form that induces a split signature inner
product. For convenience we write $\bbO^{(*)}$ to indicate that a
given context applies to both. We use the abstract properties of these
algebras to establish characteristics of the cross products; in
Section \ref{mobileO} transferring these features to the curved
geometries under study will efficiently illuminate some of their
important features. See \cite{Baez} for the details about $\bbO$, and
\cite{BaezHuerta} for facts about $\bbO^*$, though as we will see,
many of the features of the two algebras are analogous.

The algebra $\bbO^{(*)}$ has an identity, which we denote $1$, and it
is alternative but not associative; {\em alternativity} means that any
subalgebra generated by two elements is associative, or equivalently
that the associator $(x, y, z) \mapsto (xy)z - x(yz)$ is totally skew.

The norm on $\bbO$ induces a positive definite inner product $\cdot$, and the inner product on $\bbO^*$, which we also denote $\cdot$, has signature $(4, 4)$. In both cases, the inner product defines a natural \textit{conjugation} involution $\bar{\cdot}: \bbO^{(*)} \to \bbO^{(*)}$ by orthogonal reflection through the span $\ab{1}$ of $1$, which we identify with $\bbR$:
\[
    \bar{x} := 2(x \cdot 1) - x \textrm{.}
\]
It satisfies $\ol{xy} = \bar{y} \bar{x}$. We call the $7$-dimensional orthocomplement $\bbI^{(*)} := \ab{1}^{\perp}$, which is precisely the $-1$-eigenspace of the conjugation map, the \textit{imaginary (split) octonions}. The (hence orthogonal) respective projections onto the summands of the decomposition $\bbO^{(*)} \cong \bbR \oplus \bbI^{(*)}$ are thus
\begin{align*}
	&\Re: x \mapsto \tfrac{1}{2}(x + \bar{x}) = x \cdot 1 \\
	&\Im: x \mapsto \tfrac{1}{2}(x - \bar{x}) \textrm{,}
\end{align*}
and we can recover the bilinear form $\cdot$ via
\[
    x \cdot y = \Re(x \bar{y}) \textrm{.}
\]

Now, $\cdot$ restricts to a nondegenerate bilinear form, which we also denote $\cdot$, on $\bbI^{(*)}$ of signature $(7, 0)$ or $(3, 4)$, and it specializes there immediately to $x \cdot y = -\tfrac{1}{2}(xy + yx)$. The projection $\Im$ determines the nonassociative, $\bbR$-linear, skew-symmetric \textit{(split) octonionic cross product}
\[
    \btimes: \bbI^{(*)} \times \bbI^{(*)} \to \bbI^{(*)}
\]
by
\[
	(x, y) \mapsto x \btimes y := -\Im(x \bar{y}) = \tfrac{1}{2}(xy - yx) \textrm{,}
\]
which realizes $\bbI^{(*)}$ as an anticommutative, nonassociative algebra over $\bbR$ without unit.

By definition, for all $x, y \in \bbI^{(*)}$, $xy$ can be decomposed into its $\bbR$ and $\bbI^{(*)}$ components as
\begin{equation}\label{equation:octonion-decomposition}
    xy = -x \cdot y + x \btimes y \textrm{.}
\end{equation}

Now, one can easily reverse this construction and recover the full algebraic structure of $\bbO^{(*)}$ from the cross product $\btimes$ on $\bbI^{(*)}$. We will not need the details of this construction, but note that the bilinear form on $\bbI^{(*)}$ is determined by the cross product via \cite[\S4.1]{Baez}
\begin{equation}\label{equation:dot-from-cross}
	x \cdot y = -\tfrac{1}{6} \tr (x \btimes (y \btimes \cdot)) \textrm{.}
\end{equation}
So, to specify a cross product $\btimes$ with an underlying $7$-dimensional inner product space $(\bbV, \pdot)$, it is enough just to specify $(\bbV, \btimes)$. In particular, the cross products constructed on $\bbI$ and $\bbI^*$ are nonisomorphic, and since there are only two cross products in this dimension up to isomorphism, any (binary) cross product $\btimes$ on a $7$-dimensional real vector space $\bbV$ is isomorphic to either $(\bbI, \btimes)$, in which case we say that $\btimes$ is \textit{definite}, or $(\bbI^*, \btimes)$, in which case we say that $\btimes$ is \textit{split}.

Since the algebraic structure on $\bbO$ and its corresponding cross product can each be recovered from the other, the two algebras have the same automorphism group, namely, the compact real form $\G_2$ of the simple complex Lie group of type $\G_2$. The other cross product is analogously related to $\bbO^*$, and in this case the common automorphism group is the split real form $\G_2^*$ of that complex Lie group.

Since by \eqref{equation:dot-from-cross} $\cdot$ can be recovered from the algebraic structure of $\bbO^{(*)}$, the $\smash{\G_2^{(*)}}$-action preserves $1$ and hence the orthocomplement $\bbI^{(*)} = \ab{1}^{\perp}$, which is the smallest nontrivial irreducible representation of $\smash{\G_2^{(*)}}$. In particular it preserves the inner product induced there, and this defines a homomorphism (in fact, an embedding) $\G_2 \hookrightarrow \SO(\bbI) \cong \SO(7, \bbR)$ in the definite case and such a map $\G_2^* \hookrightarrow \SO(\bbI^*) \cong \SO(3, 4)$ in the split case.

A cross product on $\bbV$ canonically also determines an orientation: The form
\begin{equation}\label{equation:cross-to-volume}
    \vol_{ABCDEFG} := \tfrac{1}{42} \btimes_{K[AB} \btimes^K_{\phantom{K}CD} \btimes_{EFG]} \in \Lambda^7 \bbV^*
\end{equation}
is a volume form for $\cdot\,$; here, indices are raised and lowered with $\cdot\,$.

Now, any element $x \in \bbI^{(*)}$ determines a map $\bbJ_x: \bbI^{(*)} \to \bbI^{(*)}$ defined by
\begin{equation}\label{aJdef}
    \bbJ_x(y) := -x \btimes y \textrm{,}
\end{equation}
which part \eqref{item:cross-product-orthogonality} of Definition \ref{definition:cross-product} guarantees is skew-adjoint with respect to $\cdot\,$. Its properties will play a key role later in establishing features of the geometric structures we study in later sections.

\begin{proposition}\label{proposition:J-squared}
For any $x \in \bbI^{(*)}$,
\begin{equation}\label{equation:J-squared}
    \bbJ_x^2(y) = -(x \cdot x) y + (x \cdot y) x \textrm{.}
\end{equation}
In particular,
\begin{enumerate}
    \item \label{item:J-squared-x-null}       if $x$ is null, then $\bbJ_x^2(y) = (x \cdot y) x$, and
    \item \label{item:J-squared-y-orthogonal} if $y \in \ab{x}^{\perp}$, then $\bbJ_x^2(y) = -(x \cdot x) y$.
\end{enumerate}
\end{proposition}
\begin{proof}
If we expand the alternativity identity $(xx)y = x(xy)$ of $\bbO^{(*)}$ using the decomposition \eqref{equation:octonion-decomposition}, the left-hand size becomes
\[
    (- x \cdot x + x \btimes x) y = -(x \cdot x) y \textrm{,}
\]
and the right-hand side
\begin{align*}
    x(x \btimes y - x \cdot y)
        &= x(x \btimes y) - (x \cdot y) x \\
        &= [ x \btimes ( x \btimes y) - x \cdot (x \btimes y)] - (x \cdot y) x \\
        &=  -x \btimes (-x \btimes y)                          - (x \cdot y) x \\
        &= \bbJ_x^2(y) - (x \cdot y) x \textrm{.}
\end{align*}
Rearranging gives the identity.
\end{proof}

\begin{corollary}
If $x \cdot x = -\varepsilon \in \set{\pm 1}$, then $\bbJ_x \vert_{\ab{x}^{\perp}} \in \End(\ab{x}^{\perp})$  is an $\varepsilon$-complex structure.
\end{corollary}

The behavior of $\bbJ_x$ is especially rich for null $x$. Parts \eqref{item:Rx-subset-ker-J}-\eqref{item:dim-ker-J} of the following proposition are formulated and proved in \cite[Lemma 7]{BaezHuerta} in a different way.

\begin{proposition}\label{proposition:J-x-null}
Suppose $x \in \bbI^*$ is null and nonzero. Then,
\begin{enumerate}
    \item\label{item:Rx-subset-ker-J}  $\ab{x} \subset \ker \bbJ_x$;
    \item\label{item:ker-J-isotropic}  $\ker \bbJ_x$ is isotropic;
    \item\label{item:dim-ker-J}        $\dim \ker \bbJ_x = 3$.
\end{enumerate}
In particular, $x$ determines a proper filtration
\[
    \set{0}
        \subset  \ab{x}
        \subset  \ker \bbJ_x
        \subset (\ker \bbJ_x)^{\perp}
        \subset  \ab{x}      ^{\perp}
        \subset  \bbI^*
\]
whose filtrands respectively have dimension $0$, $1$, $3$, $4$, $6$, and $7$.
\begin{enumerate}
    \item[(d)]\label{item:J-filtration} The map $\bbJ_x$ respects the filtration in that
        \begin{enumerate}
            \item\label{item:J-image}        $\bbJ_x(\bbI^*               ) = (\ker \bbJ_x)^{\perp}$ (that is, $\im \bbJ_x = (\ker \bbJ_x)^{\perp})$,
            \item\label{item:J-X-perp-image} $\bbJ_x( \ab{x}      ^{\perp}) =  \ker \bbJ_x         $, and
            \item                            $\bbJ_x((\ker \bbJ_x)^{\perp}) =  \ab{x}              $.
        \end{enumerate}
\end{enumerate}
\end{proposition}

\begin{proof}
~
\begin{enumerate}
    \item For any $\lambda x \in \ab{x}$, $\bbJ_x(\lambda x)
= -\lambda(x \btimes x) = 0$.
    \item Pick $y \in \ker \bbJ_x$. If $y$ is a multiple $\lambda x$ of $x$, then $y \cdot y = \lambda^2 x \cdot x = 0$; so henceforth suppose it is not. By Proposition \ref{proposition:J-squared},
        \[
            0 = y \btimes \bbJ_x(y) = y \btimes (y \btimes x) = -(y \cdot y) x + (y \cdot x) y \textrm{.}
        \]
        Since $y$ is not a multiple of $x$, both terms in the last expression are zero, and in particular $y \cdot y = 0$.
    \item By Proposition \ref{proposition:J-squared}, $\ker (\bbJ_x^2) = \ab{x}^{\perp}$. So, $\dim \ker \bbJ_x^2 = 6$ and hence $\dim \ker \bbJ_x \geq 3$. On the other hand, by \eqref{item:ker-J-isotropic} $\ker \bbJ_x$ is isotropic, and so $\ker \bbJ_x \leq 3$; thus, equality holds.
    \item ~
        \begin{enumerate}
            \item For any $y \in \bbI^*$ and $z \in \ker \bbJ_x$, the skew-adjointness of $\bbJ_x$ gives
                \[
                    \bbJ_x(y) \cdot z
                        = - y \cdot \bbJ_x(z)
                        = 0 \textrm{,}
                \]
                so $\im \bbJ_x = \bbJ_x(\bbI^*) \subseteq (\ker \bbJ_x)^{\perp}$; equality holds by the Rank-Nullity Theorem.
            \item For any $y \in \ab{x}^{\perp}$, Proposition \ref{proposition:J-squared} gives $\bbJ_x(\bbJ_x(y)) = (x \cdot y) x = 0$, that is, $\bbJ_x(\ab{x}^{\perp}) \subseteq \ker \bbJ_x$. We can then view $\bbJ_x \vert_{\ab{x}^{\perp}}$ as a map $\ab{x}^{\perp} \to \ab{x}^{\perp}$. Then, $\dim \bbJ_x(\ab{x}^{\perp}) = \rank \bbJ_x \vert_{\ab{x}^{\perp}}$ is at least $\rank \bbJ_x - (\dim \bbI^* - \dim \ab{x}^{\perp}) = 3$, and because $\dim \ker \bbJ_x = 3$, the above containment is actually an equality.
            \item Pick $y \in (\ker \bbJ_x)^{\perp}$. By \eqref{item:J-image}, there is some $z \in \bbI^*$ such that $y = \bbJ_x(z)$, and Proposition \eqref{proposition:J-squared}\eqref{item:J-squared-x-null} gives that
                \[
                    \bbJ_x(y) = \bbJ_x(\bbJ_x(z)) = (x \cdot z) x \textrm{.}
                \]
                So, $\bbJ_x((\ker \bbJ_x)^{\perp}) \subseteq \ab{x}$. On the other hand, $(\ker \bbJ_x)^{\perp}$ is a proper superset of $\ker \bbJ_x$, so its image under $\bbJ_x$ cannot be trivial; hence, equality holds.
                \qedhere
        \end{enumerate}
\end{enumerate}
\end{proof}

\subsection{Stable forms on vector spaces}\label{ss: stable forms}
In this section we review the notion of stability, a type of genericity, for alternating forms on real vector spaces, including some constructions using such forms specific to dimensions $6$ and $7$: A stable $3$-form on an (oriented) real vector space of dimension $6$ induces an $\varepsilon$-complex structure (as in \cite{H2} and as extended to the para-complex case in \cite{CLSSH}), and a stable $3$-form on a real vector space of dimension $7$ determines a cross product there, and hence nondegenerate symmetric bilinear form and orientation \cite{B}.

\begin{definition}
On a real vector space $\bbW$, a $k$-form $\beta \in \Lambda^k \bbW^*$ is \textit{stable} (or \textit{generic} \cite{Agricola}, though we will reserve this term for the analogous notion on manifolds) if and only if the orbit $\GL(\bbW) . \beta$ is open in $\Lambda^k \bbW^*$, where $.$ is the action induced by the standard action on $\bbW$.
\end{definition}

The vector space dimensions and form ranks for which stable forms exist are well-known.

\begin{proposition}
A real vector space $\bbW$ of finite dimension $n > 0$ admits a stable $k$-form if and only if one of the following is true:
\begin{itemize}
	\item $n = 1$ and $k = 1$;
	\item $n \geq 2$ and $k \in \set{1, n - 1, n}$;
	\item $n > 2$, $n$ even, and $k \in \set{2, n - 2}$;
	\item $n \in \set{6, 7, 8}$ and $k \in \set{3, n - 3}$.
\end{itemize}
\end{proposition}

\begin{proof}
For any $k$, the complexification of a $\GL(\bbW)$-module $\Lambda^k \bbW^*$ is an irreducible module $\Lambda^k (\bbW^* \otimes \bbC)$ of the complex Lie group $\GL(\bbW \otimes \bbC)$. With this in hand, one can read off this list from the classification of irreducible prehomogeneous vector spaces over $\bbC$ \cite[Theorem 54]{SatoKimura}.
\end{proof}

\subsubsection{Stable $2$-forms in $2m$ dimensions}
On an even-dimensional vector space $\bbW$, $\Lambda^2 \bbW^*$ has exactly one $\GL(\bbW)$-orbit, and its elements are exactly the symplectic forms on $\bbW$.

\subsubsection{Stable $3$-forms in $6$ dimensions}\label{sss:stable 3 forms in dim 6}

Let $\bbW$ be a real $6$-dimensional vector space; then, $\Lambda^3 \bbW^*$ has exactly two open $\GL(\bbW)$-orbits. (See \cite[Theorem 2.1.13]{Baier} and \cite[Proposition 12]{B2006} for detailed calculations and a complete $\GL(\bbW)$-orbit decomposition of $\Lambda^3 \bbW^*$.)

If $\bbW$ is also oriented, then given a stable $3$-form in $\Lambda^3 \bbW^*$, we can canonically construct an $\varepsilon$-complex structure $J \in \End(\bbW)$: Fix $\beta \in \Lambda^3 \bbW^*$, let $\kappa: \Lambda^5 \bbW^* \to \bbW \otimes \Lambda^6 \bbW^*$ denote the canonical mapping, and define
\[
    \wt{J} := \kappa((\pdot \hook \beta) \wedge \beta) \in \End(\bbW) \otimes \Lambda^6 \bbW^* \textrm{.}
\]
To determine an endomorphism of $\bbW$, we define a volume form invariantly in terms of $\smash{\wt{J}}$: First set
\[
    \lambda(\beta) := \tfrac{1}{6} \tr (\wt{J}^2) \in \otimes^2 \Lambda^6 \bbV^* \textrm{.}
\]
It turns out that $\lambda(\beta) \neq 0$ if and only if $\beta$ stable, which we henceforth assume. Now, $\varepsilon \lambda(\beta)$ is a square of an element in $\Lambda^6 \bbV^*$ for exactly one value $\varepsilon \in \set{\pm 1}$, so let $\vol$ denote the unique positively oriented element there such that $\vol \otimes \vol = \varepsilon \lambda(\beta)$. Then, the endomorphism $J$ characterized by
\[
    \wt{J} = J \otimes \vol
\]
satisfies $J^2 = \varepsilon \id_{\bbW}$. It turns out that if $\varepsilon = 1$, then the $\pm 1$-eigenspaces of $J$ both have dimension $3$, and thus $J$ is an $\varepsilon$-complex structure on $\bbW$.

\subsubsection{Stable $3$-forms in $7$ dimensions}\label{subsubsection:3-forms-in-7-dimensions}

Let $\bbV$ be a real $7$-dimensional vector space; then, $\Lambda^3 \bbV^*$ has exactly two open $\GL(\bbV)$-orbits. (See \cite{B} for detailed calculations, and \cite{CohenHelminck}, which describes a full orbit decomposition of $3$-forms on $7$-dimensional vector spaces over algebraically closed fields and describes a process for generalizing it to some other base fields, including $\bbR$.)

Given a stable $3$-form in $\Lambda^3 \bbV^*$, we can canonically construct an inner product $H \in S^2 \bbV^*$: Fix $\Phi \in \Lambda^3 \bbV^*$, and define the $\Lambda^7 \bbV^*$-valued symmetric bilinear form
\begin{equation}\label{definition:H-tilde}
	\wt{H} := \tfrac{1}{6} (\pdot \hook \Phi) \wedge (\pdot \, \hook \Phi) \wedge \Phi \in S^2 \bbV^* \otimes \Lambda^7 \bbV^* \textrm{.}
\end{equation}
It turns out that $\smash{\wt{H}}$ is nondegenerate if and only if $\Phi$ stable, which we henceforth assume. In particular, $\smash{\wt{H}}$ determines a real-valued bilinear form up to scale; to fix the scale naturally, we define a volume form invariantly in terms of $\Phi$. Regarding $\smash{\wt{H}}$ as a map $\bbV \to \bbV^* \otimes \Lambda^7 \bbV^*$ and taking the determinant yields a map
\[
	\det \wt{H}: \Lambda^7 \bbV \to \Lambda^7 (\bbV^* \otimes \Lambda^7 \bbV^*) \cong \otimes^8 \Lambda^7 \bbV^* \textrm{,}
\]
and dualizing again gives a map
\[
	\det \wt{H}: \bbR \to \otimes^9 \Lambda^7 \bbV^* \textrm{.}
\]
This map turns out to be nonzero because $\Phi$ is stable; so, there is a distinguished volume form $\vol \in \Lambda^7 \bbV^*$ characterized by $\smash{\vol^{\otimes 9} = (\det \wt{H})(1)}$, which in turn determines a nondegenerate bilinear form $H: \bbV \times \bbV \to \bbR$ characterized by
\[
	\wt{H} = H \otimes \vol \textrm{.}
\]
The normalization of $\smash{\wt{H}}$ in \eqref{definition:H-tilde} (and hence of $H$) was chosen so that $\vol$ is the volume form of $H$ for the orientation it determines. Alternatively, for a $3$-form $\Phi$ and the volume form $\vol$ it determines, we can recover $H$ via
\begin{equation}\label{H-from-Phi-vol}
    H_{AB} = \tfrac{1}{144} \Phi_{ACD} \Phi_{BEF} \Phi_{GHI} \vol^{CDEFGHI} \textrm{.}
\end{equation}
Here $\vol$ is normalized so that $\vol_{CDEFGHI}
\vol^{CDEFGHI} = 7!$. For $3$-forms in one of the two
open $\GL(\bbV)$-orbits, $H$ has signature $(7, 0)$; we call such
$3$-forms \textit{definite-stable}. For $3$-forms in the other orbit,
$H$ has signature $(3, 4)$; we call these \textit{split-stable}.

In dimension $7$, stable $3$-forms can be identified with cross products, which motivates here the use of the terms definite and split.

\begin{proposition}\label{equivP}
On any real $7$-dimensional vector space $\bbV$, raising and lowering indices with the corresponding bilinear forms establishes a natural bijection
\[
    \set{\textrm{cross products $\btimes: \bbV \times \bbV \to \bbV$}}
        \leftrightarrow \set{\textrm{stable $3$-forms $\Phi \in \Lambda^3 \bbV^*$}} \textrm{.}
\]
A cross product is definite (split) if and only if the corresponding $3$-form is definite- \mbox{(split-)}stable.
\end{proposition}
\begin{proof}
As observed in Subsection \ref{subsection:octonions-cross-product-g2}, given a cross product $(\bbV, \btimes)$, the $3$-form $\Phi(x, y, z) := x \cdot (y \btimes z)$ is totally skew, and the automorphism group of $\btimes$ is $\smash{\G_2^{(*)}}$. Let $\cdot$ denote the inner product $\btimes$ induces via \eqref{equation:dot-from-cross} and $\smash{\wt{H} \in S^2 \bbV^* \otimes \Lambda^7 \bbV^*}$ the bilinear form defined by \eqref{definition:H-tilde}. As a $\smash{\G_2^{(*)}}$-representation, $S^2 \bbV^* \otimes \Lambda^7 \bbV^*$ decomposes into irreducible subrepresentations as
\[
    (S^2_0 \bbV^* \otimes \Lambda^7 \bbV^*) \oplus (\bbR \otimes \Lambda^7 \bbV^*) \textrm{.}
\]
Since $\cdot \, \otimes \vol$ is $\smash{\G_2^{(*)}}$-invariant and nonzero, it spans the $1$-dimensional subrepresentation $\bbR \otimes \Lambda^7 \bbV^*$ (here, $\vol$ denotes the volume form of $\cdot$ defined in Subsection \ref{subsection:octonions-cross-product-g2}). Likewise, $\smash{\wt{H}}$ is a contraction of $\smash{\G_2^{(*)}}$-invariant tensors, so it too is $\smash{\G_2^{(*)}}$-invariant, and thus it is some multiple of $\cdot \, \otimes \vol$. One can verify readily that $\smash{\wt{H}}$ is nonzero, so it is nondegenerate, and hence $\Phi$ is stable; in fact, our normalizations have been chosen so that the volume form of $\cdot$ and the volume form defined in terms of $\Phi$ above coincide, and thus so do $\cdot$ and $H$. In particular, if $\btimes$ is definite (split), then $\Phi$ is definite (split) stable.

Conversely, given a stable $3$-form $\Phi$, let $H$ denote the bilinear form it induces, and define the product $\btimes: \bbV \times \bbV \to \bbV$ by
\begin{equation}\label{Phi-times}
    \btimes^C_{\phantom{C}AB} := H^{CK} \Phi_{KAB} \textrm{.}
\end{equation}
Immediately, $\btimes$ satisfies condition \eqref{item:cross-product-orthogonality} in Definition \ref{definition:cross-product}, which is all that is needed for the map $\bbJ_x := x \btimes \pdot$ to be skew-adjoint. One can show that $\bbJ_x$ so defined satisfies \eqref{equation:J-squared} (it is enough to prove this for one definite-stable and one split-stable $3$-form); then, forming the inner product of both sides of that identity with $y$ and invoking the skew-adjointness of $\bbJ_x$ gives that $\btimes$ satisfies condition \eqref{item:cross-product-square} of Definition \ref{definition:cross-product} too, and so $\btimes$ is a cross product. Checking directly (again, say, just for one representative of each orbit) shows that the two constructions are inverses.
\end{proof}

\begin{corollary}
Let $\bbV$ be a $7$-dimensional real vector space and $\Phi \in \Lambda^3 \bbV^*$ a stable $3$-form. Then, under the standard action of $\GL(\bbV)$, the stabilizer of $\Phi$ is $\G_2$ if it is definite-stable and $\G_2^*$ if it is split-stable.
\end{corollary}

Finally, we can use a cross product identity to determine $H$ from $\smash{\wt{H}}$ without computing the determinant. Using \eqref{Phi-times}
 to rewrite \eqref{equation:dot-from-cross} and rearranging gives
\begin{equation}\label{PtoH}
    6 H_{AD} = \Phi_{ABC} \Phi_D^{\phantom{D} B C} \textrm{,}
\end{equation}
and contracting with $H^{AD}$ yields
\[
    42 = \Phi_{ABC} \Phi^{ABC} \textrm{.}
\]
For a $3$-form $\Phi \in \Lambda^3 \bbV^*$, scaling $H$ by $\lambda^2$
scales the right-hand size by $\lambda^{-6}$, and so when using $H$ to
raise and lower indices, $H$ is characterized among its positive
multiples by this identity.

\subsubsection{Compatible stable forms}

We now formulate a natural compatibility condition on a pair of stable forms in dimension $6$ under which the constituent forms can be used to construct a cross product (equivalently, a stable $3$-form) in dimension $7$. (See \cite[\S\S1.2-3]{CLSSH} for much more, including proofs of the below propositions.)

\begin{definition}\label{compat}
Let $\bbW$ be a $6$-dimensional real vector space. We say that a pair $(\omega, \beta)$ of a stable
$2$-form $\omega \in \Lambda^2 \bbV^*$ and a stable $3$-form $\beta
\in \Lambda^3 \bbV^*$ is \textit{compatible} if $\omega \wedge \beta = 0$. A compatible pair $(\omega, \beta)$ is \textit{normalized} if $J^* \beta \wedge \beta = \frac{2}{3} \omega \wedge \omega \wedge \omega$.
\end{definition}

Given a compatible pair $(\omega, \beta)$, we can define a nondegenerate bilinear form
\begin{equation}\label{equation:g-definition}
    g := \varepsilon \omega(\pdot , J \pdot) \textrm{,}
\end{equation}
where $J$ is the $\varepsilon$-complex structure induced by $\beta$;
checking directly in a basis
shows that compatibility is equivalent to the pair
$(g, J)$ defining an $\varepsilon$-Hermitian structure on $\bbW$.

\begin{proposition}\cite[Proposition 1.12]{CLSSH} \label{pairtog2}
Let $\bbW$ and $\bbL$ be real vector spaces respectively of dimensions
$6$ and $1$, and denote $\bbV := \bbW \oplus \bbL$. Let $\alpha \in
\bbV^*$ be a nonzero $1$-form that annihilates $\bbW$; let $(\omega,
\beta)$ be a normalized compatible pair on $\bbW$, and identify
$\omega$ and $\beta$ with their respective pullbacks by the decomposition projection $\bbV
\to \bbW$; let $\varepsilon$ be the sign determined by $\beta$, and
let $g$ denote the bilinear form defined by
\eqref{equation:g-definition}.

Then, the $3$-form
\[
    \Phi := \alpha \wedge \omega + \beta \in \Lambda^3 \bbV^*
\]
is stable, and the bilinear form that it induces via the construction above is
\[
    H = g - \varepsilon \alpha \otimes \alpha \in S^2 \bbV^* \textrm{;}
\]
in particular
$\btimes^C_{\phantom{C}AB} := H^{CK} \Phi_{KAB}$ is a cross product.
\end{proposition}

Conversely, a cross product $\btimes$ on a real $7$-dimensional vector space $\bbV$, together with a choice of (pseudo-)unit vector, determines a decomposition $\bbV = \bbW \oplus \bbL$ as above and a compatible, normalized pair on $\bbW$.

\begin{proposition}\cite[Proposition 1.14]{CLSSH}
Let $\btimes$ be a cross product on a $7$-dimensional real vector space $\bbV$, let $H \in S^2 \bbV^*$ be the inner product it induces, so that $\Phi_{ABC} = H_{CK} \btimes^K_{\phantom{K}AB} \in \Lambda^3 \bbV^*$ is the corresponding stable $3$-form. Let $n \in \bbV$ be a vector that satisfies $H(n, n) = -\varepsilon \in \set{\pm 1}$, denote $\bbW := \ab{n}^{\perp}$, and let $\iota$ denote the inclusion $\bbW \hookrightarrow \bbV$. Then, the pair $(\omega, \beta)$ defined by
\[
    \omega := \iota^*(n \hook \Phi) \in \Lambda^2 \bbW^*,
        \qquad
    \beta  := \iota^* \Phi          \in \Lambda^3 \bbW^* \textrm{,}
\]
is a pair of compatible, normalized stable forms, and the bilinear form $g \in S^2 \bbW^*$ the pair determines via \eqref{equation:g-definition} satisfies $g = \iota^* H$.
\end{proposition}

\section{Nearly (para-)K\"ahler geometry}\label{NPK-sec}
In this section we first introduce some basic notions and
constructions for nearly K\"ahler and nearly para-K\"ahler geometry,
with an emphasis on dimension 6. Both structures are closely linked to
an overdetermined natural partial differential equation which, in some
contexts, is called the Killing-Yano equation. This equation and its
prolongation provide the critical link with projective geometry that
we take up in the next section.

\subsection{Conventions for affine and (pseudo-)Riemannian geometry}\label{convs} It will at times be useful to use the abstract index notation $\ce^a$ for the tangent bundle $TM$,  and $\ce_a$ for its dual
$T^*M$.  Given a torsion-free affine connection $\nabla$ on an
$n$-manifold its curvature $R_{ab}{}^c{}_d$ is then defined by
\begin{equation}
(\nabla_a\nabla_b-\nabla_b\nabla_a)U^c=R_{ab}{}^c{}_dU^d, \quad U^d\in
  \Gamma(\ce^d) \textrm{.}
\end{equation}
The Ricci tensor of $\nabla$ is given by $R_{bd}:=R_{ab}{}^a{}_d$.

In particular this applies to the Levi-Civita connection of a metric $g$ of
any signature. In this case we
may also define the scalar curvature ${\rm Sc}=g^{ab}R_{ab}$.

\subsection{Almost $\varepsilon$-Hermitian geometry}\label{subsection:almost-eps-Hermitian}
\newcommand{\ve}{\varepsilon}

\begin{definition}
An {\em almost $\varepsilon$-complex structure} on a (necessarily even-dimensional) manifold $M$ is a linear endomorphism $J \in \End(TM)$ such that, at each $x\in M$, $J_x$ is an $\varepsilon$-complex structure on $T_xM$ (see Section
\ref{subsection:epsilon-complex-structures}).

Correspondingly, an \emph{almost $\varepsilon$-Hermitian manifold} is a triple $(M, g, J)$ where $M$ is an manifold, where $g$ is a (pseudo\nobreakdash-)Riemannian metric and $J$ is an almost complex structure so that for all $x \in M$, $(g_x, J_x)$ is an $\varepsilon$-Hermitian structure on $T_x M$, that is, if
\begin{equation}\label{hermitian_metric}
    g(J \pdot, J \pdot)=-\varepsilon g(\pdot, \pdot) \textit{,}
\end{equation}
By the remarks after Definition \ref{definition:epsilon-complex-structures}, the metric of an almost Hermitian manifold must have signature $(2p, 2q)$ for some nonnegative integers $p, q$, and an almost para-Hermitian manifold must have signature $(m, m)$, where $\dim M = 2m$.

On an almost $\varepsilon$-Hermitian manifold $(M^{2m}, g, J)$, the skew-symmetric 2-form $\omega:=g(\pdot, J \pdot)$ is called the \emph{fundamental 2-form} or \emph{K\"ahler form}. It satisfies the identities
\begin{equation}\nonumber
    \omega(J \pdot, J \pdot)=g(J \pdot, J J \pdot)=-\varepsilon g(\pdot, J \pdot)=-\varepsilon \omega(\pdot, \pdot).
\end{equation}

The \textit{Nijenhuis tensor} $N_J$ of an almost $\varepsilon$-complex structure $J$ is defined by
\begin{equation}\label{nijenhuis}
\begin{split}
    N_J (U, V) := & -\varepsilon [U, V] - [J U, J V] + J [J U, V] + J [U, J V]\\
                 =& -(\nabla_{J U} J) V + (\nabla_{J V} J) U + J(\nabla_U J) V - J(\nabla_V J) U,
\end{split}
\end{equation}
for arbitrary vector fields $U, V$, where $\nabla$ is any torsion-free connection. This tensor is the complete obstruction to the integrability of $J$.
\end{definition}

The following well-known identities are easily checked.
\begin{proposition}\label{ah-ids}
 The
 Levi-Civita connection $\nabla^g$ of an almost $\varepsilon$-Hermitian manifold $(M, g, J)$ satisfies the following identities
 (for arbitrary vector fields $U, V, W$):
\begin{equation}\label{AHformulas}
(\nabla_U J) J V = -J(\nabla_U J) V, \quad \mbox{and}\quad
g((\nabla_U J) V, W) = -(\nabla_U \omega) (V, W).
\end{equation}
\end{proposition}

\subsection{Nearly $\varepsilon$-K\"ahler geometry}

\begin{definition}
An almost $\varepsilon$-Hermitian manifold $(M, g, J)$ is
\emph{nearly $\varepsilon$-K\"ahler} if and only if its Levi-Civita connection
$\nabla$ satisfies
\begin{equation} \label{NKcond}
(\nabla_U J) U=0
\end{equation}
for all $U \in \Gamma(TM)$, or equivalently if the covariant derivative $\nabla \omega$ of the K\"{a}hler form $\omega$ is totally skew. It is \emph{strictly nearly $\varepsilon$-K\"{a}hler} if in addition $\nabla J$ or, equivalently, $\nabla \omega$, is nowhere zero. For brevity we sometimes write nearly K\"{a}hler as NK and nearly para-K\"{a}hler as NPK, and refer to both structures simultaneously using the abbreviation N(P)K.
\end{definition}

It turns out that if the dimension of an N(P)K manifold $(M, g, J)$ is
less than 6, then \nn{NKcond} implies that $\nabla J=0$ and hence that
the manifold is $\varepsilon$-K\"ahler \cite[Theorem 4.4(v)]{GrayCrossProductsManifolds}.
The definition \eqref{nijenhuis} and the second equation of
\eqref{AHformulas} together give that the Nijenhuis tensor of a nearly
$\varepsilon$-K\"ahler manifold is
\begin{equation}\label{nk_nijenhuis}
N_{J}(U,V) =4J (\nabla_U J)V \textrm{.}
\end{equation}

Next, since the Levi-Civita connection is torsion free, on any N(P)K manifold
\begin{equation}\label{ndomskew}
    d\omega = 3 \nabla \omega \textrm{.}
\end{equation}
\begin{lemma}\cite[Lemma 2.5]{MNS}\label{lemma_fields} For any vector fields $U$ and $V$ on $M$, the vector field $(\nabla_U J) V$ is orthogonal to $U, J U, V$ and $J V$.
\end{lemma}

Generalizing a well-known construction in nearly K\"ahler geometry we define,
for any nearly $\varepsilon$-K\"ahler manifold, a \emph{canonical
} $\varepsilon$\emph{-Hermitian connection}: This is the unique
connection $\bar{\nabla}$ with (totally) skew symmetric torsion that
preserves the metric $g$ and the almost $\varepsilon$-complex structure
$J$ (see \cite{IZ, S}). Explicitly, it is
\begin{equation}\label{canonical_connection}
\bar{\nabla}_U V=\nabla_U V +\tfrac{1}{2}\varepsilon J(\nabla_U J) V, \mbox{ for } U, V \in \Gamma(TM).
\end{equation}
The torsion of $\bar{\nabla}$ is then $\bar{T}(U, V)=\varepsilon
J(\nabla_U J) V=\frac{1}{4} \varepsilon N_J(U, V)$, where the last
identity follows from (\ref{nk_nijenhuis}).

\begin{proposition}\cite[Corollary 3.7]{SSH}\label{Tconst}
For any nearly $\varepsilon$-K\"{a}hler structure, $\ab{\nabla J, \nabla J}$ is constant.
\end{proposition}

\subsection{Dimension six}\label{ss: N(P)K dim 6}
Henceforth we restrict our discussion of nearly $\varepsilon$-K\"ahler
manifolds to the case that $M$ has dimension 6 and the structure is
strict, for which much stronger results are available.

A nearly $\varepsilon$-K\"{a}hler 6-manifold $(M, g, J)$ is of \emph{constant type} \cite[Theorem 5.2]{Gray76}, i.e. there is a constant $\alpha\in\mathbb{R}$ such that
\begin{equation}\label{constant_type}
    g((\nabla_U J) V, (\nabla_U J) V)=\alpha [ g(U, U) g(V, V)-g(U, V)^2+\varepsilon g(J U, V)^2 ].
\end{equation}
It is also well-known that a Riemannian strictly nearly K\"ahler manifold
is Einstein \cite{Gray76}, and
the same holds true for pseudo-Riemannian
strictly nearly K\"{a}hler \cite{S} and strictly nearly
para-K\"{a}hler structures \cite{IZ}.
 In particular we have
\begin{equation}\label{6snk_einstein}
R_{ab}=5 \alpha g_{ab},
\end{equation}
where $\alpha$ is the  constant in \nn{constant_type}.

When $\ab{\nabla J, \nabla J} \neq 0$, Lemma \ref{lemma_fields}
allows us to use adapted frames $(e_i)$ which are
convenient for local calculations.  Take $e_1$ and $e_3$ to be any two
orthonormal local vector fields such that $e_3\neq \pm J e_1$ and define
\begin{equation}
\begin{array}{l l l}
e_2:=J e_1, &\hspace{5mm} &e_4:=J e_3,\\
e_5:=\vert\alpha\vert^{- 1 / 2}(\nabla_{e_1}J)e_3, & & e_6:=J e_5 \textrm{.}
\end{array}
\end{equation}
Using this frame (and following \cite{CLSSH}) we can easily calculate
\begin{equation}
\begin{split}
\omega=&\; e^{12}+e^{34}+e^{56},\\
\nabla\omega=&\; e^{135}+\varepsilon (e^{146}+e^{236}+e^{245}),\\
\ast(\nabla\omega)=&\; -e^{246} -\varepsilon (e^{235}+e^{145}+e^{136}),\\
J^*(\nabla\omega)=&\; e^{246} +\varepsilon (e^{235}+e^{145}+e^{136}).
\end{split}
\end{equation}
Computing gives that $\nabla \omega\wedge \omega=0$ and
$J^*(\nabla\omega)\wedge\nabla\omega=\frac{2}{3}\omega\wedge \om
\wedge \om$. If $\varepsilon = -1$, then at each point $x \in M$ the
representation of $(\nabla \omega)_x$ in the coframe $(e^i)$ coincides
with the $3$-form given in \cite[Proposition 12(2)]{B2006}, which that
proposition shows is generic. If $\varepsilon = +1$, then we can find
a coframe $(f^i)$ such that $\nabla \omega = f^{123} + f^{456}$, which
at each point coincides with the $3$-form in part (1) of that
proposition, where it is shown that it, too, is generic. In both
cases, then, $\nabla \omega$ is stable, and hence $(\omega, \nabla
\omega)$ is a stable, compatible pair. In summary:

\begin{proposition}\label{stab}
If $(M,g,J)$ is a strictly N(P)K manifold such that $\langle \nabla J, \nabla J
\rangle \neq 0$, then $(\omega, \nabla\omega)$ is a pair of stable,
compatible, and normalized forms in that pointwise they satisfy
Definition \ref{compat}.
\end{proposition}

\begin{remark}
The hypothesis $\ab{\nabla J, \nabla J} \neq 0$ here is in fact an
essential nondegeneracy condition: For a strictly N(P)K structures, it
holds if and only if $\nabla \omega$ is stable, so structures that do not satisfy
that condition do not give rise to the compatible pairs as in
Proposition \ref{stab}. This is critical later, in the proof of
Theorem \ref{pcify}; that theorem generalizes Theorem \ref{cify} to include
the strictly nearly para-K\"{a}hler case. It turns out that strictly
N(P)K structures with $\ab{\nabla J, \nabla J} = 0$ are necessarily
nearly para-K\"{a}hler \cite{Gray76,Ka-Hab,S}. The first examples
of such structures were constructed recently by Sch\"{a}fer
\cite{SchaferCones}.
\end{remark}

Finally, we state an algebraic identity for nearly $\ve$-K\"ahler
manifolds that turns out to have critical consequences in the next
section. It is convenient to state this result here, although the  projective Weyl tensor $W_{ab}{}^c{}_d$ is defined in \nn{pweyl}
of the next section.
\begin{proposition}\label{prop:weyl_id}
Let $(M, g, J)$ be a $6$-dimensional strictly nearly
(para\nobreakdash-)K\"ahler manifold. Its K\"{a}hler form $\omega$ and
its projective Weyl curvature $W$ satisfy
\begin{equation}\label{weyl_id}
\omega_{k[b}^{\phantom{k} }W_{cd]\phantom{k}a}^{\phantom{cd]}k}=0.
\end{equation}
\end{proposition}
\begin{proof}
A $6$-dimensional N(P)K manifold is Einstein, and on an Einstein
manifold the projective Weyl tensor $W_{ab}{}^c{}_d$ is equal to the
conformal Weyl tensor $C_{ab}{}^c{}_d$ (the completely trace-free part
of the Riemann tensor) (see \cite[Proposition 5.5]{EM}). So we can
rewrite (\ref{weyl_id}) in index-free notation as
$\underset{S,T,U}{\mathfrak{S}}C(S,T,J U,V)=0$. Here and below
$S,T,U,V$ are any smooth tangent vector fields.

In order to prove this identity we need essentially three
formulae. The first is the standard conformal Weyl-Schouten 
 decomposition of the Riemann tensor
 $$
 R(S, T,  U,  V)= C(S, T,  U, V)+  g(U ,S ){\sf P} (T,V)-g(U,T ){\sf P} (S,V)+
 g(V ,T ){\sf P} (S,U)-g(V,S ){\sf P} (T,U),
 $$
which defines the {\em conformal Schouten tensor} ${\sf P}$ (in all dimensions $n\geq 3$), see 
also \eqref{equation:conformal-schouten}. 
The second is
\begin{equation}\label{curv}
R(S, T, J U, J V)=-\varepsilon R(S, T,  U, V)+\varepsilon g((\nabla_S J)T,(\nabla_U J)V) ;
\end{equation}
which is proved in the NK case in \cite[(3)]{Gray70}, \cite[(1.1)]{S}, and in the NPK case in \cite[Prop. 5.2]{IZ}.
The third is the polarization of the constant type formula (see \eqref{constant_type}):
\begin{equation}\label{constant_type_pol}
\begin{split}
g((\nabla_S J)T,(\nabla_U J) J V)=\alpha [ & g(S, U) g(T, J V)-g(S, J V) g(T, U)\\
&\hspace{2mm}+g(S, J U) g(T, V)-g(S, V) g(T, J U) ].
\end{split}
\end{equation}
Replacing $V$ with $J V$ in (\ref{curv}) and cycling it in $S, T, U$ we obtain
\begin{equation}\label{curv_cycling}
\underset{S,T,U}{\mathfrak{S}}R(S, T, J U, V)=\underset{S,T,U}{\mathfrak{S}}g((\nabla_S J)T,(\nabla_U J)J V).
\end{equation}

Using the fact that $M$ is Einstein we have $\Ric=\frac{\Sc}{6} g$,
and hence for the conformal Schouten tensor we have ${\sf
  P}=\frac{\Sc}{60} g$.  On the other hand $\Ric=5\alpha g$, and so
$\alpha=\frac{\Sc}{30}$. Substituting this into the Weyl-Schouten
decomposition ( or equivalently the Ricci decomposition) yields
\begin{equation}
R(S, T, JU, V)
    = C (S, T, JU, V) + \alpha [ g(S, J U) g(T, V)-g(T, J U) g(S, V) ].
\end{equation}
Now, cycling this identity in $S, T, U$, using (\ref{curv_cycling}), and comparing with (\ref{constant_type_pol}) gives
\begin{equation}
\underset{S,T,U}{\mathfrak{S}} C (S, T, JU, V)=\alpha \underset{S,T,U}{\mathfrak{S}} [ g(S, U) g(T, J V)-g(S, J V) g(T, U) ].
\end{equation}
Expanding the right-hand side and cancelling show us that (\ref{weyl_id}) holds.
\end{proof}

\subsection{Link with the Killing equation on a 2-form}\label{KYsec}
Recall that on a pseudo-Riemannian manifold $(M,g)$, the infinitesimal
isometries are precisely the solutions $U \in \Gamma(TM)$ of the
Killing equation $\mcL_U g = 0$. We can
rewrite this equation as $\nabla_{(b} U_{c)}=0$.  Similarly, lowering an index using $g$, the condition \nn{NKcond} on $J$ (which partially defines
a nearly K\"{a}hler structure) yields the following equivalent overdetermined PDE, which is is usually called the \emph{Killing-Yano
  equation}:
\begin{equation}\label{equation:Killing-Yano}
  \nabla_{a}\omega_{bc}+\nd_b\om_{ac}=0 \quad \mbox{or equivalently} \quad  \nabla_{(b}\omega_{c)d}=0 \textrm{.}
\end{equation}

\subsection{Prolonging the Killing-Yano equation}\label{prolong}

The Killing-Yano equation \eqref{equation:Killing-Yano} depends only a connection, so it can be regarded as
an equation on general affine manifolds. So, in this
subsection we work in the general setting of a manifold $M$ of
dimension $n \geq 2$
equipped with a torsion-free affine connection
$\nabla$.  For simplicity, we assume that $\nabla$ is {\em special},
that is, that locally it preserves a volume form---in Subsection \ref{pdg} we will see that for our purposes this is no restriction at all. Of course,
the Levi-Civita connection of any metric is special.

In this context we shall prolong the equation
\begin{equation}\label{KYbgg}
\nabla_b\omega_{c
  d}+\nabla_c\omega_{b d}=0
\end{equation}
where $\om$ is an arbitrary 2-form. Prolongation involves the
introduction of new variables in a way to replace a differential equation with a simpler system. For this equation, doing so will also expose a strong link with projective geometry.

On an affine manifold $(M,\nabla)$
the curvature $R_{ab}{}^c{}_d$  may be decomposed as
\begin{equation}\label{pweyl}
R_{ab\phantom{c}d}^{\phantom{ab}c}
=W_{ab\phantom{c}d}^{\phantom{ab}c}+\delta_{\phantom{c} a}^{c}{\sf P}_{bd}-\delta_{\phantom{c} b}^{c}{\sf P}_{ad} \textrm{,}
\end{equation}
where
\[
    {\sf P}_{ab}:=\frac{1}{n-1}R_{ab}
\]
is the \textit{projective Schouten tensor} of $\nabla$ and $W_{ab}{}^c{}_d$ is the {\em projective Weyl tensor} of $\nabla$. The Weyl tensor is totally trace-free, that is, it satisfies the identities $\delta^a_{\phantom{a}c} W_{ab}{}^c{}_d=0$ and $\delta^d_{\phantom{d}c}
W_{ab}{}^c{}_d=0$. These objects have special roles in projective geometry,
which we exploit in the next section.

As a first step toward prolonging \nn{KYbgg} we differentiate it
using the connection $\nabla_a$ to obtain
\begin{equation}\nonumber
\nabla_a \nabla_b\omega_{c d}+\nabla_a \nabla_c\omega_{b d}=0.
\end{equation}
Cycling on $a,b,c$ gives
\begin{equation}\nonumber
\begin{split}
\nabla_b \nabla_c\omega_{a d}+\nabla_b \nabla_a\omega_{c d}=&\,0,\\
\nabla_c \nabla_a\omega_{b d}+\nabla_c \nabla_b\omega_{a d}=&\,0.
\end{split}
\end{equation}
Now adding the first two equations and subtracting the third we have:
\begin{equation}\label{prol1}
2\nabla_a \nabla_b\omega_{c d}-[\nabla_a,\nabla_b]\omega_{c
  d}+[\nabla_a,\nabla_c]\omega_{b d}+[\nabla_b,\nabla_c]\omega_{a
  d}=0.
\end{equation}

Next, using the identity
\[
[\nabla_a,\nabla_b]\omega_{c d}=R_{a b}\sharp \omega_{c d}
:=-R_{ab\phantom{k}c}^{\phantom{ab}k}\omega_{kd}-R_{ab\phantom{k}d}^{\phantom{ab}k}\omega_{ck}
\]
and the First Bianchi Identity we can rewrite \eqref{prol1} as
\begin{equation}\nonumber
2\nabla_a\nabla_b\omega_{cd}+2 R_{bc\phantom{k}a}^{\phantom{bc}k}\omega_{dk}-R_{ba\phantom{k}d}^{\phantom{ba}k}\omega_{ck}-R_{ac\phantom{k}d}^{\phantom{ac}k}\omega_{bk}-R_{bc\phantom{k}d}^{\phantom{bc}k}\omega_{ak}=0 \textrm{,}
\end{equation}
and cycling it on $b,c,d$ gives
 \begin{equation}\nonumber
 2\nabla_a\nabla_c\omega_{db}+2 R_{cd\phantom{k}a}^{\phantom{cd}k}\omega_{bk}-R_{ca\phantom{k}b}^{\phantom{ca}k}\omega_{dk}-R_{ad\phantom{k}b}^{\phantom{ad}k}\omega_{ck}-R_{cd\phantom{k}b}^{\phantom{cd}k}\omega_{ak}=0,
 \end{equation}
 \begin{equation}\nonumber
 2\nabla_a\nabla_d\omega_{bc}+2 R_{db\phantom{k}a}^{\phantom{db}k}\omega_{ck}-R_{da\phantom{k}c}^{\phantom{da}k}\omega_{bk}-R_{ab\phantom{k}c}^{\phantom{ab}k}\omega_{dk}-R_{db\phantom{k}c}^{\phantom{db}k}\omega_{ak}=0.
 \end{equation}
Adding the last three equations, using $\nabla_{(a}\omega_{b)c}=0$, and again applying the First Bianchi Identity yields
\begin{equation}\nonumber
\nabla_a\nabla_b\omega_{cd}+\tfrac{1}{2} ( R_{bc\phantom{k}a}^{\phantom{bc}k}\omega_{dk}+ R_{cd\phantom{k}a}^{\phantom{cd}k}\omega_{bk}+R_{db\phantom{k}a}^{\phantom{db}k}\omega_{ck} ) =0.
\end{equation}

Using the placeholder variable $\mu_{abc}:=\nabla_a\omega_{bc}$, where $\mu_{abc}$ is skew, we rewrite the above equation as the first-order system
\begin{equation}\nonumber
\left\{
    \begin{array}{rl}
        0 \!\!\!&= \nabla_a\omega_{bc}-\mu_{abc} \\
        0 \!\!\!&= \nabla_a\mu_{bcd}+\tfrac{1}{2} ( R_{bc\phantom{k}a}^{\phantom{bc}k}\omega_{dk}+R_{cd\phantom{k}a}^{\phantom{cd}k}\omega_{bk}+R_{db\phantom{k}a}^{\phantom{db}k}\omega_{ck} )
     \end{array}
\right.
\textrm{.}
\end{equation}
Thus (cf.\ \cite{BCEG}) solutions of the equation \nn{KYbgg}
correspond to pairs $\Sigma := (\omega, \mu)$ parallel with respect to
the connection $\widehat{\nabla}$,
where
\begin{equation}\label{prol_step1}
\widehat{\nabla}_a\left(\begin{array}{c}
\omega_{b c}\\
\mu_{b c d}
\end{array}\right)=
\left(\begin{array}{c}
\nabla_a\omega_{b c}-\mu_{a b c}\\
\nabla_a\mu_{b c d}+\frac{1}{2} ( R_{bc\phantom{k}a}^{\phantom{bc}k}\omega_{dk}+ R_{cd\phantom{k}a}^{\phantom{cd}k}\omega_{bk}+R_{db\phantom{k}a}^{\phantom{db}k}\omega_{ck} )
\end{array}\right).
\end{equation}
This leads to the following result.
\begin{proposition}\label{prolongprop} On a manifold of dimension $n \geq 2$
equipped with a torsion-free special affine connection $\nabla$,
solutions of the equation
$$
\nabla_b\omega_{c
  d}+\nabla_c\omega_{b d}=0,
$$
on 2-form fields $\om_{bc}$, are in 1-1 correspondence with sections $(\omega, \mu)$ of
$\Lambda^2 T^* M \oplus \Lambda^3 T^*M$ which are parallel for the connection \begin{equation}\label{prol_step2}
\widehat{\nabla}_a\left(\begin{array}{c}
\omega_{b c}\\
\mu_{b c d}
\end{array}\right)=
\left(\begin{array}{c}
\nabla_a\omega_{b c}-\mu_{a b c}\\
\nabla_a\mu_{b c d}+3 {\sf P}_{a[b}\omega_{cd]}
\end{array}\right)
-\frac{1}{2}\left(\begin{array}{c}
0\\
3\,\omega_{k[b}^{\phantom{k} }W_{cd]\phantom{k}a}^{\phantom{cd]}k}
\end{array}\right).
\end{equation}
\end{proposition}
\begin{proof}
Using the decomposition (\ref{pweyl}) on $\tfrac{1}{2} ( R_{bc\phantom{k}a}^{\phantom{bc}k}\omega_{dk}+ R_{cd\phantom{k}a}^{\phantom{cd}k}\omega_{bk}+R_{db\phantom{k}a}^{\phantom{db}k}\omega_{ck} )$, it is a straightforward calculation to check that the sum of the terms involving the Weyl curvature is $-\tfrac{3}{2}\, \omega_{k[b}^{\phantom{k} }W_{cd]\phantom{k}a}^{\phantom{cd]}k}$, while the sum of the terms involving the Schouten tensor is $3\,{\sf P}_{a[b}\,\omega_{cd]}$.
Thus solutions of \nn{KYbgg} yield
sections parallel for $\widehat{\nabla}$. On the other hand if
$(\om_{bc},\mu_{bcd})$ is parallel for $\widehat{\nabla}$ then
$\nabla_a\omega_{b c}=\mu_{a b c}$; in particular $\nabla_a \omega_{bc}$ is totally skew, and so \nn{KYbgg} holds.
\end{proof}

\section{Projective geometry and N(P)K-structure}\label{engine}

We will show that strictly nearly K\"{a}hler and strictly nearly para-K\"{a}hler
structures in dimension 6 have a natural interpretation in projective
geometry, and that this facilitates links to other geometries.  The
structures treated in this section and the subsequent sections can
only exist on orientable manifolds, and so henceforth we shall
assume $M$ orientable.

\subsection{Projective differential geometry}\label{pdg}

As mentioned in the introduction a projective structure ${\bf p}$ on a
manifold $M$ (of dimension $n\geq 2$) consists of an equivalence class
of torsion-free affine connections that share the same geodesics, as
unparameterized curves.  The class is equivalently characterized by the
fact that, acting on any $U \in \Gamma(TM)$, any two connections
$\nabla$ and $\widehat\nabla$ in ${\bf p}$ are related by a
transformation of the form
 \begin{equation} \label{ptrans}
 \widehat\nabla_a U^b = \nabla_a U^b +\Upsilon_a U^b + \Upsilon_c U^c
 \delta^b_{\phantom{b} a},
\end{equation}
where $\Upsilon$ is some smooth section of $T^*M$.

According to the usual conventions in projective geometry, we write
$\ce(1)$ for the positive $(2n+2)$nd root of the bundle $(\Lambda^n
TM)^2$, which we note is canonically oriented.  A connection $\nabla
\in {\bf p}$ determines a connection on $\ce(1)$ as well as its real
powers $\ce(\ulw)$, $\ulw\in \mathbb{R}$; we call $\ce(\ulw)$ the
bundle of projective densities of weight $\ulw$. Conversely, for
$\ulw\neq 0$, a choice of connection on $\ce(\ulw)$ determines a
connection $\nabla\in {\bf p}$. Among the connections in ${\bf p}$
there is a (non-empty) distinguished class consisting of those
connections $\nabla\in {\bf p}$ that preserve some non-vanishing
section of $\ce(\ulw)$, $\ulw\neq 0$ (see e.g.\ \cite{GN}). These are
exactly the special affine connections (defined in Subsection
\ref{prolong}) in ${\bf p}$, and in the following we shall work only
with this subset of connections.  Such a connection $\nabla$ is often
called a {\em choice of scale}; the corresponding section $\tau$ of
$\ce (\ulw)$ (determined up to multiplication by a non-zero constant)
is also often called a choice of scale.  If $\nabla $ and
$\widehat{\nabla}$ are two choices of scale then $\Upsilon_b$ is
exact, meaning $\Upsilon_b=\nabla_b \phi$ for some function $\phi$.

As a point of notation: Given any vector bundle $\mathcal{B}$ we
shall write $\mathcal{B}(\ulw)$ as a shorthand for
$\mathcal{B}\otimes \ce(\ulw)$.

\subsection{The Killing-Yano type projective BGG equation} \label{pKY}
In Section \ref{KYsec} we introduced the Killing-Yano equation
\nn{KYbgg} $ \nabla_{(a}\omega_{b)c}=0 $ on a 2-form $\omega$.  We
want to consider the linear operator giving this equation in the case
that $\nabla$ is a (special, torsion-free) affine connection and
$\omega_{ab}$ is any $2$-form field of weight $w$.  Let
$\widehat{\nabla}$ and $\nabla$ be two projectively equivalent
scales. From \nn{ptrans}, we have
\begin{equation}\nonumber
\begin{split}
\widehat{\nabla}_a\omega_{bc}+\widehat{\nabla}_b\omega_{ac}&=\nabla_a\omega_{bc}+(\ulw-2)\Upsilon_a\omega_{bc}-\Upsilon_b\omega_{ac}-\Upsilon_c\omega_{ba}\\
&+\nabla_b\omega_{ac}+(\ulw-2)\Upsilon_b\omega_{ac}-\Upsilon_a\omega_{bc}-\Upsilon_c\omega_{ab},
\end{split}
\end{equation}
and it is easy to see that
$\widehat{\nabla}_{(a}\omega_{b)c}=\nabla_{(a}\omega_{b)c}$ if and only
if $\ulw=3$.

When taking $\om$ to have projective weight 3,
equation \nn{KYbgg} fits into the class of first BGG equations on projective manifolds; see
\cite{CGHpoly, CGHjlms} and references therein for a general discussion
of the class.

\subsection{The projective tractor connection} \label{trS} On
a general projective manifold $(M,{\bf p})$ there is no canonical
connection on the tangent bundle. There is, however, a canonical
connection on a related natural bundle of rank $n+1$; this so-called tractor
connection is the fundamental invariant object capturing the geometric
structure of projective geometries. We follow here the development of
\cite{BEG,CapGoMac}.

On any smooth manifold $M$, the first jet prolongation $J^1\ce(1)\to
M$ of the projective density bundle $\ce(1)$ of weight $1$, is a natural vector
bundle. Its fiber over $x\in M$ consists of all $1$-jets $j^1_x\sigma$
of local smooth sections $\sigma\in\Gamma(\ce(1))$ defined in a
neighborhood of $x$. For two sections $\sigma$ and $\tilde\sigma$ we
have $j^1_x\sigma=j^1_x\tilde\sigma$ if and only if in one, or
equivalently any, local chart the sections $\sigma$ and $\tilde\sigma$
have the same Taylor development in $x$ up to first order. On the
other hand sections $\si\in \Gamma(\ce(1))$ determine smooth sections
$j^1\si$ of $J^1\ce(1)$ via the smooth structure on the latter space.
Mapping $j^1_x\sigma$ to $\sigma(x)$ thus defines a smooth, surjective
bundle map $J^1\ce(1)\to\ce(1)$, called the \textit{jet
  projection}. If $j^1_x\sigma$ lies in the kernel of this projection,
so $\sigma(x)=0$, then the value $\nabla\sigma(x)\in
T^*_xM\otimes\ce_x(1)$ is the same for all linear connections $\nabla$
on the vector bundle $\ce(1)$. This identifies the kernel of the jet
projection with the bundle $T^*M\otimes\ce(1)$. See for example
\cite{palais} for a general development of jet bundles.

Using an abstract index notation, we shall write $\ce_A$ (in
index-free notation, $\mcT^*$) for $J^1\ce(1)$ and $\ce^A$ (or $\mcT$)
for the dual vector bundle. Then we can view the jet projection as a
canonical section $X^A$ of the bundle
$\ce^A\otimes\ce(1)=\ce^A(1)$. Likewise, the inclusion of the kernel
of this projection can be viewed as a canonical bundle map
$\ce_a(1)\to\ce_A$, which we denote by $Z_A{}^a$. So $\ce_A$ has a
composition structure which is given by the short exact sequence of
bundle maps
\begin{equation}\label{euler}
0\to \ce_a(1)\stackrel{Z_A{}^a}{\to} \ce_A \stackrel{X^A}{\to}\ce(1)\to 0.
\end{equation}
This is known as the jet exact sequence at 1-jets for the bundle
$\ce(1)$.  We write the composition series $\ce_A=\ce(1)\flplus
\ce_a(1)$ to summarize the exact sequence \nn{euler}.  As mentioned, any
connection $\nabla \in {\bf p}$ is equivalent to a connection on
$\ce(1)$. But a connection on $\ce(1)$ is precisely a splitting of the
1-jet sequence \nn{euler}. In particular this holds for special
connections. Thus given such a choice we have the direct sum
decomposition $\ce_A \stackrel{\nabla}{=} \ce(1)\oplus \ce_a(1) $ with
respect to which we define a connection on $\mcE_A$ by
\begin{equation}\label{pconn}
\nabla^{\mathcal{T}^*}_a \binom{\si}{\mu_b}
:= \binom{ \nabla_a \si -\mu_a}{\nabla_a \mu_b + {\sf P}_{ab} \si},
\end{equation}
where, recall, ${\sf P}_{ab}$ is the projective Schouten tensor.
 A simple calculation shows that \nn{pconn} is
independent of the choice $\nabla \in \mbp$, and so
$\nabla^{\mathcal{T}^*}$ is determined canonically by the projective
structure $\mbp$.

This {\em cotractor connection} is due to \cite{Thomas}. It is
equivalent to the normal Cartan connection (of \cite{Cartan}) see
\cite{CapGoTAMS}. We shall term $\ce_A$ ($\mcT^*$) the {\em cotractor bundle},
and we note that the dual {\em tractor bundle} $\ce^A$ ($\mcT$) has composition structure given by the exact sequence
\begin{equation}\label{tractorseq}
0\to \ce(-1)\stackrel{X^A}{\to} \ce^A \stackrel{Z_A{}^a}{\to}\ce^a(-1) \to 0.
\end{equation}
The dual tractor bundle is canonically equipped with the dual {\em
  tractor connection}: In terms of a splitting dual to that above this
is given by
\begin{equation}\label{tconn}
\nabla^\cT_a \left( \begin{array}{c} \nu^b\\
\rho
\end{array}\right) =
\left( \begin{array}{c} \nabla_a\nu^b + \rho \delta^b_{\phantom{b} a}\\
\nabla_a \rho - {\sf P}_{ab}\nu^b
\end{array}\right).
\end{equation}

From \nn{euler} we have invariantly the map $X^A:\ce_A\to \ce(1)$. As
mentioned above, given a special affine connection $\nabla$ on $TM$ we
also have the splitting $\ce_A \stackrel{\nabla}{=} \ce(1)\oplus
\ce_a(1) $, and so in particular the projection $W^A{}_a:\ce_A\to
\ce_a(1)$ that splits the sequence \nn{euler}. By definition then
$Z_A{}^aW^A{}_b=\delta^a_{\phantom{a}b}$. This splitting, and the dual splitting of
the sequence \nn{tractorseq}, are also equivalent to a map $Y_A:
\ce^A\to \ce(-1)$, that satisfies $X^AY_A=1$.  In terms of these,
sections $V^A\in \Gamma(\ce^A)$ and $U_A\in \Gamma(\ce_A)$ which are
represented by
$$
V^A\stackrel{\nabla}{=}\left( \begin{array}{c} \nu^a\\
\rho
\end{array}\right), \quad \mbox{and} \quad U_A\stackrel{\nabla}{=}\binom{\si}{\mu_a}
$$ in the given splitting, can be written $V^A=W^A{}_a\nu^a+ X^A
\rho$, and $U_A= Y_A\si+ Z_A{}^a\mu_a$. These expansions, and the
analogs for tensor powers of the tractor bundles, turn out to be
extremely useful for managing calculations, so we record here some
basic facts.

Under a change of special affine connection from
$\nabla$ to $\widehat{\nabla}$, as in \nn{ptrans}, we have
$$
V^A\stackrel{\widehat\nabla}{=}\left( \begin{array}{c} \nu^a\\
\rho -\Upsilon_a \nu^a
\end{array}\right), \quad \mbox{and} \quad
U_A\stackrel{\widehat\nabla}{=}\binom{\si}{\mu_a+\Upsilon_a \si}
$$
\cite{BEG} (where $\Upsilon_a$ is exact). So for the corresponding
maps $\widehat{Y}_A:\ce^A\to \ce(-1)$ and $\widehat{W}^A{}_a:\ce_A\to
\ce_a(1) $ we have,
\begin{equation}\label{ftrans}
\widehat{W}^A{}_a= W^A{}_a +X^A\Upsilon_a, \quad \widehat{Y}_A= Y_A -Z_A{}^a\Upsilon_a, \quad \widehat{X}^A=X^A, \quad \mbox{and} \quad \widehat{Z}_A{}^a=Z_A{}^a,
\end{equation}
where we have also recorded the projective invariance of $X^A$ and
$Z_A{}^a$ for convenience. Finally the data of the tractor connection
is
captured by how it acts on the splitting maps. From \nn{pconn} and \nn{tconn} we have the following:
 \begin{equation}\label{pconnX}
\begin{split}
\nabla_a X^B = W^B{}_a, &  \quad \nabla_aW^B{}_b= -{\sf P}_{ab} X^B \\
 \nabla_a Y_B= {\sf P}_{ab}Z_B{}^b, & \quad \nabla_a Z_B{}^b =-\delta^b_{\phantom{b} a} Y_B .
\end{split}
\end{equation}
In these formulae we calculate in terms of a scale $\nabla$, and the
connection in the formulae is the coupling of this special affine
connection with the tractor connection $\nabla^\cT$.

Finally in this section we recover the canonical tractor ``volume
form''. First recall that on any smooth manifold $M$ there is a
tautological weighted $n$-form $\boldsymbol{\eta}$ which identifies
$\Lambda^{n}TM$ with a line bundle. In the case that $M$ is orientable
the latter is oriented and 
 $\boldsymbol{\eta}$
gives the isomorphism
\begin{equation}\label{isop}
\boldsymbol{\eta}: \Lambda^{n}TM \to \ce(n+1),
\end{equation}
that defines $\ce(n+1)$.
For each affine connection $\nabla$, the isomorphism \nn{isop}, applied to sections, enables the
definition of $\nabla$ as a connection on $\ce(n+1)$ and hence all
density bundles.
 It follows tautologically that for any affine
connection $\nabla$ we have
$$
\nabla \boldsymbol{\eta} =0 .
$$

In particular the last display applies to special affine connections
in ${\bf p}$.  Calculating in the scale $\nabla$ and using the
formulae \nn{pconnX} it is easily verified that the tractor $(n+1)$-form
$\vol$,
defined by
\begin{equation}\label{pvol}
\vol_{AB\cdots E}^{ } := Y_{[A}Z_B{}^b\cdots Z_{E]}{}^e\boldsymbol{\eta}_{b\cdots e}\in \Gamma(\Lambda^{n+1}\cT^*)
\end{equation}
is parallel for the tractor connection. On the other hand using \nn{ftrans}
it follows at once
 that $\vol$ is independent of the choice of scale $\nabla$. Thus
we have the following (well-known) result:
\begin{proposition}\label{trvol}
An oriented projective $n$-manifold $(M,p)$ determines a canonical
parallel tractor $(n+1)$-form $\vol$ by the formula \nn{pvol}. Hence, the
projective Cartan geometry is of type $(\SL((n+1), \bbR),P)$ for suitable
$P$.
\end{proposition}

\noindent The group $P$ is a parabolic subgroup of $\SL(n+1, \bbR)$ that
may be characterised, up to conjugacy, as the isotropy of some point
on the n-sphere $S^n$, under the standard action of $\SL((n+1, \bbR)$ on
$S^n$.

\subsection{Nearly $\varepsilon$-K\"{a}hler geometry in terms of the tractor connection}\label{KYvstractor}
The tractor connection and its dual induce projectively invariant
connections on all tensor parts
of tensor products of the tractor
bundle, and its dual. In particular we will need the tractor
connection that $\nabla^{\mcT}$ induces (and which we denote by the same symbol) on
\begin{equation}\label{equation:tractor-3-form-components}
\Lambda^3\mathcal{E}_A^{ }=\mathcal{E}_{[ABC]}= \mathcal{E}_{[ab]} (3) \flplus \mathcal{E}_{[a b c]}(3)
\end{equation}
In a choice of scale this is given  by
\begin{equation}\label{tr_connection}
\nabla_a^{\mathcal{T}}\left(\begin{array}{c}
\sigma_{b c}\\
\mu_{b c d}
\end{array}
\right)=\left(\begin{array}{c}
\nabla_a\sigma_{b c}-\mu_{a b c}\\
\nabla_a\mu_{b c d}+3 {\sf P}_{a[ b}\sigma_{cd]}
\end{array}\right).
\end{equation}

We see from \nn{equation:tractor-3-form-components} and \nn{tr_connection} that if $\Phi_{ABC}\in
\Lambda^3\mathcal{E}_A$ is parallel for the tractor connection
$\nabla^\cT$ then its top component $\si_{bc}$ has projective weight 3 and solves the projective
Killing-Yano--type equation $\nabla_{(a}\si_{b)c}=0$, cf.\
\nn{KYbgg} above.  Conversely from Proposition \ref{prolongprop} (and using
again the formula \nn{tr_connection}) we see that if $\om_{bc}$ is a
solution of \nn{KYbgg} then
\begin{equation}\label{tractor_prolong}
0=\widehat{\nabla}_a\left(\begin{array}{c}
\omega_{b c}\\
\mu_{b c d}
\end{array}\right)=
\nabla^{\mathcal{T}}_a\left(\begin{array}{c}
\omega_{b c}\\
\mu_{b c d}
\end{array}\right)
-\frac{1}{2}\left(\begin{array}{c}
0\\
3\omega_{k[b}^{\phantom{k} }W_{cd]\phantom{k}a}^{\phantom{cd]}k}
\end{array}\right),
\end{equation}
where $\mu_{bcd}=\nabla_a\om_{bc}$.

What is important for us here is that for a K\"ahler form $\om$ on a
nearly K\"{a}hler or nearly para-K\"{a}hler 6-manifold the second term
in \nn{tractor_prolong} vanishes separately, and what is more the
tractor 3-form determined by prolonging $\om$ is nondegenerate. More
precisely we have the following result.
\begin{theorem} \label{NPKtoPhi}
Let $(M,g,J)$ be a $6$-dimensional strictly nearly $\varepsilon$-K\"ahler manifold
satisfying $\ab{\nabla J, \nabla J} \neq 0$ and $\mathcal{T}$ the (rank-$7$)
standard projective tractor bundle over $M$. Then the $3$-tractor
\begin{equation}\label{Phi}
\Phi=\left(\begin{array}{c}
\omega_{b c}\\
\mu_{b c d}
\end{array}\right),
\end{equation}
is generic and parallel with respect to $\nabla^{\mathcal{T}}$.  Its
pointwise stabilizer is $\G_2$ if $\Phi$ is definite-generic, or is
$\G_2^{*}$ if $\Phi$ is split-generic.
\end{theorem}
\begin{proof}
As $(M,g,J)$ is strictly nearly $\varepsilon$-K\"ahler, identity
(\ref{weyl_id}) holds, and from Proposition \ref{prolongprop} we have
that $\Phi$ is parallel with respect to $\nabla^{\mathcal{T}}$. In Proposition \ref{stab} of
Subsection \ref{ss: N(P)K dim 6} we observed that $\omega$ and
$\mu = \nabla\omega$ form a stable, compatible, and normalized pair; so, the
hypotheses of Proposition \ref{pairtog2} are satisfied (pointwise),
and hence $\Phi$ is nondegenerate, and pointwise its stabilizer is
$\smash{\G_2^{(*)}}$.
\end{proof}

\begin{remark} The observation below \nn{tr_connection} is a special case of
a general fact that applies across the entire field of parabolic
geometry. A tractor field parallel for a (normal) tractor connection
always determines a solution of a first BGG equation \cite[Theorem
  2.7]{CGH}. BGG operator
solutions arising this way are said to be {\em normal}. (The terminology follows Leitner's \cite{Leitner}, where a
class of conformal equations were treated). Thus part of the content of
Theorem \ref{NPKtoPhi} is that the
K\"ahler form of a strictly N(P)K structure is a normal solution of \nn{KYbgg}.
\end{remark}

\subsection{A digression on conformal tractor geometry} \label{conf-dig}
\newcommand{\cc}{{\bf c}} Shortly we shall see that conformal geometry
enters our picture. Recall that on a manifold $M_0$ a conformal
structure ${\bf c}$ (of signature $(p,q)$) is a conformal equivalence
class $[g]$ of metrics (of signature $(p,q)$) on $M_0$: A metric
$\widehat{g}$ is in the conformal structure $[g]$ if and only if
$\widehat{g}=f g$ for some positive function $f$.  As in projective
geometry, on a conformal manifold there is no invariant connection on
the tangent bundle but again there is on a related higher-rank bundle
that we call the standard (conformal) tractor bundle \cite{BEG}. Again
this is equivalent to a Cartan bundle and a canonical connection
\cite{CapGoTAMS}. Here we sketch some aspects of tractor calculus
following the conventions and development of
\cite{CapGo-ambient,GoPet-GJMS}, and we refer the reader to those
sources and \cite{BEG} for more details, including for the definition
of the conformal tractor bundle and its connection. For later
convenience we denote the underlying manifold $M_0$ and assume here
that this has dimension $n-1\geq 3$.

In conformal geometry density bundles are important. For
representation-theoretic reasons the convention for weights differs
from that in projective geometry: On a conformal manifold of dimension
$n-1$ we write $\ce_0[1]$ for the positive $2(n-1)$st root of the
canonically oriented bundle $(\Lambda^{n - 1} TM_0)^2$. Since a metric
$g\in {\bf c}$ trivializes $(\Lambda^{n - 1} TM_0)^2$, and hence also
$\ce_0[1]$, $(M_0,{\bf c})$ determines a canonical section
$\boldsymbol{g} \in S^2T^*M_0[2]=S^2T^*M_0 \otimes \ce_0[2]$ called
the {\em conformal metric}; then, on the fixed conformal structure
$\cc$, $g$ is equivalent to $\si\in \Gamma((\ce_0)_+[1])$ by the
relation $g=\si^{-2}\bg$. (Here $(\ce_0)_+[1]$ is the positive ray
sub-bundle of $\ce_0[1]$.) The conformal metric and its inverse are
preserved by the Levi-Civita connection of every metric in the
conformal class, and they determine an isomorphism $\bg: TM_0\to
T^*M_0[2]$.

The conformal standard tractor bundle will be denoted $\cT_{0}$, or
in abstract index notation $\ce^A_{0}$. It has rank $n+1$ and a
composition series
\begin{equation}\label{cT-comp}
\cT_{0} = \ce_0[1]\flplus TM_0[-1]\flplus \ce_0[-1] .
\end{equation}
On $\cT_{0}$, there is an invariant signature $(p+1,q+1)$ tractor
metric $H_0$ (which we may alternatively denote $H^0$),
and an invariant connection
$\nabla^{\cT_0}$ that preserves the metric; so, we use it to
lower and raise tractor indices. The canonical bundle line bundle
injection is denoted
\begin{equation}\label{Xdef}
X^A:\ce_0[-1]\to \ce^A_0 .
\end{equation}
No confusion should arise with the projective analogue of $X^A$, which
shares the same abstract index notation. In fact, we shall see that
they are suitably compatible in the setting below where they arise
together.  As a section of $\cT_0[1]$, $X$ is null, meaning
$H^0_{AB} X^A X^B = 0$, and $X_A := H^0_{AB}X^B$ gives the canonical bundle map
$X_A:\ce_0^A\to \ce_0[1]$.

A choice of metric $g\in \cc$ determines a splitting of \nn{cT-comp}
to a direct sum $\cT_0 \stackrel{g}{=} \ce_0[1]\oplus TM_0[-1]\oplus \ce_0[-1]$ and we
denote the induced projections onto the second and third components by
\begin{equation}\label{ZYdef}
Z_A{}^a: \ce_0^A\to \ce_0^a[-1], \quad Y_A:\ce_0^A\to \ce_0[-1] .
\end{equation}
We may view these as sections $Y_A\in \Gamma((\ce_0)_A[-1])$, $Z_A{}^a\in
(\ce_0)_A{}^a[-1]$ and then the tractor metric is characterized by
the identities $H^0_{AB}Z^A{}_aZ^B{}_b=\bg_{ab}$ and $H^0_{AB}X^AY^B=1$ and that all other
tractor index contractions of pairs of these splitting maps results in
zero; for example $H^0_{AB}Y^AY^B=0$. (Note that here we have raised and
lowered indices on $Z$ and $Y$ using the conventions described.) A section
$U^A\in \Gamma ((\ce_0)^A)$ may be written $U^A=\si Y^A + Z^A{}_a\mu^a + \rho X^A$, with $(\si, \mu^a, \rho)\in \Gamma(\ce_0[1]\oplus TM_0[-1]\oplus \ce_0[-1] )$.
If $\widehat{Y}^A$ and $ \widehat{Z}^A{}_b$ are the corresponding tractor splitting maps in terms of the metric $ \hat{g}=\Omega^2g\in \boldsymbol{c}$
then we have
\begin{equation}\label{Ztrans}
\textstyle
\begin{array}{ccc}
\widehat Z^{Ab}=Z^{Ab}+\Upsilon^bX^A, &
\widehat Y^A=Y^A-\Upsilon_bZ^{Ab}-\frac12\Upsilon_b\Upsilon^bX^A,& \widehat{X}^A=X^A ,
\end{array}
\end{equation}
where $\Upsilon=d \Omega$, and we have recorded the invariance of $X^A$ for convenience.

\newcommand{\V}{{\mbox{\sf P}}}
The conformal tractor connection is characterized by its action on the
splitting maps:
\begin{equation}\label{connids}
\begin{array}{rcl}
\nd_aX_A=Z_{Aa}\,, &
\nd_aZ_{Ab}=-\V_{ab}X_A-\bg_{ab}Y_A\,, & \nd_aY_A=\V_{ab}Z_A{}^b .
\end{array}
\end{equation}
Here $\V_{ab}$ is the conformal Schouten tensor, defined in Section \ref{ss: N(P)K dim 6}, and given by
\begin{equation}\label{equation:conformal-schouten}
    \V_{a b}=\frac{1}{n-3}\left(R_{a b} -\frac{\rm Sc}{2(n-2)}g_{a b}\right),
\end{equation}
where $R_{a b}$ is the usual (pseudo-)Riemannian Ricci tensor and ${\rm Sc}$ is
its metric trace. In \nn{connids} the connection used is strictly the
coupling of the tractor connection $\nabla^{\cT_0}$ with the
Levi-Civita connection, hence the use of the notation $\nabla$ (rather
than $\nabla^{\cT_0}$).

\begin{remark}
As we mentioned is the case for $X$, also the notation
$Y$, $Z$ for the objects splitting the
conformal tractor (via $g\in \cc$) is essentially the same as that
used for the corresponding objects in the projective setting. Context should prevent any confusion, and indeed there is again a degree of compatibility.
\end{remark}

\subsection{Projective almost Einstein structures}\label{pEin}

It will be important for us to understand the meaning of a
non--Ricci-flat Einstein metric in the setting of projective
geometry, following \cite{ArmstrongP1,CapGoMac}. In fact projective
geometry motivates a natural generalization of the Einstein condition
\cite{CGHjlms,CGH,CapGoMac} and this is a key point for us here.

From Theorem 3.3 of \cite{CapGoMac} we have the following.
\begin{theorem}\label{eintoH} Let $(M,g)$ be a pseudo-Riemannian
manifold of signature $(p,q)$, and dimension $n\geq 2$.  If $g$ is
positive (respectively, negative) Einstein and not Ricci-flat then
there is a canonical parallel projective tractor metric $H$ of
signature $(p+1,q)$ (respectively $(p,q+1)$) on the projective
structure $(M,[\nabla^g])$.
\end{theorem}
\noindent Here $\nabla^g$ is the Levi-Civita connection determined by
$g$. In the case of dimension 2 we take $g$ Einstein to mean that
$\operatorname{Sc}^g$ is constant, although this dimension is not
important for the current article.

In the converse direction, if a projective manifold admits a parallel metric $H$ on its
tractor bundle it determines a section
\begin{equation}\label{taudef}
\tau:=H_{AB}X^A X^B
\end{equation}
of the density bundle $\ce(2)$. On any open set where $\tau$ is nowhere
zero it  may be used to trivialize the density bundles, and
hence it determines a connection on densities and so a splitting of
the sequence \nn{euler}. Taking duals
we obtain a splitting of the tractor
sequence \nn{tractorseq}, and so from $H$ a metric $g^\tau$ on $TM$.
In the (nonvanishing) scale $\tau$, and with the corresponding
splittings, we have
\begin{equation}\label{Hform}
H_{AB}=\left(\begin{array}{c c}
\tau &  0 \\
0 &  -\varepsilon \tau g^\tau_{ab}
\end{array}\right),
\end{equation}
where  $-\varepsilon \in \set{\pm 1}$ gives the sign of the
 scalar curvature.
This formula follows easily from the formula \nn{pconn} for the
tractor connection (extended to $S^2\cT^*$) and the definition here of
the metric $g$ via Theorem \ref{HtoEin}. See e.g.\ \cite[Section
  3.3]{CapGo-projC} for a more detailed discussion, and expression
(15) in that source for the sign of the
 scalar curvature.

It turns out that  $g_\tau$  is necessarily Einstein. In detail we have the
following, as obtained from different points of view in \cite[Theorem
  3.2]{CGHjlms} and Theorem 3.1 and Proposition 3.2 in
\cite{CGH}.
\begin{theorem}\label{HtoEin}
Let $(M,{\bf p})$ be a projective structure endowed with a holonomy
  reduction given by a parallel metric $H$ of
  signature $(r,s)$ on the standard projective tractor bundle $\Cal T$.

    \begin{enumerate}
        \item The metric $H$ determines a stratification $M=M_+\cup
          M_0\cup M_-$ according to the strict sign of
          $\tau:=H_{AB}X^A X^B$.  The sets $M_+\subset M$ and
          $M_-\subset M$, where $\tau$ is positive and negative,
          respectively, are open; $M_0$ is the zero set of $\tau$ and
          (if non--empty) is a smoothly embedded separating hypersurface
          consisting of boundary points of both $M_+$ and $M_-$. Here
          $M_+$, $M_0$, and $M_-$ are not necessarily connected.

        \item The structure $(M,{\bf p},H)$ induces a Cartan geometry on $M_+$
  (respectively $M_-$) as follows: Via \nn{euler} and \nn{tractorseq}, $H$
  induces an Einstein pseudo-Riemannian metric $g_\pm$ of signature
  $(r-1,s)$, if $r\geq 1$ (respectively $(r,s-1)$ if $s\geq 1$) whose
  Levi-Civita connection lies in the (restriction of) projective
  class. The scalar curvature of $g_{\pm}$ is positive (respectively negative).

        \item If $r=0$ (or $s=0$) then $M_0=\emptyset$, and $M_+=\emptyset$ (or
  $M_-=\emptyset$, respectively). If $M_0$ is non-empty then it
  naturally inherits a conformal structure of signature $(r-1,s-1)$
  via the induced Cartan geometry. In this case the standard conformal
  tractor bundle agrees with the restriction of the projective tractor
  bundle $\cT$ to $M_0$ and the normal conformal tractor connection of
  $(M_0,\mbc)$ is naturally the corresponding restriction of the ambient
  projective tractor connection.
\end{enumerate}
\end{theorem}

\noindent The components $M_+$, $M_0$, $M_-$ are called {\em curved
  orbits} since they generalize to the curved setting an orbit
decomposition of a model structure, as explained in
\cite{CGHjlms,CGH}.  Since $(M,{\bf p},H)$ determines on the open sets
$M_{\pm}$ (whose union is dense) Einstein metrics, it is natural to
call this a projective almost Einstein structure (following the
analogous conformal notion \cite{Gal}).  In the case where $M_0$ is
nonempty we shall term the parts $M_{\pm}\cup M_0$ {\em Klein-Einstein
  manifolds} (or {\em Klein-Einstein structures}). This follows
\cite[Section 3.3]{CGHjlms} and as explained there the terminology is
appropriate and useful because these are the projective geometry
analogues of Poincar\'e-Einstein geometries; see also \cite[\S4]{FeffermanGraham}
for more about both Klein- and Poincar\'{e}-Einstein metrics, and
about the relationship between them.  (In \cite{CGHjlms} this
terminology was proposed for the negative curvature part, so we are
slightly generalizing its usage here.) Concerning terminology, in the
subsequent discussion it will be convenient to refer to $M_0$, which
is the zero locus of $\tau$, as the \textit{zero locus} of $(M, \mbp,
H)$ or simply {\em the zero locus} when the meaning is clear by
context.

From part (b) of the Theorem we have that the projective class of the
Levi-Civita connections $\nabla^{g_{\pm}}$ in $M_{\pm}$ extend
smoothly to $M_0$. In fact more is true, as follows.

 First recall
that, in a manifold, a defining function for a codimension-1 embedded
submanifold $\Sigma$ is a function $r$ such that $\Sigma$ is the zero
locus of $r$, and $dr$ is nowhere zero along $\Sigma$. Following \cite{CapGo-projC} we make the following definition.

\begin{definition} \label{pc-def}
On a
manifold $M$ with boundary $\partial M$ and interior $M_{int}$ an
affine connection $\nabla$ on $M_{int}$ is called \textit{projectively
  compact} of order $\al\in \mathbb{R}_+$ if for any $x\in\partial M$,
there is a neighborhood $U$ of $x$ in $M$ and a defining
function  $r:U\to\Bbb R$ for $U \cap \partial M$ such
that the connection
\begin{equation}\label{pceq}
\hat\nabla=\nabla+\tfrac{d r}{\alpha r}
\end{equation}
on $U\cap M_{int}$ extends to all of $U$. A metric is said to be
projectively compact of order $\alpha$ if its Levi-Civita connection
satisfies this condition.
\end{definition}

This notion applies in an obvious way to the setting of the Theorem
\ref{HtoEin} above, by considering the manifolds with boundary
$M_{\pm} \cup M_0$. Then from Theorem 12 of
\cite{CapGo-projC} we have that a non--Ricci-flat projectively compact
metric must have $\alpha =2$, and hence the following result.
\begin{proposition}\label{projComp}
The metrics $g_{\pm}$ in Theorem \ref{HtoEin} are projectively compact
of order 2.
\end{proposition}
\noindent For special affine connections the behavior \nn{pceq} guarantees
a uniformity in the rate of asymptotic volume growth as the boundary
is approached. The value $\alpha=2$ shows, for example, that this
growth rate is different from that on conformally compact manifolds,
cf.\ \cite[Section 2.2]{CapGo-projC}.

For later (implicit) use, we record a convenient alignment of the
projective and conformal conventions for weighted bundles.
\begin{proposition}\label{wtsalign}
Let $(M, \mbp)$ be a projective manifold of dimension $n \geq
2$ and $H$ a parallel tractor metric for which the zero locus
$M_0$ is nonempty. Then, using the notation in Subsection
\ref{conf-dig} for the objects on $M$, for all real $k$ there is a
canonical identification
\[
    \mcE_0[w]  \cong  \mcE(w) \vert_{M_0}    \textrm{.}
\]
\end{proposition}
\begin{proof}
Recall that on an oriented conformal $(n-1)$-manifold the conformal
density bundle of weight $\ulw$, denoted $\ce_0[\ulw]$, is the positive
$\ulw/(n-1)$ root of the oriented line bundle $\Lambda^{n-1}TM_0$. We
will show that $\Lambda^{n-1}TM_0$ can be identified with
$\ce(n-1)|_{M_0}$.

Recall that
on an oriented projective manifold $(M,{\bf p})$ of dimension $n$
there is a tautological weighted $n$-form $\boldsymbol{\eta}$ which gives the isomorphism
$$
\boldsymbol{\eta}: \Lambda^{n}TM \to \ce(n+1),
$$
that defines $\ce(n+1)$, see \nn{isop}. On the other hand along $M_0$, there is a canonical section
$\boldsymbol{n}\in \mathcal{N}\otimes \ce(2)$, where $\mathcal{N}$ is
the conormal bundle. If $\nabla\in {\bf p}$, then
$\boldsymbol{n}=(\nabla \tau)|_{M_0}$, but $\boldsymbol{n}$ is independent
of the choice of $\nabla$, since $M_0$ is the zero locus of $\tau= H(X,X)$.

Now let $\boldsymbol{n}^\sharp$ be any section of $(TM\otimes
\ce(-2))|_{M_0}$ satisfying
$\boldsymbol{n}(\boldsymbol{n}^\sharp)=1$. Then, identifying
$\Lambda^{n-1}T M_0$ with its image in $\Lambda^{n-1} TM_0 |_{M_0}$, we
obtain a surjective bundle map
\begin{equation}\label{canmap}
\boldsymbol{\eta}(\boldsymbol{n}^\sharp,\,\cdot\,,\cdots,\,\cdot\,): \Lambda^{n-1}T M_0\to \ce(n-1)|_{M_0} .
\end{equation}
Now it is easily verified that \nn{canmap} is independent of the choice of  $\boldsymbol{n}^\sharp\in \Gamma((TM\otimes
\ce(-2))|_{M_0} )$ satisfying $\boldsymbol{n}(\boldsymbol{n}^\sharp)=1$. Thus \nn{canmap} gives a canonical isomorphism
$$
\ce_0[n-1]\to \ce(n-1)|_{M_0}
$$
and thus we obtain $ \mcE_0[\ulw]\cong \mcE(\ulw) \vert_{M_0}$, for all real weights $\ulw$.
 \end{proof}

\section{Projective $6$-manifolds with a parallel tractor split cross product}
\label{mobileO}

We now want to consider a situation essentially converse to that of
Theorem \ref{NPKtoPhi}.  Namely we consider a projective $6$-manifold
$(M,{\bf p})$ that is equipped with a parallel tractor 3-form $\Phi$
that is {\em generic}, meaning that at one, equivalently any, point
$x\in M$, $\Phi_x$ is stable on $\cT_x$ (in the sense of Section
\ref{ss: stable forms}).
 So the projective Cartan/tractor connection
admits a holonomy reduction to either $\G_2$ or $\G_2^*$. We will see
shortly that this is the same as a parallel cross product on the
tractor bundle.

\subsection{The tractor metric on an $(M,p,\Phi)$ manifold}\label{Phi-geom}
Recall that we assume $(M,{\bf p})$ is orientable. If $M$ has
dimension 6 then by Proposition \ref{trvol} the tractor connection
\nn{tconn} preserves a (non-trivial) parallel tractor $7$-form $\vol$, which for
convenience we shall call the {\em tractor volume form}, although of
course $\vol$ is not a tensor. We write $\vol^{-1}$, or $\vol^{A_1 \cdots A_7}$, for the section of $\Lambda^7\cT$ satisfying
$$
\vol^{A_1 \cdots A_7}\vol_{A_7 \cdots A_7}=7!~.
$$

The first observation
is that with $\Phi$ this object determines a tractor metric.
\begin{theorem} \label{Hdef}
On a projective 6-manifold let $\Phi$ be a parallel generic tractor 3-form. Then $\Phi$ determines a nondegenerate parallel tractor
\begin{equation}\label{HHdef}
H_{AB} := \tfrac{1}{144}\Phi_{A C_1 C_2}\Phi_{B C_3 C_4}\Phi_{C_5 C_6 C_7}\vol^{C_1\ldots C_7}
\end{equation}
of signature either $(7,0)$ or $(3,4)$.
\end{theorem}
\begin{proof} As $\Phi$ and $\vol$ are parallel, this is immediate from
the algebraic result \nn{H-from-Phi-vol}.
\end{proof}

\begin{remark}\label{sig-remark}
The signatures $(7,0)$ and $(3,4)$ are compatible (in a way that \ e.g.\ $(7,0)$ and $(4,3)$ are not) in the sense that we take the inner product on both $\bbO$ and $\bbO^{*}$ in \S2 so that the real line $\bbR \subset \bbO^{(*)}$ is positive definite.
\end{remark}

Thus the $\G_2$ holonomy reduction of the projective Cartan bundle is
subordinate to a $\SO(7,0)$ reduction, and similarly the $\G_2^*$
reduction is subordinate to a $\SO(3,4)$ reduction. Furthermore we can
at once exploit Theorem \ref{HtoEin} and Proposition \ref{projComp}.
\begin{corollary}\label{firststrat} Suppose that $(M,{\bf p})$
is a projective 6-manifold equipped with a generic parallel 3-form
tractor $\Phi$.  Then:
\begin{itemize}
    \item If $\Phi$ is definite-generic then it determines a signature $(6,0)$ positive Einstein metric on $M$.
    \item If $\Phi$ is split-generic then it determines a decomposition $M = M_+ \cup M_0 \cup M_-$ of $M$ into a union of 3 (not necessarily connected) disjoint curved orbits; $M_{\pm}$ are open while the zero locus $M_0$ is closed. If the zero locus $M_0$ is non-empty, then both $M_+$ and $M_-$ are non-empty. Conversely, if $M$ is connected and both $M_{+}$ and $M_{-}$ are non-empty then $M_0$ is non-empty and is a smoothly embedded separating hypersurface consisting of boundary points of both $M_+$ and $M_-$. Furthermore, $M_0$ has a conformal structure of signature $(2,3)$ with normal conformal tractor connection in agreement with the pullback of the ambient projective tractor connection.  $M_+$ has canonically a positive Einstein metric $g_+$ of signature $(2,4)$, while $M_-$ has canonically a negative Einstein metric $g_-$ of signature $(3,3)$. The metrics $g_\pm$ are projectively compact of order 2.
\end{itemize}
\end{corollary}
\begin{remark}In principle it is a little misleading at this point to describe
$M_+$, $M_0$, $M_-$ as curved orbits in the spirit of
\cite{CGH}, as at this stage we are only using the data of the
holonomy reduction to $\SO(H)$, rather than the full reduction to $\smash{\G_2^{(*)}}$. This choice of language will be justified, however,
in Section \ref{Ptypes}.
\end{remark}

\subsection{The cross product and the N(P)K-structure on $M_\pm$}\label{cross-section}
To obtain further details concerning the geometry of the $(M,{\bf
  p},\Phi)$-structure it is useful to first introduce the product
structure alluded to above.

The tractor metric $H$, determined by $\Phi$ in Theorem \ref{Hdef},
may be used to identify the tractor bundle $\ce^A$ with
its dual $\ce_A$; since $H$ is parallel, raising and lowering of tractor
indices with it commutes with covariant differentiation. Thus from a generic
parallel tractor 3-form $\Phi$ we obtain $\Phi^A{}_{BC}\in
\Gamma(\ce^A{}_{BC})$, which may be interpreted as the parallel tractor
cross product
\begin{equation}\label{cp}
\btimes : \cT\times\cT\to \cT, \qquad (U^B, V^C)
\mapsto \Phi^A{}_{BC}U^B V^C .
\end{equation}
Conversely from the parallel cross product we may recover $\Phi$ and
$H$ via Proposition \ref{equivP} and \nn{HHdef} (or \nn{PtoH}).  We will say
the tractor cross product $\btimes$ is \textit{definite} or \textit{split} if and only if the corresponding tractor $3$-form $\Phi$ is definite-generic or split-generic, respectively.  In summary we have the following.
\begin{proposition}\label{crossprod}
A projective $6$-manifold $(M,{\bf p})$ equipped with a parallel
generic tractor 3-form $\Phi$ has a uniquely determined parallel cross
product $\btimes$, as given by \nn{cp}. Conversely given a parallel
definite (split) generic cross product on the projective tractor
bundle we obtain a unique parallel definite (respectively, split) generic tractor 3-form, and these constructions are inverses.
 \end{proposition}
Thus the structure $(M, {\bf p},\Phi)$ may equivalently be viewed as a triple
$(M,{\bf p}, \btimes)$. More precisely, each fiber $\cT_x$ of the
tractor bundle canonically carries the algebraic structure of the
imaginary octonions $\bbI^{(*)}$, and furthermore we can meaningfully
say that this imaginary octonion structure is parallel for the tractor
connection.

This means that the algebraic results from Section \ref{alg} transfer
effectively and uniformly into the tractor calculus.  As a first
application note that by the formula \nn{equation:cross-to-volume} we have
a parallel tractor 7-form
\begin{equation}\label{PhitoVol}
 \vol_{ABCDEFG} := \tfrac{1}{42} \Phi_{K[AB} \Phi^K_{\phantom{K}CD} \Phi_{EFG]} .
\end{equation}
We henceforth assume that $\Phi$ is normalized so that this agrees
with the canonical projective tractor volume form of Proposition \ref{trvol}.

Next from \nn{equation:J-squared} of
Proposition \ref{proposition:J-squared} we have the next result.
\begin{proposition}\label{Malid}
The tractor cross product satisfies
$$
U\btimes (U\btimes V) = -H(U, U) V + H(U, V) U,
$$
for all $U,V\in \Gamma(\cT)$.
\end{proposition}

The cross product determines a canonical map
\begin{equation}\label{Jdef}
\bJ: \cT\to \cT(1) \quad \mbox{defined by} \quad V\mapsto -X \btimes V ,
\end{equation}
where $X\in \Gamma(\ce^A(1))$ is the canonical weighted tractor from
\nn{euler} (cf.\ \nn{aJdef}).  An easy calculation verifies that $\bJ$
is not parallel. In fact $\bJ$ is a section of the non-trivially
weighted tractor bundle $\End (\cT)(1)$, whereas invariantly parallel
tractor fields must have weight zero. Despite this weight, it is
useful to view $\bJ$ as essentially an endomorphism of the tractor
bundle for reasons at which the next Proposition hints.
\begin{proposition}\label{proposition:J-properties}
Let $(M, {\bf p}, \Phi)$ be a 6-dimensional projective manifold equipped with a
parallel tractor cross product $\btimes$ (or equivalently, a generic parallel tractor 3-form $\Phi$). Then
\begin{equation}\label{JX}
\bJ^{A}{}_{B}X^B=0
\end{equation}
and
\begin{equation}\label{bJ2}
\mathbb{J}_{\phantom{A}C}^{A}\mathbb{J}_{\phantom{C}B}^{C}=-\tau \delta_{\phantom{A}B}^{A}+X^A X_B^{ } \quad \mbox{where} \quad \tau:=H_{AB}X^A X^B .
\end{equation}
\end{proposition}
\begin{proof} This follows at once from the algebraic Proposition \ref{proposition:J-squared}.
Alternatively the results can be seen in terms the discussion
immediately above as follows: As a bilinear tractor form $\btimes $ is
skew by construction, since $\Phi\in \Gamma(\Lambda^3 \cT^*)$. Thus
\nn{JX} is immediate from the definition \nn{Jdef} of $\bJ$. On the other hand \nn{bJ2} follows at once from Proposition \nn{Malid}.
\end{proof}

The result \eqref{JX} and the transformation rule \eqref{ftrans}
together imply that the weighted endomorphism component
\[
    J^a_{\phantom{a} b} := \bbJ^A_{\phantom{A} B} Z_A^{\phantom{A} a} W^B_{\phantom{B} b} \in \Gamma(\End(TM)(1))
\]
of $\bbJ$ is invariant, that is, it is independent of the choice of scale. Together \eqref{bJ2} and the
identity $\delta^A_{\phantom{A} B} = \delta^a_{\phantom{a} b}
W^A_{\phantom{A} a} Z_B^{\phantom{B} b} + X^A Y_B^{ }$ give that
\begin{equation}\label{equation:J-squared-base}
    J^a_{\phantom{a} c} J^c_{\phantom{c} b} = -\tau \delta^a_{\phantom{a} b} \textrm{.}
\end{equation}

On any open set $U$ on which $\pm \tau \in \mcE(2)$ is positive, we can trivialize density bundles using $\pm \tau$, and in particular we get a canonical unweighted endomorphism
\[
    J_{\pm} := (\pm \tau)^{- 1 / 2} J \vert_U \in \End(TU) \textrm{;}
\]
formally, this trivialization has the effect of ``setting $\tau$ to $\pm 1$''. Substituting gives that the unweighted endomorphism $J_{\pm}$ satisfies $J_{\pm}^2 = \mp \id_{TM}\vert_U$.

Recall from Subsection \ref{pEin} that the tractor metric $H$ determined by $\Phi$ via \eqref{HHdef} also determines a metric $g_{\pm}$ on $U$. Now, $\Phi(X, \pdot, \pdot) = -\Phi(\pdot, X , \pdot) = H(\pdot, \bbJ \pdot) $, so on $U$ contracting both sides with $W^A_{\phantom{A} a} W^B_{\phantom{B} b}$ and trivializing the involved density bundles gives that the top slot $\omega$ (now regarded via the trivialization as an unweighted $2$-form) of $\Phi$ satisfies
\[
    \om\vert_U = g_{\pm}(\pdot, J_{\pm} \pdot) \textrm{.}
\]
In particular, $J_{\pm}$ is $g_{\pm}$-skew, so $(g_{\pm}, J_{\pm})$ is a $(\mp 1)$-Hermitian structure (in particular it implies in the $-$ case that the $(\pm 1)$-eigenspaces of $J_-$ have the same dimension, that is, that $J_-$ is a paracomplex structure).

So, via $H$ and $\bJ$, the parallel tractor cross product $\btimes$ (equivalently $\Phi$) determines almost $(\mp 1)$-Hermitian structures $(g_{\pm}, J_{\pm})$ on the sets where $\pm \tau$ is positive, and these turn out moreover to be strictly nearly $(\mp 1)$-K\"{a}hler structures on those sets.

\begin{theorem}\label{local}
Let $(M, {\bf p})$ be a $6$-dimensional projective manifold equipped with a parallel tractor cross product $\btimes$ (equivalently a generic, parallel tractor 3-form $\Phi$). Then on any open set $U$ where $\tau$ is positive (respectively, negative), $\Phi$ defines a strictly nearly K\"ahler (resp.\ strictly nearly para-K\"ahler) structure $(g_{\pm}, J_{\pm})$ in the projective class ${\bf p}\vert_U$ on $U$, that is, for which $\nabla^{g_{\pm}} \in \mbp \vert_U$.\footnote{Here, $\mbp \vert_U$ is the projective structure on $U$ containing all of the restrictions $\nabla \vert_U$ of connections $\nabla \in \mbp$.}
\end{theorem}

\begin{proof}
For simplicity of notation, we may as well replace $M$ with the given open set.

Recall from the discussion before the theorem that $\Phi$ determines an $(\mp 1)$-Hermitian structure $(g_{\pm}, J_{\pm})$ on $M$. Decompose $\Phi$ with respect to any splitting as in \eqref{equation:tractor-3-form-components}. Since $\nabla^{\mcT} \Phi = 0$, \eqref{tr_connection} implies that $\mu = \nabla \omega$, and a fortiori that $\nabla \omega$ is totally skew, so $(g_{\pm}, J_{\pm})$ is nearly $(\mp 1)$-K\"{a}hler.

In the scale determined by $\tau$, $\nabla \omega = \mu$ is nowhere zero, as it is easily seen that if it were zero at any point then $H$ as defined by \nn{HHdef} would be degenerate there, which would be a contradiction. Thus, $(g_{\pm}, J_{\pm})$ is strictly nearly $(\mp 1)$-K\"{a}hler.
\end{proof}

\begin{remark}
The $\varepsilon$-K\"ahler structures in the Theorem are necessarily
Einstein by dint of being strict and in dimension 6, cf.\ Section
\ref{ss: N(P)K dim 6}. It is useful to see that, in our current context,
this follows from Theorem \ref{HtoEin} as Corollary \ref{firststrat}.
\end{remark}

\begin{remark}
We observed earlier that $\bJ$ is not parallel. Indeed the algebraic
relationship between $\btimes$ (which is parallel) and $X$ (which is
far from parallel) varies across the manifold. When $\btimes$ is
split-generic, it is precisely this relationship which enables the
 single ``tractor
endomorphism'' field $\bJ$ to deliver a nearly K\"{a}hler structure on one
part of the manifold, a nearly para-K\"{a}hler structure on another
part, and yet a different structure on the separating hypersurface. We
now turn our attention to the latter.
  \end{remark}

\subsection{The zero locus $M_0$} \label{inf}

In this subsection, we analyze the structure a parallel split cross product (split-generic
tractor $3$-form) $\Phi$ induces on and along the hypersurface zero locus $M_0$, which in this section we assume is nonempty. (Recall that if a parallel tractor cross product $\btimes$ is definite, the induced tractor metric is definite, and the zero locus is empty.) The geometry
on $M_0$ itself is intrinsically interesting, but the common source of
the induced geometries on $M_{\pm}$ and $M_0$---namely the holonomy
reduction itself---establishes a close relationship between the
N(P)K structures on $M_{\pm}$ along $M_0$ and the geometry on $M_0$ itself.  Indeed, Corollary \ref{firststrat} already shows that via projective geometry we may view the geometry on $M_0$ as a (simultaneous) limit structure at infinity of the geometries on $M_{\pm}$. Later, in Subsection \ref{cpctsec}, we exploit this relationship to formulate a suitable notion of compactification for N(P)K Klein-Einstein metrics.

By the third part of that corollary, the $\SO(3, 4)$ holonomy
reduction, to which the reduction to $\G_2^*$ is subordinate, determines
a normal parabolic geometry $(\mcG_0 \to M_0, \eta_0)$ of type
$(\SO(3, 4), P_0)$ on $M_0$, which corresponds to a conformal
structure $\mbc$ of signature $(2, 3)$ there. (This is a rephrasing of
the statement there that the parallel tractor metric determines on
$M_0$ a conformal structure of signature $(2,3)$ with normal conformal
tractor connection in agreement with the pullback of the ambient
projective tractor connection.)  We may identify the standard
conformal tractor bundle $\mcT_0 := \mcG_0 \times_{P_0} \bbV$ (where
$\bbV$ denotes the standard representation of $\SO(3, 4)$) with the
restriction $\mcT\vert_{M_0}$ of the projective tractor bundle
$\mcT$. Then the normal connection $\nabla$ on $\mcT$ restricts to a
connection $\nabla_0$ on $\mcT_0$, and the latter coincides with the
normal tractor connection induced by $\eta_0$. In particular the
$\nabla$-parallel cross product $\btimes$ that defines the $\G_2^*$
holonomy reduction of $\nabla$ restricts to give a parallel cross
product on $M_0$, and so in summary we have the following result.
\begin{theorem}\label{confRed}
Suppose a $6$-dimensional projective manifold $(M, \mbp)$ admits a parallel split tractor cross product for which the zero locus $M_0$ is nonempty. Then, the conformal structure $(M_0,{\bf c})$ defined above canonically admits a
$\nabla_0$-parallel split-generic cross product $\btimes: \mcT_0
\times \mcT_0 \to \mcT_0$, and hence it defines a $\G_2^*$ holonomy
reduction of $\nabla_0$.
 \end{theorem}
The geometric meaning of a holonomy reduction to $\G_2^*$ of a normal
conformal tractor connection has been analyzed previously: It
corresponds to the existence of an underlying, canonically associated (oriented)
$(2, 3, 5)$-distribution, see \cite{HammerlSagerschnig,
  NurowskiDifferential}. One direction of this correspondence is as follows.

\begin{theorem}\cite[\S5.3]{NurowskiDifferential}\label{theorem:Nurowski}
Any $(2, 3, 5)$-distribution $D$ on a $5$-manifold $M_0$ canonically induces a conformal structure $\mbc_D$ on $M_0$.
\end{theorem}

This is a relatively simple example of a so-called Fefferman
construction and, as outlined in the proof sketch below, it can be
framed efficiently in the language of parabolic geometry. To explain
this, we describe briefly the realization of $(2, 3, 5)$-distributions
in this context, cf.\ \cite{CapSlovak}.

First, recall from Subsection
\ref{subsection:octonions-cross-product-g2} that the $\G_2^*$-action
on $\bbI^*$ preserves a signature-$(3, 4)$ inner product $\cdot$ (and
orientation) there and hence determines a canonical inclusion $\G_2^*
\hookrightarrow \SO(\bbI^*) \cong \SO(3, 4)$. An (oriented) $(2, 3, 5)$-distribution is precisely the structure
underlying a parabolic geometry of type $(\G_2^*, Q)$, where $Q$ is
the stabilizer in $\G_2^*$ of a null ray in $\bbI^*$. By construction,
$Q = P_0 \cap \G_2^*$, where $P_0$ is the stabilizer in $\SO(3, 4)$ of
a null ray. As a parabolic subgroup, $Q$ determines a $\bbZ$-grading
$(\mfg_a)$ on the Lie algebra $\mfg := \mfg_2^*$ for which\footnote{Of
  course, the reader should not confuse the graded component $\mfg_2$
  with the compact, real Lie algebra for which we use the same symbol,
  or the asterisk $^*$ denoting the split real form with an indication
  of a vector space dual.}
\begin{enumerate}
    \item the grading respects the Lie bracket in that $[\mfg_a, \mfg_b] \subseteq \mfg_{a + b}$ for all $a, b \in \bbZ$,
    \item $\mfg_a \neq \set{0}$ if and only if $|a| \leq k$ for some positive integer $k$ (for this particular parabolic subgroup, $k = 3$),
    \item $\mfg_{-1}$ generates (under the bracket operation) the subalgebra $\mfg_{-k} \oplus \cdots \oplus \mfg_{-1} < \mfg$, and
    \item the Lie algebra $\mfq$ of $Q$ satisfies $\mfq = \mfg_0 \oplus \cdots \oplus \mfg_k$.
\end{enumerate}
The grading determines a natural filtration
\[
    \mfg^a := \mfg_a \oplus \cdots \oplus \mfg_k \textrm{.}
\]
of $\mfg$; by definition $\mfq = \mfg^0$, and by construction the adjoint action of $\mfq$ on $\mfg$ preserves this filtration (but not the underlying grading).

Now, given a normal, regular parabolic geometry $(\mcG^{\G_2^*} \to
M_0, \eta^{\G_2^*})$ of type $(\G_2^*, Q)$ \footnote{Here, normality
  and regularity are normalization conditions that together ensure a
  bijective correspondence between parabolic geometries of a given
  type, and underlying geometric structures of the corresponding type. Normality
  is a natural generalization of Cartan's normalization condition for
  conformal Cartan connections, and regularity is a condition that
  ensures suitable compatibility between a Cartan connection and the
  natural filtration structure of the underlying geometry.}, the
underlying $(2, 3, 5)$-distribution is just the associated bundle
\begin{equation}\label{equation:D-from-Cartan-connection}
    D := \mcG^{\G_2^*} \times_Q (\mfg^{-1} / \mfq) \subset \mcG^{\G_2^*} \times_Q (\mfg_2^* / \mfq) \cong TM_0 \textrm{,}
\end{equation}
and the derived $3$-plane distribution can be recovered as
\[
    [D, D] = \mcG^{\G_2^*} \times_Q (\mfg^{-2} / \mfq) \subset TM_0 \textrm{.}
\]
One can generalize these identifications to the general (that is, not necessarily orientable) case by instead starting with a parabolic geometry of type $(\G_2 \times \bbZ_2, Q \times \bbZ_2)$.

Though Nurowski's construction was not originally formulated in
parabolic language, from the parabolic viewpoint it simply exploits
the isomorphism $\mfg_2^* / \mfq = \mfg_2^* / (\mfg_2^* \cap \mfp_0)
\cong \mfso(3, 4) / \mfp_0$ (as $\mfq$-representations) at the level
of associated bundles:

\begin{proof}[(Sketch of proof of Theorem \ref{theorem:Nurowski}.)]
An oriented $(2, 3, 5)$-distribution determines a unique parabolic
geometry $(\mcG^{\G_2^*} \to M_0, \eta^{\G_2^*})$ of type $(\G_2^*,
Q$) \cite{Cartan} \cite[\S4.3.2]{CapSlovak}. Then, the bundle $\mcG_0
:= \mcG^{\G_2^*} \times_Q P_0 \to M_0$, together with the form
$\eta_0$ defined by extending $\eta^{\G_2^*}$ $P_0$-equivariantly to
all of $\mcG_0$, comprise a parabolic geometry of type $(\SO(3, 4),
P_0)$ (for which $\eta_0$ turns out to be normal, see
\cite[Proposition 4]{HammerlSagerschnig}), and hence an oriented
conformal structure on $M_0$.

The construction is local, and reversing the orientation of the
underlying distribution fixes the conformal structure (and simply
reverses its orientation), so for a non-oriented distribution one can
apply the construction to orientable sets that together cover $M_0$,
disregard the orientation, and patch together the conformal
structures.
\end{proof}

Such (oriented) conformal structures are precisely characterized by the holonomy reduction we are investigating: The following restates \cite[Theorem A]{HammerlSagerschnig} in the language of tractor geometry and parallel split cross products, but see also \cite[Theorem 9]{NurowskiDifferential}.

\begin{theorem}\cite[Theorem A]{HammerlSagerschnig} \label{HSthm}
An oriented conformal structure $(M_0, \mbc)$ is induced by some $(2, 3, 5)$-distribution, that is, $\mbc = \mbc_D$ for some distribution $D$ on $M_0$, if and only if the normal conformal tractor connection $\nabla_0$ admits a holonomy reduction to $\G_2^*$, that is, if and only if $\mcT_0$ admits a $\nabla_0$-parallel split cross product $\btimes$.
\end{theorem}

The algebraic properties of the (split) cross product described in Subsection \ref{subsection:octonions-cross-product-g2} let us efficiently characterize in native tractor language the $(2, 3, 5)$-distribution determined on a signature-$(2, 3)$ conformal manifold $(M_0, \mbc)$ by a parallel tractor split cross product $\btimes$. Per \eqref{aJdef}, define
\[
    \bbJ_0: \mcT_0 \to \mcT_0[1] \qquad \textrm{by} \qquad \bbJ_0 (V) := -X \btimes V \textrm{.}
\]
In particular, if $M_0$ is the zero locus determined by a split
parallel tractor cross product $\btimes$ on a projective $6$-manifold
$(M, \mbp)$, which determines a map $\bbJ: \mcT \to \mcT(1)$ via \eqref{Jdef}, then $\bbJ_0 = \bbJ
\vert_{M_0}$. Since $X \in \Gamma(\mcT_0[1])$ is null, applying Proposition \ref{proposition:J-x-null} yields a filtration
\begin{equation}\label{equation:J0-filtration}
    \mcE_0[0] \cong \ab{X} \subset (\ker \bbJ_0) [1] \subset \im \bbJ_0 \subset \ker X \cong TM_0 \flplus \mcE_0[0] \textrm{,}
\end{equation}
where in the last filtrand $X$ is regarded as a map $\mcT_0[1] \to \mcE_0[2]$. In
particular, the $TM_0$ component of any tractor (of conformal weight $1$)
in $\im \bbJ_0$ is invariant.

Using that $X$ is null, the definition of $\bbJ_0$ and Proposition \ref{proposition:J-squared}\eqref{item:J-squared-x-null} give a conformal analogue of Proposition \ref{proposition:J-properties}:

\begin{proposition}
Let $(M_0, \mbc)$ be a $5$-dimensional conformal manifold equipped with a (necessarily split) parallel tractor cross product, or equivalently, a generic parallel tractor $3$-form. Then, $\bbJ_0$ satisfies
\begin{equation}\label{equation:J0-X}
    (\bbJ_0)^A_{\phantom{A} B} X^B = 0
\end{equation}
and
\begin{equation}\label{equation:J0-squared}
    (\bbJ_0)^A_{\phantom{A} C} (\bbJ_0)^C_{\phantom{C} B} = X^A X_B \textrm{.}
\end{equation}
\end{proposition}

With these identities we can produce conformal analogues of some of the objects constructed before Theorem \ref{local}. The transformation rule \eqref{Ztrans}, together with \eqref{equation:J0-X}, implies that the weighted endomorphism component
\begin{equation}
    (J_0)^a_{\phantom{a} b} := (\bbJ_0)^A_{\phantom{A} B} Z_A^{\phantom{A} a} Z^B_{\phantom{B} b} : TM_0 \to TM_0 [1]
\end{equation}
of $\bbJ_0$ is invariant, that is, independent of the choice of
representative metric. Again if $M_0$ is the zero locus determined by
a parallel tractor cross product on a projective $6$-manifold $(M,
\mbp)$, then by construction $J_0 = J \vert_{TM_0}$. By
construction it is the
map induced by $\bbJ_0$ on the subquotient $\ker X / \ab{X} \cong TM_0$
of $\mcT_0 [1]$, or equivalently, the restriction of the map $J \vert_{M_0}: TM \vert_{M_0} \to TM \vert_{M_0} [1]$ to $TM_0$ (the image of the restriction is contained in $TM_0[1]$ by construction). Together \eqref{equation:J0-squared} and the conformal identity $\delta^A_{\phantom{A} B} = X^A Y_B + \delta^a_{\phantom{a} b} Z^A_{\phantom{A} a} Z_B^{\phantom{B} b} + Y^A X_B$ (or just \eqref{equation:J-squared-base} with the above identity $J_0 := J \vert_{TM_0}$) imply that $J_0^2 = 0$.

Let $\varpi$ denote the projection $\ker X \to TM_0$, which is just contraction with $Z_A^{\phantom{A}a}$, or equivalently, reduction modulo $\ab{X}$. Applying $\varpi$ to the filtration \eqref{equation:J0-filtration} yields a natural filtration of $TM_0$.

\begin{lemma}\label{lemma:J0-filtration}
~
\begin{enumerate}
    \item\label{item:varpi-ker-J0} $\varpi(\ker \bbJ_0 [1]) = ( \im J_0) [-1]$
    \item\label{item:varpi-im-J0}  $\varpi( \im \bbJ_0    ) =  \ker J_0      $
\end{enumerate}
In particular, the filtration \eqref{equation:J0-filtration} determines a filtration
\[
    (\im J_0)[-1] \subset \ker J_0 \subset TM_0
\]
of $TM_0$.
\end{lemma}

\begin{proof}
~
\begin{enumerate}
    \item By Proposition \ref{proposition:J-x-null}(d)\eqref{item:J-X-perp-image}, $\ker \bbJ_0 [1] = \bbJ_0(\ker X) = \im \bbJ_0 \vert_{\ker X}$, and so by the characterization before the proposition,
        \[
            \varpi(\ker \bbJ_0 [1]) = \varpi(\im \bbJ_0 \vert_{\ker X}) = \im J_0 [-1] \textrm{,}
        \]
        which proves the claim. Since $\dim \ker \bbJ_0 = 3$ (by Proposition \ref{proposition:J-x-null}\eqref{item:dim-ker-J}) and $\ker \varpi = \ab{X} \subset \ker \bbJ_0$, $\rank J_0 = 2$.
    \item By \eqref{equation:J0-squared}, $\bbJ_0(\im \bbJ_0) = \ab{X}$, so
        \[
            \set{0} = \varpi(\bbJ_0(\im \bbJ_0)) = \varpi(\bbJ_0 \vert_{\ker X}(\im \bbJ_0)) = J_0(\varpi(\im \bbJ_0))\textrm{,}
        \]
        and thus $\varpi(\im \bbJ_0) \subseteq \ker J_0$. Proposition \ref{proposition:J-x-null}\eqref{item:dim-ker-J} and the fact that $\rank J_0 = 2$ together give that both sides of the containment have rank $3$, and hence equality holds.
\end{enumerate}
\end{proof}

As one expects, this filtration coincides with the filtration $D \subset [D, D] \subset TM_0$ determined by $\Phi_0$ mentioned earlier in this subsection.

\begin{theorem}\label{theorem:J-characterization-D}
Let $(M_0, \mbc)$ be an oriented, signature-$(2, 3)$ conformal
structure and $\btimes$ a parallel split cross product on the standard
conformal tractor bundle $\mcT_0$, and let $D$ be the underlying $(2,
3, 5)$-distribution described by Theorem \ref{HSthm}. Then,
\begin{enumerate}
    \item $D = (\im J_0) [-1]$, and
    \item $[D, D] = \ker J_0$.
\end{enumerate}
\end{theorem}

\begin{proof}
Let $\bbV$ denote the irreducible $7$-dimensional representation of
  $\mfg_2^*$, fix a null nonzero vector $x \in \bbV[1]$, and take $\mfq$ to be the Lie subalgebra of $\mfg_2^*$ that preserves the line $\ab{x}$ (the weight $1$ here is chosen so that, when we pass to the associated bundle picture below, we can identify $x$ with $X \in \Gamma(\mcT[1])$).
Then, let $x
  \hook \pdot: \mfg_2^* \to \bbV[1]$ denote the map $x \hook \phi := \phi(x)$,
  where here we view $\phi \in \mfg_2^*$ as an element of $\End(\bbV[1])
  \cong \End(\bbV)$. The weight is chosen so that $x \hook \pdot$
  intertwines the $\mfq$-actions on $\mfg_2^*$ and $\bbV[1]$.
Checking the (representation-theoretic) weights of $\bbV$ as a $\mfg_2^*$-representation shows that (1) the filtration $\bbV \supset \mfg^1
  . \bbV \supset \mfg^1 . (\mfg^1 . \bbV) \supset \cdots$ of $\bbV$
  induced by $\mfq$ is exactly the one identified in Proposition
  \ref{proposition:J-x-null}, and in particular that $\mfg^1 . (\ker
  (\bbJ_0)_x) = \ab{x}$, and (2) $x \hook \pdot$ satisfies
\begin{enumerate}
    \item[(i)]   $x \hook \mfg^{0\phantom{-}} = x \hook \mfq     = \ab{x}                      $,
    \item[(ii)]  $x \hook \mfg^{-1}                              =  \ker (\bbJ_0)_x         [1]$,
    \item[(iii)] $x \hook \mfg^{-2}                              = (\ker (\bbJ_0)_x)^{\perp}[1]$, and
    \item[(iv)]  $x \hook \mfg^{-3}           = x \hook \mfg_2^* = \ab{x}           ^{\perp}   $.
\end{enumerate}

We prove (a) explicitly; the argument for (b) is entirely analogous. Since $x \hook \pdot$ is equivariant with respect to $\mfq = \mfg^0$, it is $\mfg^1$-equivariant, and hence by (ii) its restriction to $\mfg^{-1}$ induces a map
\[
    \mfg^{-1} / \mfq \cong \mfg^{-1} / (\mfg^1 . \mfg^{-1}) \to (\ker (\bbJ_0)_x [1]) / (\mfg^1 . (\ker (\bbJ_0)_x [1])) \cong \ker (\bbJ_0)_x [1]/ \ab{x} \textrm{,}
\]
where $.$ denotes the adjoint action; since the domain and codomain both have dimension $2$, and because $x \hook \pdot$ maps $\mfg^{-1}$ onto $\ker (\bbJ_0)_x [1]$, this map is an isomorphism. Passing to associated bundles identifies $x$ with $X \in \mcT_0 [1]$, and using the characterization of the $(2, 3, 5)$-distribution $D$ earlier in the subsection gives that
\[
    D = \mcG^{\G_2^*} \times_Q (\mfg^{-1} / \mfq) = \ker (\bbJ_0)_x [1]/ \ab{X} = \varpi(\ker (\bbJ_0)_x [1]) \textrm{,}
\]
but by Lemma \ref{lemma:J0-filtration}\eqref{item:varpi-im-J0} this is exactly $(\im J_0) [-1]$.
\end{proof}

\begin{corollary}\label{corollary-D-orthogonal-bracket}
A $(2, 3, 5)$-distribution $D$ satisfies $D^{\perp} = [D, D]$ with respect to the conformal structure $\mbc_D$ it induces; in particular, $D$ is totally null.
\end{corollary}
\begin{proof}
The condition is local, so by restriction we may assume $D$ is
oriented, which is a hypothesis of Theorem
\ref{theorem:J-characterization-D}. The identity holds because the
preimages of $D$ and $[D, D]$ under $\varpi$ are $H_0$-orthogonal and
$\dim D^{\perp} = 3 = \dim [D, D]$: Pick $U \in D$, $V \in [D, D]$. By
Theorem \ref{theorem:J-characterization-D}, there are weighted
tractors $S \in \ker \bbJ_0[1]$ such that $U^a = Z_A^{\phantom{A}a}
S^A$ and $T \in (\ker \bbJ_0)^{\perp}[1]$ such that $V^b =
Z_B^{\phantom{B}b} T^B$; in particular, $H_0(S, T) = 0$, and so we
have
\begin{equation}\label{equation:tractor-3-form-splitting}
0
    = (H_0)_{AB} S^A T^B
    = (X_A Y_B + \mbg_{ab} Z_A^{\phantom{A}a} Z_B^{\phantom{B}b} + Y_A X_B) S^A T^B \textrm{.}
\end{equation}
Proposition \ref{proposition:J-x-null} gives that $S, T \in \ab{X}^{\perp}[1]$, and so distributing leaves just
\[
    0 = \mbg_{ab} (Z_A^{\phantom{A}a} S^A) (Z_B^{\phantom{B}b} T^B) = \mbg_{ab} U^a V^b. \qedhere
\]
\end{proof}

So, on a projective $6$-manifold $(M, \mbp)$ (possibly with a suitable
boundary, per Theorem \ref{cify}) with a split parallel tractor cross
product whose zero locus in nonempty, this characterization of the induced
distribution realizes simply and concretely the $(2, 3,
5)$-distribution on the zero locus $M_0$ in terms of the strictly N(P)K
structure on $M_\pm$: Regarded as a suitably
projectively weighted endomorphism, the almost $\varepsilon$-complex structure
extends to, and degenerates along, the zero locus in a controlled way,
and the distribution is precisely the image of $J_0 = J \vert_{TM_0}$,
regarded (via tensoring with $\mcE[-1]$) as a subset of $TM_0$. The theorem establishes which subspaces of $\bbI^*$ corresponds to which tangent distributions, so one can extract further identifications of these just by proving the corresponding algebraic statements: For example, consulting Proposition \ref{proposition:J-x-null}, passing to the quotient, and using the above identifications gives the following realization.
\begin{proposition}
Let $(M, \mbp)$ be a projective $6$-manifold with a parallel split tractor cross product for which the zero locus $M_0$ is nonempty, and let $J$ be the corresponding weighted endomorphism field. Then, the underlying $(2, 3, 5)$-distribution $D$ on $M_0$ described by Theorem \ref{HSthm} satisfies
\begin{enumerate}
    \item $    D  =  \im J \vert_{TM_0}$ and
    \item $[D, D] = \ker J \vert_{TM_0}$.
\end{enumerate}
\end{proposition}

Using the above results, we can just as well realize the $2$-plane distribution in terms of the projecting part $\omega$ of the parallel $3$-form $\Phi$ corresponding to a parallel split tractor cross product, and hence in terms of the K\"{a}hler forms of the N(P)K structures on $M_{\pm}$. A conformal characterization of $D$ closely related to the following is given in \cite[\S4.5]{HammerlSagerschnigTwistor}.
\begin{proposition}\label{proposition:D-characterization-omega}
Let $(M, \mbp)$ be a projective $6$-manifold with a parallel split tractor cross product for which the zero locus $M_0$ is nonempty, let $\Phi$ be the corresponding parallel projective tractor $3$-form, and denote its projecting part by $\omega$. Then, the underlying $(2, 3, 5)$-distribution $D$ on $M_0$ described by Theorem \ref{HSthm} satisfies
\begin{enumerate}
    \item \label{item:D-characterization-omega}  $    D  = \ker \omega\vert_{M_0} := (\ker \omega_0)^{\perp}$ and
    \item \label{item:DD-characterization-omega} $[D, D] = (\ker\omega\vert_{M_0})^{\perp} = \ker \omega_0$,
\end{enumerate}
where $\iota$ is the natural inclusion $M_0 \hookrightarrow M$ and $\omega_0 := \iota^* \omega$. (By the kernel of a $k$-form we mean its annihilator under contraction.) In particular, $\omega_0$ is nonzero and locally decomposable, and the $2$-plane distribution spanned by the decomposable (weighted) bivector field $(\iota^* \omega)^{ab}$ is exactly $D$.
\end{proposition}

\begin{remark}
By construction, $\iota^* \omega$ coincides with the projecting part
$$X^A Z^B_{\phantom{B} b} Z^C_{\phantom{C} c} (\Phi_0)_{ABC} \in
\Gamma(\Lambda^2 T^*M_0 [3])$$
of the parallel conformal tractor
$3$-form $\Phi_0 := \Phi \vert_{M_0}$.
\end{remark}

\subsection{The model}
\label{modelS}
Given a dimension-($n+1$) real vector space $\mathbb{V}$ equipped with a
volume form, its ray projectivization
$S^n=\mathbb{P}_+(\mathbb{V})$ provides the standard
projectively flat model for $n$-dimensional projective differential
geometry. Topologically a sphere, this is a homogeneous manifold for
$\SL(\mathbb{V})\cong \SL(n+1,\mathbb{R})$ and so may be identified with
$\SL(\mathbb{V})/P$ where $P$ is the parabolic subgroup of
$\SL(n+1,\mathbb{R})$ that stabilizes a nominated ray in $\mathbb{V}$.

 In view of the canonical fibration $\pi:\mathbb{V} \setminus \set{0}\to
 S^n$, we may regard $\mathbb{V}_*:=\mathbb{V}\setminus \set{0}$ as
 a cone manifold over $S^n$. The tangent bundle
 $T\mathbb{V}_*$ is trivial and has a
 global parallelization by a canonical flat torsion free affine
 connection $\tilde{\nabla}$, namely the parallel transport arising
 from the affine structure of the vector space
 $\mathbb{V}$. Alternatively we may view $\mathbb{V}_*$ as the total
 space of an $\mathbb{R}_+$-principal bundle over $S^n$ with the
 $\mathbb{R}_+$ action given by simply scaling vectors in
 $\mathbb{V}_*$: $\lambda . v := \lambda v$, $\lambda \in \bbR_+$. It follows that
 $\mathbb{V}_*\to S^n$ has canonically a vertical vector field
 $\tilde{X}$ that infinitesimally generates this action on the fibres;
 this is usually called the {\em Euler vector field}.

We can define an equivalence relation on vectors in $T\mathbb{V}_*$ by declaring
$v_x\sim u_y$ if and only if $v_x$ and $ u_y$ are parallel, with also $x$ and $y$
being points of the same fibre. It is straightforward to check that
$\cT:=T \mathbb{V}_* / \!\sim$ is a rank-$(n+1)$ vector bundle on $S^n$,
with a connection $\nabla $ induced from $\tilde{\nabla}$. In fact it
is easily verified that $\cT$ is the associated vector bundle $
\SL(\mathbb{V}) \times_P \mathbb{V}$, and that $\nabla$ is the flat
connection arising from the fact that $\cT$ is canonically trivialized
by the mapping $\SL(\mathbb{V}) \times_P \mathbb{V}\to
\SL(\mathbb{V})/P\times \mathbb{V}$ given by $(g,v)\mapsto (gP, g .
v)$, where $g . v$ indicates the standard action of $g\in
\SL(\mathbb{V})$ on $v\in {\mathbb{V}}$. Thus $(\cT,\nabla)$ is seen to
be, by definition, the standard projective tractor bundle and connection.

It follows easily now that the pullback, via $\pi$, of each parallel
tractor field on $S^n$ is a parallel tensor field on $\mathbb{V}_*$,
and all parallel tensors on $\mathbb{V}_*$ may identified with such a
pullback. More generally each tractor field $T$ of a given weight $\ulw$
corresponds to a tensor field $\tilde{T}$ on $\mathbb{V}_*$ that is
{\em homogeneous of weight} $\ulw$, meaning that
$\tilde{\nabla}_{\tilde{X}}\tilde{T}=\underline{w} \tilde{T}$. In particular the
canonical tractor $X$ of weight $1$, defined in \nn{euler}, corresponds to
$\tilde{X}$ on $\mathbb{V}_*$.

In Theorem \ref{orbit}, and more generally in this section, we have considered a projective $6$-manifold equipped
with a parallel cross product, equivalently a parallel tractor
$3$-form $\Phi$. From our discussion here it follows at once that the
model for this is to take $\mathbb{V}$, as above but of dimension 7
and equipped with a cross product as defined in Section
\ref{subsection:octonions-cross-product-g2}; we denote this structure by
$\mathbb{P}_+(\mathbb{I^{(*)}})$.

Now the general results from above apply to this setting, with some
seen to hold by more direct reasoning using the above. Theorem
\ref{Hdef} is in the latter category, as the inner product on
$H$ on
$\mathbb{I}^{(*)}$ determines a constant metric on $\mathbb{V}$ and
hence a parallel projective tractor metric $H$. Thus we have Corollary
\ref{firststrat} giving a decomposition of the sphere
$S^n=\mathbb{P}_+(\mathbb{I}^{(*)})$ with Einstein metrics on the open
orbits. Recall that these open orbits are where the projective
density $\tau= H(X,X)$ is positive (and also, respectively, negative
in the split case). This density is equivalent to the homogeneous
degree function $\tilde{\tau}=H(\tilde{X},\tilde{X})$ on
$\mathbb{V}_*$ and it is an elementary exercise to show that, where
$\tau$ is non-vanishing, working in the scale $\tau$ as on $S^n$ (as
in e.g. the proof of Theorem \ref{local}) corresponds to working on the
level sets where $\tilde\tau =1$ and $\tilde{\tau}=-1$
\cite{CGHjlms}. In particular on the model
$\mathbb{P}_+(\mathbb{I^{(*)}})$ the Einstein structures on $M$ in the
definite case, and on $M_\pm$ in the split case, are just the induced
metrics on the applicable level sets $\tilde{\tau} = \pm 1$.  Further details for these facts here which follow from
the inner product $H$ on $\mathbb{V}$ and the
corresponding projective holonomy reduction can be found in
\cite{CGHjlms,CapGoMac}.

All of the theorems and results earlier in this Section
now specialize to the model.  So we recover the result that
in the definite case $\mathbb{P}_+(\mathbb{I})$ has a positive
definite strictly nearly K\"ahler structure. In the split case we see
the parts $\mathbb{P}_+(\mathbb{I}^{*})_\pm$ given by Corollary
\ref{firststrat}---topologically, $\mathbb{P}_+(\bbI^*)_+ = S^{2, 4}$ and $\mathbb{P}_+(\bbI^*)_- = S^{3, 3}$---respectively have a signature $(2,4)$ nearly K\"ahler structure
and a signature $(3,3)$ nearly para-K\"ahler
structure. Most importantly, this realizes the conformal structure and
$(2,3,5)$-distribution $\Delta$ on $\mathbb{P}_+(\mathbb{I^*})_0 = S^2 \times S^3$
as a simultaneous (projective) limit of these structures, as a special case of the treatment in
Section \ref{inf}; $(\mathbb{P}_+(\mathbb{I^*})_0, \Delta)$ is the flat model of the geometry of $(2, 3, 5)$-distributions.

Now the decomposition of $\mathbb{P}_+(\mathbb{I}^{(*)})$ given by
Corollary \ref{firststrat} is exactly the orbit decomposition
$\mathbb{P}_+(\mathbb{V})$ under the action of $\SO(H)$
\cite{CGHjlms,CGH}. An important question at this stage is
whether this agrees with the orbit decomposition of
$\mathbb{P}_+(\mathbb{V})$ under the action of $\smash{\G_2^{(*)}}$. In fact it does.
\begin{proposition}\label{trans}
$\G_2$ acts transitively on $\mathbb{P}_+(\mathbb{I})$, while $\G^{*}_2$
  acts transitively on $\mathbb{P}_+(\mathbb{I}^{*})_+$, on
  $\mathbb{P}_+(\mathbb{I}^{*})_-$, and on $\mathbb{P}_+(\mathbb{I}^{*})_0$.
\end{proposition}
\begin{proof}
This uses the classification in \cite[\S 3.4, p.42]{Wolf}. It is also
possible to verify it using Bryant's argument in \cite{Bmsri} for the
$\G_2$ case, and it is not hard to extend this to treat the $\G_2^*$
variant. See also \cite{Sagerschnig} for a discussion of $\G^*_2$
acting on $S^2\times S^3$.
\end{proof}
\noindent It follows that Theorem \ref{orbit} holds for the models
$\mathbb{P}_+(\mathbb{I}^{(*)})$.

\subsection{Treatment of Theorem \ref{orbit}}
\label{Ptypes} We are now ready to  bring together essential parts of the
developments above to prove this summary
 theorem from the
introduction.
\begin{proof}[Proof of Theorem \ref{orbit}]
~
\begin{itemize}
    \item $\Phi$ is definite-generic: The result is contained in Corollary
  \ref{firststrat} and Theorem \ref{local}. However, Corollary
  \ref{firststrat} only gives the stratification into curved orbits
  determined by the holonomy reduction to $\SO(H)$, that is, to the parallel tractor metric. Thus to see that the
  stated results are sufficient we must check that a parallel definite
  generic tractor 3-form has only 1 curved orbit type.  The projective
  6-sphere $S^6=\mathbb{P}_+(\mathbb{R}^7)$ is a homogeneous space for
  $\SL(7,{\mathbb{R}})$. But the subgroup $\G_2\subset
  \SL(7,{\mathbb{R}})$ also acts transitively on $S^6$ and so there is
  only one orbit type in the model structure. It now follows from
  \cite[Theorem 2.6]{CGH} that there is only a single curved orbit
  type for the case of Cartan holonomy reduction to $\G_2$.

    \item $\Phi$ is split-generic: In this case all is contained in Corollary
  \ref{firststrat}, Theorem \ref{local}, Theorem \ref{confRed}, and
  Theorem \ref{HSthm}, except that, once again, it remains to check that
  $\Phi$ does not induce a finer curved orbit decomposition than that
  given by $H$ in Corollary \ref{firststrat}.
 Again considering
  $S^6=\mathbb{P}_+(\mathbb{R}^7)$, but now under the action of
  $\G_2^*\subset \SL(7,{\mathbb{R}})$ we see that we obtain only 3 orbit
  types and these are indeed given by the strictly sign of $H_{AB}X^A X^B$.
  So $M=M_{-}\cup M_0 \cup M_+$ is the curved
  $\G_2^*$-orbit decomposition.
\end{itemize}
\end{proof}

\subsection{Compactification of N(P)K-geometries}\label{cpctsec}
One may ask whether there are effective ways to treat
compactifications of complete nearly K\"{a}hler, or nearly
para-K\"{a}hler, manifolds. In view of the models for these structures
(see Section \ref{modelS}), and also Theorem \ref{NPKtoPhi}, it is
natural to approach this via projective geometry.

Taking this viewpoint, it is then interesting to investigate the
possibilities for the geometry of the set of boundary
points.  Theorem \ref{cify} is concerned with this question in the
nearly K\"{a}hler case, and there we find that (under the assumptions
there) one is essentially forced back into the setting of Theorem
\ref{orbit}. Here we prove this after first extending the theorem to
include the nearly para-K\"{a}hler case.

\begin{theorem} \label{pcify} Let $(\overline{M},{\bf p})$ be a
6-manifold with  boundary $\partial M\neq \emptyset $ and interior $M$.
Suppose further that $M$ is equipped with a geodesically complete
nearly K\"{a}hler structure $(g, J)$ (or a nearly para-K\"{a}hler
structure $(g, J)$ that satisfies $\ab{\nabla J, \nabla J} \neq 0$
everywhere) such that the projective class $[\nabla^g]$ of the
Levi-Civita connection $\nabla^g$ coincides with ${\bf p}|_M$. Then:
$g$ has signature $(2,4)$ (resp.\ signature $(3,3)$), the metric $g$
is projectively compact of order 2, and the boundary has a canonical
conformal structure equivalent to an oriented Cartan
$(2,3,5)$-distribution.
\end{theorem}

\begin{proof}
From Theorem \ref{NPKtoPhi},
$M$ has canonically a parallel generic tractor $3$-form $\Phi\in
\Lambda^3\cT^*$.  This determines a parallel tractor metric $H$ on $M$
via the formula \nn{HHdef}.  On $M$ the density field $\tau:=H_{AB}X^A
X^B$ is nowhere zero and $H$, $\tau$ and the Einstein metric are
related as in the expression \nn{Hform}.

Now ${\bf p}|_M$ is the restriction of a smooth projective structure
${\bf p}$ on $\overline{M}$. Thus the projective tractor connection on
$M$ is the restriction of the smooth tractor connection on
$\overline{M}$. Working locally it is straightforward to use parallel
transport along a congruence of curves to give a smooth extension of
$\Phi$ to a sufficiently small open neighborhood of any point in
$\partial M$, and since $M$ is dense in $\overline{M}$ the extension
is parallel and unique.   It follows that $\Phi$ extends
as a parallel field to all of $\overline{M}$.

Thus the tractor metric $H$ extends parallelly to
$\overline{M}$. It follows that $\tau:=H_{AB}X^A X^B$ also extends smoothly to
all of $\overline{M}$. Now it must be that $\tau(x)=0$ for all points $x\in
\partial M$.  Otherwise if $\tau(x)$ were nonzero at such a point
$x$ then it would be nonzero in an open connected neighborhood of
such a point $x$ and we could easily conclude (via Theorem
\ref{HtoEin}) that the metric $g$ and its Levi-Civita connection
extend to this neighborhood that includes points of $\partial M$
and also points of $M$.  But this contradicts the assumption that
$(M,g)$ is complete.

We have that $\partial M=\mathcal{Z}(\tau)$, where $\mathcal{Z}(\tau)$
denotes the zero locus of $\tau$.  It follows that $g$ cannot have
Riemannian signature, as if it were to have Riemannian signature then
$H_{AB}$ would be positive definite (see e.g.\ \nn{Hform}) while $X$
is nowhere zero. Thus we may assume that
the metric has signature $(2,4)$ (or $(3,3)$ in the nearly
para-K\"ahler case), without loss of generality. Using the
nondegeneracy of $H$ it is straightforward to show that, for any
connection $\overline{\nabla}\in {\bf p}$, $\overline{\nabla} \tau
(x)\neq 0$ at all points $x$ in the zero locus $\mathcal{Z}(\tau)$ of
$\tau$; see the proof of Theorem 12 in \cite{CapGo-projC}.
By \cite[Theorem
  12]{CapGo-projC} it follows that $g$ is projectively compact of
order 2, that the conformal tractor bundle on $\partial M$ may be
identified with the restriction to $\partial M$ of the projective
tractor bundle, and that the conformal tractor connection is the
pullback of the projective tractor connection. Thus $\Phi|_{\partial
  M}$ is a parallel tractor for the conformal structure on $\partial
M$, and in particular its holonomy is contained in $\G_2^*$. So, the
final conclusion follows from Theorem \ref{HSthm}.
\end{proof}

\begin{remark}
Concerning the possible compactifications of geodesically
complete nearly-K\"{a}hler (or nearly para-K\"ahler) manifolds, much
stronger results are available. For example using the results of
\cite{DiScalaManno} one can easily show the boundary points at infinity
cannot lie in a submanifold of codimension at least 2.
\end{remark}

\section{The geometric Dirichlet problem}\label{gDp}

Theorem \ref{orbit} shows that a parallel projective tractor cross
product $\btimes$ determines a stratification $M = M_- \cup M_0 \cup M_+$
into curved orbits, and each of these canonically inherits an
exceptional geometric structure: The open curved orbits $M_{\pm}$ respectively inherit strictly nearly $(\mp 1)$-K\"{a}hler structures, and the hypersurface $M_0$ that separates them inherits an oriented $(2, 3, 5)$-distribution.

This raises the natural questions of
\begin{itemize}
    \item which oriented $(2, 3, 5)$-distributions $(\Sigma, D)$ arise
      this way, that is, for which $D$ can one produce a projective
      structure $\mbp$ on a collar $M$ of $\Sigma$ and a parallel
      projective split tractor cross product $\btimes$ on $(M, \mbp)$
      for which $(\Sigma, D)$ is the induced geometry on the zero
      locus it determines, and
    \item for any $D$ that admits such a structure $(M, \mbp, \btimes)$, to what degree is it unique.
\end{itemize}

It turns out that one can produce such a collar essentially uniquely for any $D$, at least formally and hence also in the real-analytic category (cf. \cite[Theorem 1.1]{GrahamWillse}, which gives an analogous result in the language of Fefferman-Graham ambient metrics):

\begin{theorem}\label{theorem:Dirichlet-problem-G2}
Let $(\Sigma, D)$ be an oriented, real-analytic $(2, 3,
5)$-distribution on a connected manifold. Then, there is a projective
manifold $(M, \mbp)$ and a parallel projective split tractor cross product
$\btimes$ for which $\Sigma$ is the zero locus $M_0$ in
the stratification of $M$ that $\btimes$ determines, and $D$ is the
$(2, 3, 5)$-distribution induced there. Moreover, $(\Sigma, D)$
determines the triple $(M, \mbp, \btimes)$ uniquely up to an overall
nonzero constant scale of $\btimes$, and up to pullback by
diffeomorphisms fixing $\Sigma$ pointwise. Thus $(\Sigma, D)$
determines the induced N(P)K structures $(M_{\pm}, g_{\pm}, J_{\pm})$
uniquely up to homothety of $g_{\pm}$ and up to pullback by
diffeomorphism.
\end{theorem}

This immediately implies natural bijective correspondences between the moduli spaces of all of the involved structures, at least in the real-analytic setting; this in particular enables holographic investigation of general $(2, 3, 5)$-geometry. To formulate the bijections appropriately, we give suitable notions of equivalence for the involved structures:

\begin{definition}\label{definition:equivalence}
Suppose $(N_a, g_a, J_a)$, $a = 1, 2$, are $6$-dimensional N(P)K Klein-Einstein manifolds with respective projective infinities $\partial N_a$. We say that $(N_1, g_1, J_1)$ and $(N_2, g_2, J_2)$ are {\em equivalent near infinity} if and only if for $a = 1, 2$ there are open sets $A_a \subseteq N_a \cup \partial N_a$ such that $A_a \supset \partial N_a$ and a diffeomorphism $\phi: A_1 \to A_2$ such that
\begin{enumerate}
    \item $(\phi\vert_{A_1 \cap N_1})^* (g_2 \vert_{A_2 \cap N_2}) = g_1 \vert_{A_1 \cap N_1}$, and
    \item $T \phi \vert_{A_1 \cap N_1} \circ J_1 \vert_{A_1 \cap N_1} = J_2 \vert_{A_2 \cap N_2} \circ T \phi \vert_{A_1 \cap N_1}$

\end{enumerate}

Similarly, if $(M_b, {\bf p}_b)$, $b = 1, 2$ are $6$-dimensional projective structures (with respective standard tractor bundles $\mcT_b$) endowed respectively with parallel split tractor cross products $\btimes_b$ for which the hypersurface curved orbits $(M_b)_0$ are both nonempty, we say that $(M_1, {\bf p}_1, \btimes_1)$ and $(M_2, {\bf p}_2, \btimes_2)$ are {\em equivalent along the zero locus} if and only if for $b = 1, 2$ there are open sets $B_b \subseteq M_b$ such that $B_b \supset (M_b)_0$ and a diffeomorphism $\psi: B_1 \to B_2$ such that
\begin{enumerate}
    \item $\psi^* ({\bf p}_2 \vert_{B_2}) = {\bf  p}_1 \vert_{B_1}$, and
    \item $\Psi \cdot (U \btimes_1 V) = (\Psi \cdot U) \btimes_2 (\Psi \cdot V)$ for all $x \in M$ and all vectors $U, V$ in the fiber $(\mcT_1)_x$, where $\Psi: \mcT_1 \to \mcT_2$ is the bundle isomorphism induced by $\psi$.
\end{enumerate}
\end{definition}

\begin{corollary}\label{corollary:bijective-correspondences}
There are bijective correspondences:
\begin{align*}
    &
        \left\{
            \begin{array}{c}
               \textrm{real-analytic projective parallel tractor cross products $(M, {\bf p}, \btimes)$} \\
               \textrm{with $M_0\neq \emptyset$ modulo equivalence along the curved hypersurface orbit $M_0$}
            \end{array}
        \right\} \\
    & \leftrightarrow
        \left\{
            \begin{array}{c}
               \textrm{real-analytic, strictly NK Klein-Einstein structures $(M_+, J_+, g_+)$} \\
               \textrm{modulo equivalence near infinity}
            \end{array}
        \right\} \\
    & \leftrightarrow
        \left\{
            \begin{array}{c}
               \textrm{real-analytic, strictly NPK Klein-Einstein structures $(M_-, J_-, g_-)$} \\
               \textrm{modulo equivalence near infinity}
            \end{array}
        \right\} \textrm{.}
\end{align*}

Furthermore, any structure of a type in the above correspondence
determines a unique oriented, real-analytic $(2, 3,
5)$-distribution. Conversely, any such (connected) distribution
determines a real-analytic projective parallel tractor cross product
$(M, \mbp, \btimes)$ modulo equivalence along the zero locus and a
positive constant rescaling of $\btimes$, and hence real-analytic,
strictly N(P)K Klein-Einstein structure $(M_{\pm}, J_{\pm}, g_{\pm})$
modulo equivalence near infinity and homothety.
\end{corollary}

With a view toward proving Theorem \ref{theorem:Dirichlet-problem-G2},
we first recall from Subsection \ref{Phi-geom} that a reduction of
holonomy to $\G_2^*$ determines a unique reduction of holonomy of
$\nabla^{\mcT}$ to $\SO(3, 4)$; this reduction is realized explicitly
by \eqref{HHdef} and \eqref{PhitoVol}, which respectively give the
nondegenerate symmetric bilinear form $H$ and the compatible tractor volume
form in
terms of a generic tractor cross product $\btimes$; recall that per the comment after \eqref{PhitoVol}, we assume that the volume form coincides with the canonical tractor volume $\vol$.
Furthermore, recall from Theorem \ref{HtoEin} that the holonomy
reduction afforded by $H$ and $\vol$ determines a decomposition of the
underlying manifold into three orbits and canonical geometric
structures on each: Two open orbits with (oriented) Klein-Einstein
metrics, one of signature $(2, 4)$ and the other signature $(3, 3)$,
and a hypersurface curved orbit with an (oriented) conformal structure
of signature $(2, 3)$.  So, any solution to the Dirichlet problem
corresponding to a $\G_2^*$ holonomy reduction must also be a solution
to Dirichlet problem corresponding to the weaker holonomy reduction to $\SO(3, 4)$, which suggests that to understand the former it would be helpful to investigate the latter.
More explicitly, given an oriented conformal structure $(\Sigma, \mbc)$ of
signature $(2, 3)$, we want to understand the existence and uniqueness of a triple $(M, \mbp, H)$ comprising a projective structure $\mbp$ on a $6$-dimensional collar $M \supset \Sigma$ and a parallel projective tractor metric $H$ (of signature $(3, 4)$) for which the zero locus and the induced conformal structure are exactly $(M, \mbc)$.

This latter Dirichlet problem, however, is a special case of the problem addressed by
the Fefferman-Graham ambient metric construction
\cite{FeffermanGraham}, although it is typically formulated in pseudo-Riemannian terms rather than
the projective tractor framework used here. We thus proceed as follows:
First, we describe the Fefferman-Graham ambient construction and the
existence and uniqueness result we need in the original language of
that construction, in which the output is a pseudo-Riemannian manifold
$\smash{(\wt{\Sigma}, H)}$. Next, we introduce the projective Thomas cone,
which lets us identify (1) part of this output with the standard projective
tractor bundle, (2) the Levi-Civita
connection $\nabla^H$ of $H$ with the data of the normal projective
tractor connection, and hence (3) $H$ with a parallel fiber metric on
the tractor bundle, which in particular gives the desired holonomy reduction to $\SO(\cdot)$; so, by construction, this identification solves the
conformal Dirichlet problem in projective language. Finally, we use
the equivalence of these formulations to translate the Dirichlet
problem corresponding to reduction of holonomy to $\G_2^*$ into
ambient language and solve it in that setting.

\subsection{The Fefferman-Graham ambient
construction}\label{subsection:ambient-metric}

Given a conformal structure $(\Sigma, \mbc)$ of signature $(r, s)$, $r
+ s \geq 2$, the Fefferman-Graham ambient metric construction aims to produce a
metric canonically determined, to the extent possible, by $\mbc$.
Consider the
metric bundle $\pi: \ol{\Sigma} \to \Sigma$ whose fiber over $x \in
\Sigma$ comprises all of the inner products on $T_x \Sigma$
in
$\mbc_x$, that is,
\[
    \ol{\Sigma}_x := \set{g_x : g \in \mbc} \textrm{.}
\]
It admits a tautological, degenerate bilinear form $h_0 \in \Gamma(S^2 T^* \ol{\Sigma})$, namely,
\[
    (h_0)_{g_x}(U, V) := g_x(T\pi \cdot U, T\pi \cdot V) \textrm{,}
\]
which is, by construction, homogeneous of degree $2$ with respect to
the natural dilations $\delta_s: \ol{\Sigma} \to \ol{\Sigma}$, $s > 0$,
$\delta_s(g_x) := s^2 g_x$, which together in turn realize
$\ol{\Sigma}$ as an $\bbR_+$-principal bundle. Though $h_0$ is
degenerate (it annihilates $\ker T\pi$), it is natural to look for metrics on a collar
$\smash{\wt{\Sigma}}$ of $\smash{\ol{\Sigma}}$ which pull back to $h_0$ and then
attempt to formulate suitable admissibility criteria for such metrics
that guarantee uniqueness and existence. Identify $\smash{\ol{\Sigma}}$ with
$\smash{\ol{\Sigma} \times \set{0} \subset \ol{\Sigma} \times \bbR}$, and
denote the inclusion by $\iota$. We call a metric $H$ on a
dilation-invariant open neighborhood $\smash{\wt{\Sigma}}$ of $\smash{\ol{\Sigma}}$ in
$\smash{\ol{\Sigma} \times \bbR}$ a \textit{pre-ambient metric} for $(\Sigma,
\mbc)$ if $\iota^* H = h_0$ and if it is homogeneous of degree $2$
under the dilations $(z, \rho) \mapsto (\delta_s(z), \rho)$, which we also denote $\delta_s$;
necessarily $H$ has signature $(r + 1, s + 1)$.
If $\dim \Sigma$ has odd dimension $n := r + s \geq 3$, we say that a
pre-ambient metric for $(\Sigma, \mbc)$ is an \textit{ambient metric}
for $(\Sigma, \mbc)$ if it (a) is Ricci-flat, and (b) satisfies the
identity $\nabla^H X = \id_{T\wt{\Sigma}}$, where $X := \partial_s
\vert_1 \delta_s$ is the infinitesimal generator of the group of
dilations $\delta_s$.\footnote{Condition (b) is called
  \textit{straightness}; it is convenient to include it here in the
  definition of an ambient metric, though this is not done in
  \cite{FeffermanGraham}.} For concreteness of exposition, we state
this result just for real-analytic, odd-dimensional conformal
structures, the case we need.\footnote{This formulation avoids two separate issues: (1) The ambient metric is a formal (power series) construction, so in the odd-dimensional case a pre-ambient metric is sometimes elsewhere (including in \cite{FeffermanGraham}) called \textit{ambient} if its Ricci curvature vanishes to infinite order in $\rho$ along $\mcG$, instead of requiring that it vanish identically on $\smash{\wt{\Sigma}}$. (2) The even-dimensional case is more subtle, but see Remark \ref{remark:even-case}.}

\begin{theorem}\label{theorem:ambient-existence-uniqueness}\cite[Theorem 2.3]{FeffermanGraham}
Suppose $\mbc$ is a real-analytic conformal structure of signature $(r, s)$ and odd dimension $n = r + s \geq 3$. Then, there is a real-analytic ambient metric $\smash{(\wt{\Sigma}, H)}$ (necessarily of signature $(r + 1, s + 1)$), and it is unique up to pullback by a diffeomorphism fixing $\ol{\Sigma}$ pointwise.
\end{theorem}

\subsection{The Thomas cone}\label{subsection:Thomas-cone}

We now describe the Thomas cone construction (\cite{CGHjlms},
\cite[\S5.2.6]{CapSlovak}, \cite{Thomas}), which lets us translate
the guarantee of existence and uniqueness into the projective language
in which our (conformal) Dirichlet problem is formulated. Given an
$n$-dimensional projective structure $(M, \mbp)$, $n \geq 3$, with
associated normal Cartan geometry, say, $(\mcG \to M, \eta)$, of type
$(\mfsl(n + 1, \bbR), P)$, let $\bbV$ denote the standard
representation of $\SL(n + 1, \bbR)$ so that $\mcT := \mcG \times_P
\bbV$ is the standard tractor bundle and $\nabla^{\bbV}$ the normal
connection it induces there. Pick a nonzero vector $e_0 \in \bbV$ in
the ray stabilized by $P$, and define $P_0$ to be the (closed)
stabilizer of $e_0$ in $P$. Since $\eta$ is $P$-equivariant, it is
$P_0$-equivariant, and we may define the \textit{Thomas cone} $\smash{\wh{M}}$
to be the (fiberwise) quotient $\mcG / P_0$ (called such because
$\smash{\wh{\pi}: \wh{M} \to M}$ is a principal bundle with fiber $P / P_0
\cong \bbR_+$) and regard $\smash{(\mcG \to \wh{M}, \eta)}$ as a Cartan
geometry of type $(\mfsl(n + 1, \bbR), P_0)$.

Now, since $\mfsl(n + 1,
\bbR) / \mfp_0$ is isomorphic (as a $P_0$-module) to $\bbV$ itself,
the Cartan connection $\eta$ canonically induces a vector bundle
connection $\smash{\wh{\nabla}}$ on $\smash{T\wh{M} = \mcG \times_{P_0} \bbV}$, and
the normalization condition (or just as well the formula \nn{tconn}
for the tractor connection) ensures that this connection is Ricci-flat
and torsion-free. Again because $P$-equivariance implies
$P_0$-equivariance, sections of any associated bundle $\mcG \times_P
\bbW \to M$ correspond to $P$-equivariant sections of $\smash{\mcG
\times_{P_0} \bbW \to \wh{M}}$. So, if $\bbW$ is a restriction of an
$\SL(n + 1, \bbR)$-representation, and hence of a subrepresentation of
$\bbV^k \otimes (\bbV^l)^*$ for some $k$ and $l$, sections of the
corresponding tractor bundle $\mcG \times_P \bbW$ correspond to tensor
fields on $\smash{\wh{M}}$ that (checking shows, using that $\smash{\wh{\nabla}}$ is torsion-free) are parallel along the fibers
of $\smash{\wh{\pi}}$. By construction, the sections of $\mcE(\underline{w})$, $\underline{w} \in
\bbR$, correspond under this identification to functions on $\smash{\wh{M}}$
homogeneous of degree $\underline{w}$ with respect to the $\bbR_+$-action. A section of $\mcT$ corresponds to a section of
$\smash{T\wh{M}}$ homogeneous of degree $-1$ with respect to the
$\bbR_+$-action, and the canonical section $X \in \Gamma(\mcT(1))$
corresponds itself to the infinitesimal generator of that action.
Taking $\bbW$ to be the
restriction of the $\SL(n + 1, \bbR)$-representation $S^2 \bbV^*$, we
see that a metric $H \in \Gamma(S^2 \mcT^*)$ corresponds to a metric
on $\smash{\wh{M}}$ homogeneous of degree $2$.

We formalize the common features of the projective Thomas cone and the
ambient metric construction to show that, in the presence of a holonomy
reduction of the normal projective tractor connection to $\SO(r + 1, s
+ 1)$ that determines a stratification whose zero locus is nonempty,
the constructions essentially coincide.

\begin{theorem}\label{theorem:ambient-metric-Thomas-cone}
    ~
    \begin{enumerate}
        \item\label{item:Thomas-to-ambient} Suppose $(M, \mbp)$ is a
          projective structure of dimension $r + s + 1 \geq 4$
whose
          standard tractor bundle admits a parallel fiber metric $H$
          of signature $(r + 1, s + 1)$ (equivalently, a reduction of
          holonomy to $\SO(H)$) with a nonempty zero locus
          $M_0$, and let $\mbc$ denote the conformal structure on $M_0$. If we
          replace $M$ with any open collar of $M_0$ in $M$ and regard $H$ as a
          homogeneous parallel metric on $\smash{\wh{M}}$,
 then $\smash{(\wh{M}, H)}$
          is an ambient metric for $(M_0, \mbc)$, and this
           identifies the data of the normal projective tractor
          connection with the Levi-Civita connection $\nabla^H$ of
          $H$.
        \item\label{item:ambient-to-Thomas} Conversely, suppose
          $\smash{(\wt{\Sigma}, H)}$ is an ambient metric for a conformal
          structure $(\Sigma, \mbc)$ of signature $(r, s)$, $r + s
          \geq 3$. Then, there is a canonically determined projective
          structure $\mbp$ on the pointwise quotient $\smash{M := \wt{\Sigma}
          / \bbR_+}$ (where $\bbR_+$ is the dilation orbit) so that
          $\smash{\wt{\Sigma}}$ is the Thomas cone $\smash{\wh{M}}$, the normal
          connection $\smash{\wh{\nabla}}$ that $(M, \mbp)$ determines there
          coincides with the Levi-Civita connection $\nabla^H$ of $H$,
          and so the metric $H$ regarded as a (parallel) fiber metric
          on the projective tractor bundle determines a holonomy
          reduction to $\SO(p + 1, q + 1)$.
    \end{enumerate}
\end{theorem}

\begin{proof}[Proof of part \eqref{item:Thomas-to-ambient}]
By hypothesis, take $M$ to be an open collar of $M_0$. Then the
 connection $\smash{\wh{\nabla}}$ is torsion-free and
regarded as a connection on $\smash{\wh{M}}$ it preserves $H$; hence, it must be the Levi-Civita connection of $H$. It has
homogeneity $2$ with respect to the dilations, by Theorem \ref{HtoEin} it pulls back to the tautological form $h_0$ on
the metric cone $\smash{\ol{\Sigma}}$ of $(M, \mbc)$, and it was observed
before the statement of the theorem that it is Ricci-flat, so it is an
ambient metric for $\mbc$.
\end{proof}

To prove \eqref{item:ambient-to-Thomas}, we will construct a candidate
$\mcT$ for the tractor bundle, a linear connection on it, and an
 adapted frame bundle $\mcG \to M$. Using a
characterization of \v{C}ap and Gover, we will show that the linear
connection is induced by a Cartan connection $\eta$ on $\mcG$ such
that $(\mcG, \eta)$ is a parabolic geometry of type $(\mfsl(n + 1,
\bbR), P)$, and then that $\eta$ is normal. In particular, $(\mcG,
\eta)$ will determine a projective structure $\mbp$ on $M$ for which
the normal tractor connection is the given linear connection on
$\mcT$.

Take $M$ to be the quotient $\smash{\wt{\Sigma} / \bbR_+}$ of $\smash{\wt{\Sigma}}$ by the dilation action and denote the projection onto that space by $\smash{\wh{\pi}: \wt{\Sigma} \to M}$; in particular, the below
argument will show that $\smash{\wh{\pi}}$ coincides with the map so named
above. By construction, sections of the weighted bundle $\mcE(\ulw) \to M$ (see Subsection \ref{pdg}) can be identified with smooth functions on $\smash{\wt{\Sigma}}$ homogeneous of weight $\ulw$ with respect to the dilation action. The discussion before the theorem motivates that we take for the tractor bundle the bundle $\mcT \to M$ with fibers
\[
    \mcT_x := \set{U \in \Gamma(T\wt{\Sigma} \vert_{\wt{\Sigma}_x}) : \textrm{$U$ is homogeneous of degree $-1$}} \textrm{,}
\]
where the smooth structure is characterized by the fact that a section
of $\mcT$ is smooth if and only if the corresponding homogeneous vector field in
$\smash{\Gamma(T\wt{\Sigma})}$ is smooth. Then the connection $\nabla^H$
descends to a connection $\nabla$ on $\mcT$, and the volume form of
$H$ descends to $\mcT$ to a $\nabla$-parallel tractor volume form
$\vol \in \Gamma(\Lambda^{n + 1} \mcT^*)$.

Let $\mcG$ denote the principal $P$-bundle
comprising the frames $(U^A)$ that (a) are adapted to the composition structure \eqref{euler} in that $U^0 = \sigma X$ for some section $\sigma$ of $\mcE(-1)$ and (b) satisfy $\vol(U^0, \ldots, U^n) = 1$. By
construction, $\mcT = \mcG \times_P \bbV$, and as usual we may identify the sections of
$\mcT$ with the $P$-equivariant maps $\mcG \to \bbV$. Any element $u
\in \mcG$ determines an isomorphism $\ul{u}: \bbV \to \mcT_{\pi(u)}$
defined by $v \mapsto [u, v]$, and the $P$-equivariant function
$\smash{\wt{t}: \mcG \to \bbV}$ corresponding to a tractor $t \in
\Gamma(\mcT)$ is just $u \mapsto \ul{u}^{-1}(t(\pi(u)))$.

To formulate \v{C}ap and Gover's condition for a linear connection
$\nabla$ on $\mcT$ to be induced by a Cartan connection on $\mcG$, we
need the following construction: Pick $u \in \mcG$ and $\xi \in T_u
\mcG$, and denote $x := \pi(u)$. Then, any $t \in \Gamma(\mcT)$
determines an element $(\nabla_{T\pi \cdot \xi} t)(x) \in \mcT_x$, and
thus the image of that point under $\ul{u}^{-1}$ in $\bbV$. Checking
shows that the difference
\[
    \ul{u}^{-1}((\nabla_{T\pi \cdot \xi} t)(x)) - (\xi \cdot t)(u)
\]
depends only on $t(x)$, so $\xi$ defines a linear map $\Psi(\xi): \bbV \to \bbV$ characterized by
\begin{equation}\label{equation:characterization-Psi}
    \ul{u}^{-1}((\nabla_{T\pi \cdot \xi} t)(x)) - (\xi \cdot t)(u) = \Psi(\xi)(\wt{t}(u))
\end{equation}
for (all) smooth sections $t$.

By \cite[Theorem 2.7]{CapGoTAMS}, a linear connection $\nabla$ on a tractor bundle $\mcT$ for a general parabolic geometry is induced by a Cartan connection $\eta$ on $\mcG$ if and only if (where in our case, $\mfg = \mfsl(n + 1, \bbR)$)
\begin{enumerate}
    \item[(A)] for each $\xi \in T_u \mcG$ the linear map $\Psi(\xi) : \bbV \to \bbV$ is given by the action of some element in $\mfg$, and
    \item[(B)] for each $x \in M$ and nonzero $U \in T_x M$, there is some index $a$ and a local smooth section $t \in \Gamma(\mcT^a)$ for which $(\nabla_U t)(x) \not\in \mcT^a_x$. (Here, $(\mcT^a)$ is the usual natural filtration of $\mcT$ induced by the $\mfp_+$-action on $\bbV$.)
\end{enumerate}

\begin{proof}[Proof of part \eqref{item:ambient-to-Thomas} of Theorem \ref{theorem:ambient-metric-Thomas-cone}]

We first check that conditions (A) and (B) hold for $\nabla^H$ on
$\mcT$:
\begin{enumerate}
    \item[(A)] Since $\nabla^H$ preserves $\vol$, it preserves the bundle of oriented, unit volume frames,
and so for any $\xi$, $\Psi(\xi)$ in \eqref{equation:characterization-Psi} is given by the action of an element of $\mfsl(n + 1, \bbR)$.
    \item[(B)] In the projective case there are only two distinct
      nonzero filtrands of $\mcT$, namely, $\mcT$ itself and the
      bundle $\mcT^1$ with fiber $\mcT^1_x := \set{\mu X : \mu \in
        \mcE(-1)_x} \cong \mcE(-1)_x$,
and so we must have $a = 1$. Fix
      $\smash{y \in \wt{\Sigma}}$. For any $\smash{\wt{U} \in T_y \wt{\Sigma}}$ and any nowhere zero local section $\mu X \in \Gamma(\mcT^1)$, we have
        \[
            (\nabla^H_{\wt{U}} (\mu X))(y) = (\wt{U} \cdot \mu)(y) X + \mu(y) (\nabla^H_{\wt{U}} X)(y) \textrm{.}
        \]
        The first term on the right is in $\mcT^1$. Regarded as an object on the ambient space, $\nabla^H X = \id_{T\wt{\Sigma}}$, and so as an object on $\mcT$, $\nabla_b X^A = Z^A_{\phantom{A} b}$. Now, $\mu(x) \neq 0$, the image of $Z$ is complementary to $\mcT^1$, and $Z^A_{\phantom{A} b}:TM(-1) \to \cT$ is injective, so the second term is not in $\mcT^1$, and thus neither is $(\nabla^H_{\wt{U}} (\mu X))(y)$.
\end{enumerate}

By the result given immediately before the proof of this part,
$\nabla^H$ corresponds to a Cartan connection on $\mcG$. An algebraic
normality condition guarantees uniqueness of the tractor connection
\cite{CapGoTAMS} (i.e. we have the {\em normal} tractor connection in
the sense of that source), and it is satisfied because $\nabla^H$ is
Ricci-flat \cite{CGHjlms}.
\end{proof}

\begin{remark}
One can easily describe the projective structure $\nabla^H$ determines on $M$: Pick any section $\sigma \in \Gamma(\wh{\pi}: \wh{M} \to M)$, and define the connection $\nabla^{\sigma}$ on $M$ by
\[
    \nabla^{\sigma}_U V := T\wh{\pi} \cdot \nabla^H_{T\sigma \cdot U} (T\sigma \cdot V) \textrm{.}
\]
Checking directly shows that this indeed defines a connection, and its projective class $p := [\nabla^{\sigma}]$ is independent on the choice of section $\sigma$.
\end{remark}

Now, Theorem \ref{theorem:ambient-metric-Thomas-cone} immediately
yields a translation of Theorem
\ref{theorem:ambient-existence-uniqueness} into natively projective
language and hence furnishes a solution to the Dirichlet problem
corresponding to a reduction of the normal projective tractor
connection to $\SO(r + 1, s + 1)$ for $r + s \geq 3$ odd.

\begin{theorem}\label{theorem:Dirichlet-problem-SO}
Let $(\Sigma, \mbc)$ be a connected, real-analytic conformal structure of
signature $(r, s)$ of odd dimension $n := r + s \geq 3$. Then, there
is (a) a real-analytic $(n + 1)$-dimensional projective structure
$\mbp$ on a collar $M$ of $\Sigma$ and (b) a holonomy reduction of the
normal projective tractor connection to $\SO( H)$ such that
geometry induced on the zero locus is $(\Sigma, \mbc)$. The solution
$(M, \mbp, H)$ is unique up to a equivalence along the zero locus and positive constant scaling of $H$.

Hence, there are Klein-Einstein metrics $(M_{\pm}, g_{\pm})$ of
signature $(r + 1, s)$ and $(r, s + 1)$ with (common) projective
infinity $(\Sigma, \mbc)$, and these are unique modulo homothety and equivalence
near infinity.
\end{theorem}

The notions of equivalence in the theorem are the same as those for holonomy
reductions to $\G_2^*$ described in Definition \ref{definition:equivalence}, replacing the cross products
$\btimes_a$ with tractor metrics $H_a$ and eliminating the criterion on
the N(P)K endomorphism fields $J_{\pm}^a$.)

\begin{proof}
By Theorem \ref{theorem:ambient-existence-uniqueness}, there is an
ambient metric $\smash{(\wt{\Sigma}, H)}$ for $(\Sigma, \mbc)$, which by
Theorem \ref{theorem:ambient-metric-Thomas-cone} determines a
projective structure $(M, \mbp)$ and the claimed holonomy
reduction. Subsection \ref{Phi-geom} recalled that this structure
determines the conformal structure, which coincides with $(\Sigma,
\mbc)$ by construction, and Klein-Einstein metrics $(M_{\pm},
g_{\pm})$. Unwinding definitions, the uniqueness of the ambient metric
in the real-analytic, odd-dimensional case described in Theorem
\ref{theorem:ambient-existence-uniqueness} implies the uniqueness
conditions stated here.
\end{proof}

\subsection{Normal forms for ambient metrics and Klein-Einstein metrics}\label{subsection:ambient-normal-form}

When working with an ambient metric $H$ for a conformal structure
$(\Sigma, \mbc)$, it is often convenient to pick a representative
metric $g \in \mbc$. Then, with respect to such a metric $g$, $H$
admits an essentially unique
normal form  \cite[Definition 2.7 and Theorem
  2.9]{FeffermanGraham}. Translating into projective language using
the identification
in \ref{subsection:Thomas-cone} gives a normal form
for parallel projective tractor metrics with a given zero locus geometry.

\begin{proposition}\label{norm-prop}
Let $(M, \mbp)$ be a projective structure of dimension $n \geq 3$,
$\mcT$ its standard tractor bundle, and $\nabla^{\mcT}$ its normal
connection. Let $H$ be a $\nabla^{\mcT}$-parallel tractor metric such
that the hypersurface curved orbit $M_0 = \set{x \in M: H_{AB} X^A X^B
  = 0}$ is nonempty, and let $\mbc$ be the conformal structure that
$H$ induces on $M_0$. We can identify an open collar $U \subseteq M$ containing $M_0$ with
an open subset of $M_0 \times \bbR$ containing $M_0 \times \set{0}
\leftrightarrow M_0$. (As in Subsection \ref{subsection:ambient-metric}, we denote by $\rho$ the standard coordinate on $\bbR$.) For any representative $g \in \mbc$, there is a
representative $\nabla$ of $\mbp$ and a weighted bilinear form $g_{\rho} \in
\Gamma(S^2 T^*U (2))$, such that
\begin{enumerate}
    \item the pullback to $M_0$ of $g_{\rho}$ is $\mbc$, and furthermore after trivializing density bundles with the scale $\nabla$, the pullback to $M_0$ of $g_{\rho}$ is $g$,
    \item in the scale determined by $\nabla$, $H$ is given by
        \begin{equation}\label{equation:ambient-normal-form}
            H
                =
                    \left(
                        \begin{array}{cc}
                            2 \rho & d\rho    \\
                             d\rho & g_{\rho} \\
                        \end{array}
                    \right)
        \end{equation}
        and
    \item\label{item:rho-degeneracy} $g_{\rho}(\partial_{\rho}, \pdot) = 0$.
\end{enumerate}
\end{proposition}

\begin{remark}
The identification of $U$ with a subspace of $M_0 \times \bbR$ and condition \eqref{item:rho-degeneracy} together enable us to view $g_{\rho}$, as a $1$-parameter family of weighted bilinear forms on $M_0$, which correspond to unweighted forms in the scale determined by the representative $g$.
\end{remark}

Note that for any parallel tractor metric in the given normal form, $\tau = H_{AB} X^A X^B = 2 \rho$ (here, $\rho$ denotes the weighted function that corresponds to $\rho$ under the trivialization of $\mcE(2)$ with respect to $\nabla$), so the curved orbits in the decomposition $M = M_+ \cup M_0 \cup M_-$ determined by $H$ are just
\begin{align*}
    M_+ &= \set{(k, \rho) \in M : \rho > 0} \\
    M_0 &= \set{(k, \rho) \in M : \rho = 0} \\
    M_- &= \set{(k, \rho) \in M : \rho < 0} \textrm{.}
\end{align*}

In fact, on any projective structure with a $\SO(p + 1, q + 1)$
holonomy reduction and non-empty zero locus $M_0$, in a collar
neighborhood of $M_0$ we can realize the open curved orbit structures
$(M_{\pm}, g_{\pm})$ of the projective structure as pseudo-Riemannian
submanifolds of the ambient metric structure $\smash{(\wt{\Sigma}, H)}$.  This
is because weighted (density and tractor) objects on $M$ can be
identified with homogeneous objects on $\smash{\wt{\Sigma}}$: Trivializing
(appropriate subsets) with respect to the scales $\tau = \pm 1$
corresponds to identifying
\[
    M_{\pm} \leftrightarrow \set{y \in \wt{\Sigma} : H_{AB} X^A X^B = \pm 1} \textrm{,}
\]
and by construction the corresponding Klein-Einstein metrics satisfy
$g_{\pm} = \iota_{\pm}^* H$, where $\iota_{\pm}$ denotes the inclusion
$\smash{M_{\pm} \hookrightarrow \wt{\Sigma}}$. For an ambient metric in normal
form \eqref{equation:ambient-normal-form}, the Klein-Einstein metrics
assume the form
\begin{equation}\label{equation:Klein-Einstein-normal-form}
    g_{\pm} = \frac{1}{2 (\pm \rho)} g_{\rho} \mp \frac{1}{4 \rho^2} d\rho^2 \textrm{.}
\end{equation}

\subsection{The proof of Theorem \ref{theorem:Dirichlet-problem-G2}}

With the existence and uniqueness result for the ambient metric in
hand, as well as the above formulation of the relationship between
the ambient metric and the projective Thomas cone, we are all but
prepared to solve the (real-analytic) Dirichlet problem for a
$\G_2^*$ holonomy-reduced projective structure: A real-analytic
oriented generic $(2, 3, 5)$-distribution $(\Sigma, D)$ induces a
conformal structure $\mbc_D$ of signature $(2, 3)$ on $\Sigma$. So, by
Theorem \ref{theorem:ambient-existence-uniqueness} this conformal
structure determines an essentially unique real-analytic ambient
metric $\smash{(\wt{\Sigma}, H)}$ which by Theorem \ref{theorem:ambient-metric-Thomas-cone}\eqref{item:ambient-to-Thomas} can be regarded as the Thomas cone for
a projective structure on $\smash{\wt{\Sigma} / \bbR_+}$ together with a
holonomy reduction of the Thomas cone connection $\nabla$, and hence
(equivalently) of the tractor projective connection, to $\SO(3,
4)$. Via the Thomas cone identification, the discussion before Theorem
\ref{theorem:ambient-metric-Thomas-cone} shows that a solution to the
$\G_2^*$ Dirichlet problem corresponds in the ambient setting to a
holonomy reduction of $\nabla^H$ to $\G_2^*$, that is, a
$\nabla^H$-parallel split cross product $\btimes$, for which the
induced zero locus geometry is exactly $(\Sigma, D)$. The last
remaining tasks, then, are to translate the initial data in the
Dirichlet problem (namely, a $(2, 3, 5)$-distribution) into the
language of the ambient and Thomas cone settings, and to show that one
can always use this to produce such a $\btimes$. The first half is essentially the content of Theorem \ref{HSthm}. A general tractor extension result \cite[Theorem 1.4]{GrahamWillse} resolves the second half. We translate that result into projective language for even-dimensional projective structures.

\begin{theorem}\label{theorem:Graham-Willse}
Let $(M, \mbp)$ be a real-analytic projective structure of even dimension $n \geq 4$, $\mcT$ its standard tractor bundle, and $\nabla^{\mcT}$ its normal connection. Let $H$ be a (necessarily indefinite) $\nabla^{\mcT}$-parallel tractor metric such that the zero locus $M_0 = \set{x \in M: H_{AB} X^A X^B = 0}$ is nonempty, and let $\mbc$ be the conformal structure that $H$ induces there, $\mcT_0 \subset \mcT$ its standard conformal tractor bundle, and $\nabla^{\mcT_0}$ the normal conformal tractor connection.

Suppose $\chi_0 \in \Gamma(\bigotimes^r \mcT_0^*)$ is a real-analytic $\nabla^{\mcT_0}$-parallel conformal tractor tensor. Then, there is a connected open subset $U \supset M_0$ of $M$ and a real-analytic $\nabla^{\mcT}$-parallel projective tractor tensor $\chi \in \Gamma(\bigotimes^r \mcT^* \vert_{U})$ such that $\chi \vert_{M_0} = \chi_0$, and any two such extensions agree on some open set containing $M_0$.
\end{theorem}

\begin{proof}[Proof of Theorem \ref{theorem:Dirichlet-problem-G2}]
Denote by $\mbc_D$ the conformal structure that $D$ induces on
$\Sigma$. Theorem \ref{HSthm} yields the
corresponding split-generic $3$-form $\Phi_0$, on the conformal tractor
bundle $\mcT_0$ of $\mbc_D$, that is parallel with respect to the normal
conformal tractor connection; by naturality $\Phi_0$ is
real-analytic. Then, Theorem \ref{theorem:Dirichlet-problem-SO} gives
that there is a projective structure $\mbp$, say, with normal tractor
connection $\nabla^{\mcT}$, and a $\nabla^{\mcT}$-parallel tractor
metric $H$ on a collar $M$ of $\Sigma$ for which the geometry on the
zero locus is $(\Sigma, \mbc_D)$. Next, Theorem
\ref{theorem:Graham-Willse} guarantees the existence of a parallel
split-generic tractor $3$-form $\Phi$ on $M$ such that $\Phi
\vert_{\Sigma} = \Phi_0$ (after possibly replacing $M$ with a smaller
collar of $\Sigma$). Raising an index of $\Phi$ using $H$ gives a
parallel split tractor cross product $\btimes$ on $(M, \mbp)$; by
construction, the geometry $\btimes$ determines on the zero locus is $(\Sigma, D)$. Tracing the uniqueness statements in
the involved theorems yields the claimed uniqueness.
\end{proof}

\begin{remark}\label{remark:even-case}
Theorems \ref{theorem:Dirichlet-problem-SO} and \ref{theorem:Graham-Willse} can be extended to the case where $(M_0, \mbc)$ is an even-dimensional conformal structure of dimension $n \geq 4$, but in that case the involved constructions are only guaranteed to work to a finite order that depends on $n$, and hence the precise existence and uniqueness statements are rather more subtle than in the odd-dimensional case.
\end{remark}

\begin{remark}\label{remark:generality-of-structures}
The Klein-Einstein condition appears to impose severe restrictions on
strictly N(P)K structures $(M_{\pm}, g_{\pm}, J_{\pm})$, and hence on
those structures that admit bounding $(2, 3, 5)$-distributions. This
is true in at least a na\"{\i}ve sense: Applying the Cartan-K\"{a}hler
Theorem gives that, locally, $6$-dimensional strictly N(P)K structures
depend, modulo diffeomorphism, on $2$ arbitrary functions of $5$ real
variables, the same generality as for K\"{a}hler structures in this
dimension; see \cite[\S4.3 and Remark 23]{B} for the definite and indefinite strictly NK cases.
On the other hand, Corollary
\ref{corollary:bijective-correspondences} shows that near a point on
the projective infinity, a (real-analytic) Klein-Einstein strictly
N(P)K structure is mutually determined by the $(2, 3, 5)$-distribution
it defines there, and a na\"{\i}ve count shows that locally $(2, 3,
5)$-distributions (and hence such strictly N(P)K structures) are
considerably less general: Modulo diffeomorphism, they depend on just
$1$ function of $5$ variables.
\end{remark}

\section{Examples}\label{section:examples}

Section \ref{gDp} suggests a method for producing explicit examples of projective structures $(M, \mbp)$ with parallel split tractor cross products $\btimes$ for which the zero locus $M_0$ is nonempty:
\begin{enumerate}
    \item Select a (real-analytic, oriented) $(2, 3, 5)$-distribution $(\Sigma, D)$.
    \item\label{item:construct-conformal-structure} Compute the Nurowski conformal structure $\mbc_D$ that $D$ induces.
    \item\label{item:construct-conformal-cross-product} Compute the parallel split-generic conformal tractor cross product $\btimes$ (or equivalently, the parallel split-generic conformal tractor $3$-form $\Phi_0$) on $\Sigma$, which $D$ determines up to constant scale. (The cross product can be normalized using a choice of oriented conformal tractor volume form, by demanding that it coincide with the volume form determined pointwise by \eqref{equation:cross-to-volume}.)
    \item\label{item:construct-ambient-metric} Compute the (essentially unique) real-analytic ambient metric $\smash{(\wt{\Sigma}, \wt{g}_D)}$ of $\mbc_D$.
    \item Set $\smash{M := \wt{\Sigma} / \bbR_+}$; the Levi-Civita connection of $\smash{\wt{g}_D}$ descends to a connection $\nabla^{\mcT}$ on $\smash{T\wt{\Sigma} / \bbR_+ \to M}$, which we may view as the standard projective tractor bundle for the projective structure $\mbp$ that $\nabla^{\mcT}$ determines.
    \item\label{item:construct-projective-cross-product} Compute the parallel extension of $\btimes$ (or $\Phi_0$) to a parallel split-generic projective tractor cross product $\btimes$ (respectively, to a parallel split-generic projective tractor $3$-form $\Phi$).
\end{enumerate}

Theorem \ref{theorem:Dirichlet-problem-G2} ensures that this construction always yields such a triple $(M, \mbp, \btimes)$, and Corollary \ref{corollary:bijective-correspondences} shows that essentially all such triples $(M, \mbp, \btimes)$ that have nonempty zero locus (and that are real-analytic) arise this way. Computing the indicated data explicitly for a general $(2, 3, 5)$-distribution $D$, however, is generally difficult: Step \eqref{item:construct-ambient-metric} amounts to solving a typically intractable system of partial differential equations---indeed, explicit ambient metrics have only been produced for a limited number of classes of conformal structures, and for only a few families of Nurowski conformal structures. The other parts of the procedure are variously less formidable: Step \eqref{item:construct-conformal-structure} amounts to carrying out Cartan's normalization procedure for these geometries (giving a $\mfg_2^*$-valued Cartan connection) and exploiting the inclusion $\mfg_2^* \hookrightarrow \mfso(3, 4)$, or, in the case that $D$ is given in so-called Monge normal form (see below), simply computing using a formula of Nurowski. Step \eqref{item:construct-conformal-cross-product} is first a matter of computing the normal conformal Killing form associated to $D$: Proposition \ref{proposition:D-characterization-omega} shows that simply forming the wedge product of the vectors in an oriented local basis and lowering indices using the conformal structure gives this form up to a positive real-valued function (in five local coordinates), and hence solving for the normal conformal Killing form amounts locally to solving an overdetermined PDE for this function (which always admits a solution). With that solution in hand, one can recover the corresponding parallel section of the conformal tractor bundle $\Lambda^3 \mcT_0^*$ of $\mbc_D$ by applying to that solution the so-called BGG splitting operator for that bundle; this operator was recorded for this purpose in \cite[(15)]{HammerlSagerschnig}, and we reproduce it in Appendix \ref{appendix:BGG-splitting-operator}. Finally, Step \eqref{item:construct-projective-cross-product} amounts to solving a system of $35$ ordinary differential (parallel transport) equations in $\rho$.

\subsection{A Monge quasi-normal form for $(2, 3, 5)$-distributions}

Any ordinary differential equation $z' = F(x, y, y', y'', z)$, where $y$ and $z$ are regarded as functions of $x$ can be encoded on the jet space $J^{2, 0}_{xypqz} \cong \bbR^5 \supseteq \dom F$ (where $p$ and $q$ are jet coordinates respectively corresponding to $y'$ and $y''$) as the $2$-plane distribution
\begin{equation}\label{equation:Monge-normal-form}
    D_F := \ker \set{dy - p\,dx, dp - q\,dx, dz - F\,dx} \textrm{.}
\end{equation}
Checking shows that this is a $(2, 3, 5)$-distribution if and only if $\smash{\partial_q^2 F}$ vanishes nowhere, so any such smooth function $F(x, y, p, q, z)$ specifies such a distribution. In fact, \cite[\S76]{Goursat} shows that every $(2, 3, 5)$-distribution is locally equivalent at each point of the underlying manifold to $D_F$ for some function $F$.

Nurowski has a practical (albeit complicated) formula \cite[(54)]{NurowskiDifferential} that gives for general $F$ (such that $\partial_q^2 F$ vanishes nowhere) an explicit coordinate expression for a representative of the conformal structure induced by $D_F$; it is polynomial in the $4$-jet of $F$.

\subsection{A family with homogeneous curved orbits}

\begin{example}\label{example:qm-class}
We apply the construction at the beginning of the section to describe a $1$-parameter family of deformations of the flat model described in Subsection \ref{modelS} for which the geometries induced on the curved orbits are all still homogeneous.

In \cite{Cartan}, Cartan solved the equivalence problem for $(2, 3,
5)$-distributions. Moreover, he showed that if such a distribution $D$ is not locally equivalent to the flat model for that geometry, then the infinitesimal symmetry algebra\footnote{The infinitesimal symmetry algebra of a distribution $(M, D)$ is the vector space of all vector fields on $M$ whose flows preserve $D$, that is, the vector fields $U$ such that $\mcL_U V \in \Gamma(D)$ for all vector fields $V \in \Gamma(D)$.} $\mfinf(D)$ of $D$ has dimension at most $7$ and gave (in the
complex setting) an
explicit coframe description of the local distributions for which equality holds.\footnote{Strictly speaking, he established this bound under
  the modest assumption of ``constant root type'', which in particular
  holds for any homogeneous distribution. This restriction was
  recently lifted in \cite[Theorem 5.5.2]{KruglikovThe}.}
We construct parallel projective split tractor cross products $(M, \mbp,
\btimes)$ for (real versions of) these distributions using
a well-known realization amenable to our purposes: Up to local
equivalence (and again a suitable notion of complexification) the
distributions with $7$-dimensional symmetry algebra are exhausted by
the distributions defined on $\bbR^5_+ := \set{(x, y, p, q, z) : q > 0}$ via \eqref{equation:Monge-normal-form} by $D_m := D_{q^m}$ and $D_{\log q}$,
where $m \not\in
\set{-1, 0, 1, \tfrac{1}{3}, \tfrac{2}{3}, 2}$. For $m \in \set{-1,
  \frac{1}{3}, \frac{2}{3}, 2}$, $D_m$ is locally equivalent to the
flat model for that geometry; for $m \in \set{0, 1}$, $D_m$ is not a
$(2, 3, 5)$-distribution. Some distinct values of $m$ yield locally
equivalent distributions $D_m$, but varying $m$ over the
half-open interval $[\frac{1}{2}, 1)$ exhausts all such distributions
  without any such redundancy, and includes the flat model at $m =
  \frac{2}{3}$; see \cite{Kruglikov} and the minor correction thereto
  in \cite{DoubrovGovorov} for details. It turns out that the
  conformal structures of the distributions $D_m$ and $D_{\log q}$ are
  all almost Einstein (in fact, almost Ricci-flat) \cite{WillseHeisenberg}, which leaves their
  ambient metrics more amenable to explicit computation. We exploit
  this to give explicit data for the parallel projective tractor cross
  products determined by the distributions $D_m$; the distribution
  $D_{\log q}$ can be handled similarly, but we do not do so here. By Theorem
  \ref{theorem:Dirichlet-problem-G2} a real-analytic $(2, 3,
  5)$-distribution determines (at least in a collar along the
  underlying manifold) the real-analytic parallel projective tractor
  cross product, so for the locally flat distributions $D_m$, $m \in
  \set{-1, \tfrac{1}{3}, \tfrac{2}{3}, 2}$ the corresponding parallel
  projective tractor cross product $(M, \mbp, \btimes)$ is locally
  equivalent to the flat model in Subsection \ref{modelS}. In
  particular, we may view the family of resulting structures as one of smooth
  deformations of that model. The ambient metrics associated to these
  distributions were first given in coordinates in
  \cite{NurowskiExplicit}.

We describe this family of geometries in terms of a well-adapted global frame (or rather, a corresponding family of global frames) on $\bbR^5_+$. All of the claims can be verified by direct (if tedious) computation, in particular using the data \eqref{equation:example-qm-connection-form} for the representative connection of the underlying projective structure.

Fix $m \in \bbR \setminus \set{0, 1}$. The infinitesimal symmetry algebra $\mfinf(D_m)$ of $D_m$ satisfies \cite[\S 5]{Kruglikov}
\begin{equation}
    \mfinf(D_m) \supseteq \ab{\xi_1, \ldots, \xi_7} \textrm{,}
\end{equation}
where
\begin{gather*}\label{equation:example-qm-symmetries}
           \xi_1 := \partial_x,
    \qquad \xi_2 := \partial_y,
    \qquad \xi_3 := x \partial_y + \partial_p,
    \qquad \xi_4 := y \partial_y + p \partial_p + q \partial_q + m z \partial_z, \\
    \qquad \xi_5 := \partial_z,
    \qquad \xi_6 := x \partial_x + 2 y \partial_y + p \partial_p + z \partial_z, \\
    \qquad \xi_7 := q^{m - 1} \partial_x + \left(p q^{m - 1} - \tfrac{1}{m} z\right) \partial_y + \left(1 - \tfrac{1}{m}\right) q^m \partial_p
                            + (m - 1) \int q^{2 m - 2} dq \cdot \partial_z \textrm{.}
\end{gather*}
If $D_m$ is not flat, that is, if $m \not\in \set{-1, \frac{1}{3}, \frac{2}{3}, 2}$, then equality holds.
The generators $\xi_1, \ldots, \xi_5$ together span a subalgebra $\mfs_m$ of $\mfinf(D_m)$ that acts infinitesimally transitively on $\bbR^5_+$. Then, analyzing the flows of the generators $\xi_i$ enables us to identify $\bbR^5_+$ with the connected, simply connected Lie group $S_m$ with Lie algebra $\mfs_m$; it is isomorphic to the matrix group
\[
    \left\{
        \left(
            \begin{array}{ccccc}
                  a_5 & \cdot &   a_4 &   a_3 &   a_1 \\
                \cdot & a_5^m & \cdot & \cdot &   a_2 \\
                \cdot & \cdot &   a_5 & \cdot &   a_3 \\
                \cdot & \cdot & \cdot &     1 & \cdot \\
                \cdot & \cdot & \cdot & \cdot &     1 \\
            \end{array}
        \right)
        : a_1, a_2, a_3, a_4 \in \bbR; a_5 > 0 \right\}
    \textrm{.}
\]

Consider the left-invariant frame $(E_a)$ of $TS_m$ given by
\begin{align*}
    E_1 &:=  \frac{q}{15 \sqrt{10} m^4} \partial_y \\
    E_2 &:= -\frac{(3 m + 14) (m + 1)}{30 m^2} (\partial_x + p \partial_y + q \partial_p + q^m \partial_x) + \frac{q (m - 1) }{45 \sqrt{10} m^5} \partial_y \\
    E_3 &:= -\frac{1}{2 \sqrt{5} m^2} [(m + 2) \partial_x + (m + 2) p \partial_y + (m + 1) q \partial_p + 2 q^m \partial_z] \\
    E_4 &:=  \frac{m - 1}{15 m^3}(\partial_x + p \partial_y + q \partial_p + q^m \partial_z) + \sqrt{10} q \partial_q \\
    E_5 &:= -\frac{1}{5 m^2} (\partial_x + p \partial_y + q \partial_p + q^m \partial_z) \textrm{.}
\end{align*}
With respect to this basis, the (left-invariant) distribution $D_m$ is
\[
    D := \ab{E_4, E_5} \textrm{,}
\]
the derived $3$-plane distribution is
\[
     [D, D] := \ab{E_3, E_4, E_5} \textrm{,}
\]
and the conformal structure $c_D$ it induces admits the representative
\[
    2 e^1 e^4 + 2 e^2 e^5 - (e^3)^2 \textrm{.}
\]

Consider the $6$-manifold $M := S_m \times \bbR$, identify the vector fields $E_a \in \Gamma(TS_m)$, $a \in \set{1, \ldots, 5}$, respectively with $(E_a, 0) \in \Gamma(TM)$, let $\rho$ denote the standard coordinate on $\bbR$, and denote $E_6 := \partial_{\rho} \in \Gamma(TM)$. Take $\mbp$ to be the projective class on $M$ containing the connection $\nabla$ whose nonzero connection forms with respect to the global frame $(E_{\alpha}) = (E_1, \ldots, E_6)$ of $M$ (and its dual coframe $(e^{\alpha})$) are given in \eqref{equation:example-qm-connection-form} in Appendix \ref{appendix:qm-class-data}.

Consider the weighted $2$-form
\[
    \omega := \sqrt{2} e^1 \wedge e^2 + \sqrt{2} [- (m + 1) (m - 2) e^1 \wedge e^5 + e^4 \wedge e^5] \cdot \rho - e^3 \wedge d\rho \in \Gamma(\Lambda^2 T^* M [3])
\]
(here written with respect to the scale determined by $\nabla$). Computing gives that
\[
    \tfrac{1}{3} d\omega = - e^1 \wedge e^3 \wedge e^4 - e^2 \wedge e^3 \wedge e^5 + \sqrt{2} e^4 \wedge e^5 \wedge d\rho \textrm{.}
\]
Now, computing gives that the projective tractor $3$-form
\[
    \Phi_{ABC}
        :=
            \left(
                \begin{array}{c}
                                   \omega _{ bc} \\
                    \tfrac{1}{3} (d\omega)_{abc} \\
                \end{array}
            \right) \in \Gamma(\Lambda^3 \mcT^*) \textrm{,}
\]
is parallel and split-generic.

The corresponding operator $\bbJ: \mcT \to \mcT[1]$ is
\[
    \bbJ^A_{\phantom{A} B}
        =
            \left(
                \begin{array}{cc}
                    0 & \chi_b    \\
                    0 & J^a_{\phantom{a} b} \\
                \end{array}
            \right) \textrm{,}
\]
where (in the scale determined by $\nabla$)
\begin{align*}
       J &= - E_3 \otimes d\rho - \sqrt{2} E_4 \otimes e^2 + \sqrt{2} E_5 \otimes e^1 \\
            &\qquad\qquad\qquad + \rho [- \sqrt{2} E_1 \otimes e^5 + \sqrt{2} (m + 1) (m - 2) E_2 \otimes e^1 \\
            &\qquad\qquad\qquad\qquad\qquad\qquad + \sqrt{2} E_2 \otimes e^4 + \sqrt{2} (m + 1) (m - 2) E_4 \otimes e^5 + 2 \partial_{\rho} \otimes e^3] \\
    \chi &= - e_3 \textrm{.}
\end{align*}

In the scale $\nabla$, the parallel tractor metric $\Phi$ determines on $\mcT$ is (given by the normal form of Subsection \ref{subsection:ambient-normal-form})
\begin{equation}\label{equation:example-qm-H}
    H =
        \left(
            \begin{array}{cc}
                2  \rho & d\rho    \\
                  d\rho & g_{\rho} \\
            \end{array}
        \right) \textrm{,}
\end{equation}
where
\[
    g_{\rho} = 2 e^1 e^4 + 2 e^2 e^5 - (e^3)^2 - 2 (m + 1) (m - 2) (e^5)^2 \cdot \rho \textrm{.}
\]

The strictly nearly $(\mp 1)$-K\"{a}hler structure $(g_{\pm}, J_{\pm})$ that the scale $\tau = \pm 1$ determines on $M_{\pm}$ is given by (cf. \ref{equation:Klein-Einstein-normal-form})
\[
    g_{\pm} = \frac{1}{2 (\pm \rho)} [2 e^1 e^4 + 2 e^2 e^5 - (e^3)^2 - 2 (m + 1) (m - 2) (e^5)^2 \cdot \rho] \mp \frac{1}{4 \rho^2} d\rho^2
\]
and
\begin{multline}
    J_{\pm} = (\pm \rho)^{- 1 / 2} \left( \tfrac{1}{\sqrt{2}} E_3 \otimes d\rho + E_4 \otimes e^2 - E_5 \otimes e^1 \right) \\ \pm (\pm \rho)^{1 / 2} [E_1 \otimes e^5 - (m + 1) (m - 2) E_2 \otimes e^1 \\ \qquad\qquad\qquad - E_2 \otimes e^4 - (m + 1) (m - 2) E_4 \otimes e^5 - \sqrt{2} \partial_{\rho} \otimes e^3] \textrm{,}
\end{multline}
and the K\"{a}hler form is
\begin{multline*}
    \omega_{\pm}
        =
              (\pm \rho)^{- 3 / 2} \left(   \tfrac{1}{2         }                 e^1 \wedge e^2
                                          - \tfrac{1}{2 \sqrt{2}}                 e^3 \wedge d\rho \right)
          \pm (\pm \rho)^{- 1 / 2} \left[ - \tfrac{1}{2         } (m + 1) (m - 2) e^1 \wedge e^5
                                          + \tfrac{1}{2}                          e^4 \wedge e^5   \right] \textrm{.}
\end{multline*}

The infinitesimal symmetries $\xi_a$ given in
\eqref{equation:example-qm-symmetries} are (interpreted as vector
fields $(\xi_a, 0)$ on $M$) symmetries of $(M, \mbp, \Phi)$ except for
$\xi_6 := x \partial_x + 2 y \partial_y + p \partial_p + z \partial_z$;
however, $\smash{\wt{\xi}_6 := \xi_6 + 2 \rho \partial_{\rho}}$ is a symmetry of
$(M, \mbp, \Phi)$. Define $\smash{\wt{\xi}_a = (\xi_a, 0)}$ for $a \in \set{1, 2,
  3, 4, 5, 7}$.

The integral curves of $\smash{\wt{\xi}_6}$ include $t \mapsto ((0, 0, 0, 1, 0),
C e^{2t})$, $C \in \bbR$, so the infinitesimal symmetry algebra
$\mfinf(M, \mbp, \Phi)$\footnote{This infinitesimal symmetry algebra is defined as the vector space of all vector fields on $M$ whose flows preserve both $\mbp$ and $\Phi$.} acts transitively on each of the
three curved orbits; in particular, the underlying structures,
$(M_{\pm}, g_{\pm}, J_{\pm})$ and $(M_0, D_m)$, are homogeneous. In
fact, $\smash{\wt{\mfs}_m := \ab{\wt{\xi}_1, \ldots \wt{\xi}_6}}$ is a subalgebra
of $\smash{\ab{\wt{\xi}_1, \ldots, \wt{\xi}_7} \cong \ab{\xi_1, \ldots, \xi_7}}$, so
we may regard the structures $(g_{\pm}, J_{\pm})$ as left-invariant
structures on the connected, simply connected Lie group $\smash{\wt{S}_m}$
with Lie algebra $\smash{\wt{\mfs}_m}$.

\end{example}

\appendix

\section{The splitting operator for $\Lambda^3 \mcT_0^*$}\label{appendix:BGG-splitting-operator}

On a $5$-dimensional conformal structure $(M, c)$, the BGG splitting operator $L_0: \Gamma(\mcE_{[ab]}[3]) \to \Gamma(\mcE_{[ABC]})$ is given in any scale (say, that given by the representative metric $g \in c$) by \cite[(15)]{HammerlSagerschnig}
\[
    L_0 : (\omega_0)_{bc}
            \mapsto
                \left(
                    \begin{array}{c}
                        (\omega_0)_{bc} \\
                        (\omega_0)_{bc, a}
                            \vert -\tfrac{1}{4} (\omega_0)_{kc,}^{\phantom{kc,} k} \\
                        - \tfrac{1}{15} (\omega_0)_{bc, k}^{\phantom{bc, k} k}
                            - \tfrac{2}{15} (\omega_0)_{k [b, c]}^{\phantom{k [b, c]} k}
                            - \tfrac{1}{10} (\omega_0)_{k [b, \phantom{k} c]}^{\phantom{k [b,} k}
                            - \tfrac{4}{5} {\sf P}^k_{\phantom{k} [b} (\omega_0)_{c] k}
                            - \tfrac{1}{5} {\sf P}^k_{\phantom{k} k} (\omega_0)_{bc} \\
                    \end{array}
                \right) \textrm{.}
\]
Here, ${\sf P}$ is the conformal Schouten tensor of $g$, and indices are raised using $c$, regarded as section of $\Gamma(S^2 T^* M[2])$.

\section{The connections in Example \ref{example:qm-class}}
\label{appendix:qm-class-data}
In this appendix we specify the connections used in the Example \ref{example:qm-class} by giving their connection forms $\nu_a^b$:
\begin{align}
    \nu_1^2 &=   \tfrac{  \sqrt{ 5}}{3 \sqrt{2}} (3 m^2 + 3 m + 2) (m + 1)                   e^1 \notag \\
    \nu_1^4 &= - \tfrac{  \sqrt{ 5}}{3 \sqrt{2}} (3 m^2 + 3 m + 2) (m + 1)                   e^5 \notag \\
    \nu_1^6 &= -                                                                             e^4 \notag \\
    \nu_2^6 &= -                                                                             e^5 \notag \\
    \nu_3^1 &=   \tfrac{         1 }{  \sqrt{2}}                                             e^5 \notag \\
    \nu_3^2 &=   \tfrac{         1 }{2 \sqrt{2}}                   (m + 1) (7 m + 6)         e^1
               - \tfrac{         1 }{  \sqrt{2}}                                             e^4 \notag \\
    \nu_3^4 &= - \tfrac{         1 }{2 \sqrt{2}}                   (m + 1) (7 m + 6)         e^5 \notag \\
    \nu_3^6 &=                                                                               e^3 \notag \\
    \nu_4^2 &=            \sqrt{10}                                (m + 1)                   e^1
               +          \sqrt{ 2}                                                          e^3 \notag \\
    \nu_4^3 &=            \sqrt{ 2}                                                          e^5 \label{equation:example-qm-connection-form} \\
    \nu_4^4 &= -          \sqrt{10}                                (m + 1)                   e^5 \notag \\
    \nu_4^6 &= -                                                                             e^1 \notag \\
    \nu_5^1 &= -          \sqrt{10}                                                          e^1
               -          \sqrt{ 2}                                                          e^3 \notag \\
    \nu_5^2 &= - \tfrac{  \sqrt{10} (m - 1)^2}{3 m}                                          e^1
               +          \sqrt{10}                                                          e^2
               -        2 \sqrt{10}                                (m + 1)           (m - 2) e^5 \cdot \rho
               -                                                   (m + 1)           (m - 2) d\rho \notag \\
    \nu_5^3 &= -          \sqrt{ 2}                                (m + 1)           (m - 2) e^1
               -          \sqrt{ 2}                                                          e^4 \notag \\
    \nu_5^4 &= -          \sqrt{ 2}                                (m + 1)           (m - 2) e^3
               +          \sqrt{10}                                                          e^4
               + \tfrac{  \sqrt{10} (m - 1)^2}{3 m}                                          e^5 \notag \\
    \nu_5^5 &= -          \sqrt{10}                                                          e^5 \notag \\
    \nu_5^6 &= -                                                                             e^2 \notag \\
    \nu_6^2 &= -                                                   (m + 1)           (m - 2) e^5 \notag \textrm{.}
\end{align}

\end{document}